\documentclass{jms_My}
\usepackage{color}
\paperTitle{The calculation of the Mittag-Leffler function}
\articleColonName{The calculation of the Mittag-Leffler function}
\authorsShort{V.\,V.~Saenko}
\authorsFull{V.\,V. Saenko\first}
\addAuthorInfo{Ulyanovsk State University, S.P. Kapitsa Research Institute of Technology,   L. Tolstoy St. 42, Ulyanovsk,  Russia, 432017 e-mail: \url{saenkovv@gmail.com}}

\paperAbstract{The problem of calculating the Mittag-Leffler function $E_{\rho,\mu} (z)$ is considered in the paper. To solve this problem integral representations for the function $E_{\rho,\mu}(z)$ are transformed in such a way that they could not contain complex variables and parameters. Integral representations written in this form allow one to use standard methods of numerical integration to calculate integrals contained in them. To verify the correctness of the integral representations obtained the function $E_{\rho,\mu}(z)$ was calculated both with the use of obtained formulas and with the use of known representations of the Mittag-Leffler function. The calculation results demonstrate their exact matching. This fact is indicative of the correctness of new integral representations of the function $E_{\rho,\mu}(z)$  that were obtained.}

\begin{document}
\maketitle
\section{Introduction}

The two parameter Mittag-Leffler function
\begin{equation}\label{eq:MLF_gen}
  E_{\rho,\mu}(z)=\sum_{k=0}^{\infty}\frac{z^k}{\Gamma(\mu+k/\rho)},\quad \rho>0,\quad \mu\in\C,\quad z\in\C
\end{equation}
is studied in the paper where $\Gamma(x)$ is the Gamma-function. This function firstly appeared in the papers of A. Wiman \cite{Wiman1905,Wiman1905a}. Later in 1953 this function was studied in the papers of Humbert and Agarwal \cite{zbMATH03082752,zbMATH03078845,zbMATH03081895} and in 1954 in the papers of M.M. Djrbashian \cite{Dzhrbashyan1954_eng,Dzhrbashian1954b_eng} (see also~\cite{Dzhrbashyan1966_eng} Chapter 3, \S 2, 4) The function $E_{\rho,\mu}(z)$ was introduced as a result of generalization of the classical Mittag-Leffler function
$ E_\rho(z)=\sum_{n=0}^{\infty}\frac{z^n}{\Gamma(1+n/\rho)}$, $\rho>0$, $z\in\C $. The function $E_{\rho}(z)$ was introduced in a number of papers published from 1902 to 1905 \cite{Mittag-Leffler1900,Mittag-Leffler1901,Mittag-Leffler1901a,Mittag-Leffler1902,Mittag-Leffler1905,Mittag-Leffler1920}. This generalization was obtained by replacing the additive unit in the argument of the Gamma-function by an arbitrary parameter $\mu\in\C$. More detailed information about the properties of the Mittag-Leffler function the reader can find in the monograph \cite{Gorenflo2014} as well as in the review papers \cite{Haubold2011,Popov2013,Rogosin2015,Gorenflo2019}.

Special attention was paid to the Mittag-Leffler function first of all due to the possibility to use this function for solving differential and integral  equations of fractional order. In such problems the Mittag-Leffler function serves as an eigenfunction of fractional differentiation operators and integration. Solutions to this kind of equations are expressed in terms of the Mittag-Leffler function. For example, the authors of the paper \cite{Baeumer2018} showed that the solution to the fractional diffusion equation with a power-law function kernel in the fractional derivative operator was expressed in terms of the Mittag-Leffler function. The fractional diffusion equation and the fractional-advection diffusion equation with boundary condition are considered in the paper \cite{Yu2018}. A distinctive feature of the studied equations from the classical analogues is the presence of the fractional derivative operator with kernel in the form of the Mittag-Leffler function. It is shown in this paper that the solutions to these equations are expressed in terms of the Mittag-Leffler function.
The works \cite{Amsalu2018,Suthar2019a} are devoted to the study of the properties of the generalized Mittag-Leffler function and the Mittag-Leffler function of several variables with the action of a generalized fractional integral operator on them. New integral equalities for the Mittag-Leffler function of several variables were obtained in the work \cite{Suthar2019}.
The Mittag-Leffler function is also used when solving a master equation describing the fractional Poisson process \cite{UCHAIKIN2008} and when solving  the fractional relaxation equation in dielectrics \cite{Uchaikin2009a}. In these two papers the Monte-Carlo method was used to express the obtained solutions. From here we can see that the question of the Mittag-Leffler function calculation is not less important than the question of studying the properties of this function. The capability of calculating the Mittag-Leffler function gives an opportunity to obtain numerical values for solutions to fractional differential equations.

This paper is focused on two main goals: firstly, to obtain integral representations for the function  $E_{\rho,\mu}(z)$ that make it possible to calculate its value with the use of standard numerical methods  and, secondly, to verify the correctness of integral representations obtained in the works \cite{Saenko2020,Saenko2020d}. It should be noted that this work is a final part of the series of papers devoted to the study of integral representations of the Mittag-Leffler function. The fundamental paper in this respect is the paper \cite{Saenko2020}. In this paper, we obtain an integral representation of the Mittag-Leffler function expressing the value of this function through the contour integral. In the subsequent work \cite{Saenko2020a} the singular points of the obtained integral representation are investigated. Particular cases are determined in which it is possible to calculate the contour integral using the theory of residues. In these particular cases, the Mittag-Leffler function can be written in terms of elementary functions. In the third work \cite{Saenko2020d} a transition was made from integration over a complex variable to integration over real variables. As a result of such a transition, the Mittag-Leffler function can be written in two forms: representation “A” and representation “B”. The representation  “A” of the Mittag-Leffler function  (see Theorem~2 and Corollary~1 in \cite{Saenko2020d})  is written in the form of a sum of the improper and definite integrals. The representation “B” (see Theorem~3 and Corollary~3 in \cite{Saenko2020d}) in the general case, is written as one improper integral. However, the representation “B” is valid for the parameter values  $\mu$, satisfying the condition $\Re\mu<1+1/\rho$. Also in this representation there are special cases in which you have to bypass a singular point. As a result, in these particular cases, the representation of “B” is much more complicated.

To verify integral representations of the function $E_{\rho,\mu}(z)$ obtained in the papers   \cite{Saenko2020,Saenko2020d} one should make a comparison of numerical values of the functions  $E_{\rho,\mu}(z)$  calculated with the use of these integral representations  and known representations for the Mittag-Leffler function,  for example, expressing it in terms of elementary functions. To make such a comparison it is necessary to be able to calculate integrals in obtained integral representations. In the general case, these integrals cannot be calculated analytically and numerical methods must be used to calculate them.  However, the integral representations obtained in \cite{Saenko2020d} are inconvenient for the direct use of numerical integration methods in them.  The fact is that in these representations the argument $z$ and the parameter $\mu$ are complex numbers. For these representations to be suitable for the use of numerical algorithms it is necessary to transform them in such a way that the integral representations of the function $E_{\rho,\mu}(z)$ consist of the sum of the real and imaginary parts, each of which would be a function of only real variables and parameters. If such a transformation is performed, then it will be possible to use standard methods of numerical integration to calculate the Mittag-Leffler function. This paper is devoted to obtaining such integral representations for the function $E_{\rho,\mu}(z)$.

It should be pointed out that the problem of calculating the Mittag-Leffler function is not a new one.  To date, most of the results associated with the calculation of the Mittag-Leffler function are based on the integral representation obtained in the book \cite{Dzhrbashyan1966_eng}. This representation expresses the Mittag-Leffler function in terms of the contour integral in the complex plane. Using this integral representation, the authors of the work \cite{Gorenflo2002b} passed from loop integration to integration over a real variable. Such a transition allowed them to represent the Mittag-Leffler function as a sum of improper and definite integrals. Further, to calculate the Mittag-Leffler function, the authors divide the complex plane into several regions and depending on what point of the complex plane where it is necessary to calculate the function $E_{\rho,\mu}(z)$ either a series representation (\ref{eq:MLF_gen}), either the integral representation they obtained, or the recursion formula for the function $E_{\rho,\mu}(z)$ is used. In the paper  \cite{Parovik2012_eng} some particular aspects of the calculation algorithm for the function  $E_{\rho,\mu}(z)$ are considered on the basis of the work \cite{Gorenflo2002b}. In the works \cite{Seybold2005,Hilfer2006,Seybold2009} to calculate the function $E_{\rho,\mu}(z)$ the authors use an approach which is similar to the approach proposed in the work  \cite{Gorenflo2002b}. In these works, the authors also split the complex plane into several regions, and depending on the region where the value of the function $E_{\rho,\mu}(z)$ is calculated, either the representation (\ref{eq:MLF_gen}), or the integral representation, or the recursion formula is used. It is important to note that the authors of the works \cite{Seybold2005,Hilfer2006,Seybold2009} as the integral representation, use the integral representation of the function $E_{\rho,\mu}(z)$ obtained in \cite{Dzhrbashyan1966_eng}. Apparently, the calculation methods proposed in the papers \cite{Gorenflo2002b,Seybold2005,Hilfer2006,Seybold2009} underlay the computer code \cite{Podlubny2012} implemented in the MatLab language. 

However, as was mentioned in the paper \cite{Popolizio2014} the computer code \cite{Podlubny2012} are likely to give an incorrect result in some cases. The author of this paper points to the error of calculating the function $E_{\rho,\mu}(z)$ with the use of code \cite{Podlubny2012} (see~Fig.~1~in~\cite{Popolizio2014} and the discussion of it ). The authors give a simple example: they compare the values of the function $|E_{1/2,1}(-i\lambda)|$ for $\lambda\in[0,20]$ calculated with the help of the computer code \cite{Podlubny2012} and the representation
\begin{equation*}
  E_{1/2,1}(z)=e^{z^2}\mathrm{erfc}(-z).
\end{equation*}
It turned out that the results generated by the code \cite{Podlubny2012} in a sufficiently large interval of values $\lambda$ are incorrect.

A completely different approach to the calculation of the Mittag-Leffler function was proposed in \cite{Popolizio2014,Garrappa2015,Garrappa2018}. In these papers to calculate the values of the function $E_{\rho,\mu}(z)$ the authors propose to calculate the inverse Laplace transform of the Laplace image of the function $E_{\rho,\mu}(z)$. It enables them to exclude the integral representations for the function $E_{\rho,\mu}(z)$ from consideration completely and develop a totally different approach for the calculation of the Mittag-Leffler function. This algorithm was realized in the form of a computer code for MatLab and is available in \cite{Garrappa2015a}.

Thus, today there is only one algorithm of calculating the Mittag-Leffler function (which is known to the author of this paper at least) that allows one to calculate correctly the values of the function over the entire complex plane. Nevertheless, to verify the correctness of integral representations for the Mittag-Leffler function obtained in this paper, we are going to compare the calculation results from the obtained formulas with the calculation results using the known representations for the Mittag-Leffler function. Such an approach seems to be justifiable in the issue of verifying the correctness of new representations obtained. In simpler cases, for example, at $\rho=1$, the Mittag-Leffler function can be represented in terms of elementary functions.  In this paper we are going to use these representations to verify the correctness of the formulas obtained.

We need to make  some remarks. In the paper \cite{Saenko2020d} it was shown that the representations  “A” and “B” can be described either by four parameters ($\rho,\mu,\delta_{1\rho},\delta_{2\rho}$), or by three parameters ($\rho,\mu,\delta_{\rho}$), or by two parameters ($\rho,\mu$). The parameters $\rho$ and $\mu$ are related to the Mittag-Leffler function and the parameters $\delta_{1\rho},\delta_{2\rho}, \delta_\rho$ determine the contour of integration in the integral representation of the function $E_{\rho,\mu}(z)$. At the end of the paper  \cite{Saenko2020d} its own notation was used for each set of parameters.  For the four-parameter representation, the notation “Parametrization 1” was introduced, for the three-parameter representation, the notation “Parametrization 2” was introduced, and for the two-parameter representation, “Parameterization 3” was introduced. Further in this paper, we will adhere to this terminology.

Finally, it should be pointed out that in this paper the letters $\rho$, $\mu$ are used to denote the parameters of the Mittag-Leffler function (\ref{eq:MLF_gen}). These notations were introduced by M.M. Djrbashian in his works \cite{Dzhrbashyan1954_eng,Dzhrbashian1954b_eng,Dzhrbashyan1966_eng,Dzhrbashian1993} (see also \cite{Popov2013}). Parameters $\rho$ and $\mu$ of the function (\ref{eq:MLF_gen}) are connected with the commonly used notations $\alpha$ and $\beta$ of the parameters of the Mittag-Leffler function by simple relations $\rho=1/\alpha, \mu=\beta$.

\section{The calculation of the Mittag-Leffler function}

We will start our consideration with the representation “A” of Parameterization 1 formulated in  \cite{Saenko2020d} (see Theorem~2 in \cite{Saenko2020d}).    This representation is convenient for analytical studies of the Mittag-Leffler function. However, they are not suitable for numerical calculations, since these formulas include the variables $z$, $\mu$, which are complex numbers. To be able to calculate the value of the function $E_{\rho,\mu}(z)$ these representations must be transformed in such a way that they could be divided into real and imaginary parts, each of which will depend only on real arguments and parameters.   As a result, the following theorem is valid:

\begin{theorem}\label{lemm:MLF_int2}
For any real $\rho>1/2$, $\epsilon>0$  and $\delta_{1\rho}$, $\delta_{2\rho}$, satisfying the conditions
\begin{equation}\label{eq:deltaRho_cond_lemm_int2}
 \pi/(2\rho)<\delta_{1\rho}\leqslant\min\left(\pi,\pi/\rho\right),\quad \pi/(2\rho)<\delta_{2\rho}\leqslant\min\left(\pi,\pi/\rho\right),
\end{equation}
any complex $\mu=\mu_R+i\mu_I$ and any complex $z=t e^{i\theta}$, satisfying the condition
\begin{equation}\label{eq:argZ_cond_lemm_int2}
  \pi/(2\rho)-\delta_{2\rho}+\pi<\theta<-\pi/(2\rho)+\delta_{1\rho}+\pi,
\end{equation}
the Mittag-Leffler function can be represented in the form
\begin{multline}\label{eq:MLF_int2}
  E_{\rho,\mu}(z)=\int_{1+\epsilon}^{\infty}K_{\rho,\mu}^{Re}(r,-\delta_{1\rho},\delta_{2\rho}, t, \theta)dr
  + \int_{-\delta_{1\rho}-\pi}^{\delta_{2\rho}-\pi}P_{\rho,\mu}^{Re}(1+\epsilon,\varphi,t, \theta)d\varphi\\
  +i\left\{ \int_{1+\epsilon}^{\infty}K_{\rho,\mu}^{Im}(r,-\delta_{1\rho},\delta_{2\rho},t, \theta)dr
  + \int_{-\delta_{1\rho}-\pi}^{\delta_{2\rho}-\pi}P_{\rho,\mu}^{Im}(1+\epsilon,\varphi,t, \theta)d\varphi\right\}.
\end{multline}
Here
\begin{align}\label{eq:K_Re_lemm_int2}
K_{\rho,\mu}^{Re}&(r,\varphi_1,\varphi_2,t, \theta) =\frac{\rho}{2\pi}\frac{(t r)^{\rho(1-\mu_R)}} {(r^2+2r\cos\varphi_1+1)(r^2+2r\cos\varphi_2+1)}\nonumber\\
\times  & [f(r, \varphi_2-\pi,t, \theta)(r^2+2r\cos\varphi_1+1)\nonumber\\
 \times & (r\sin(\xi(r,\varphi_2-\pi,t, \theta)) +\sin(\xi(r,\varphi_2-\pi,t, \theta)+\varphi_2))\nonumber\\
-&f(r, \varphi_1-\pi,t, \theta)(r^2+2r\cos\varphi_2+1)\nonumber\\
\times&(r\sin(\xi(r,\varphi_1-\pi,t, \theta)) +\sin(\xi(r,\varphi_1-\pi,t, \theta)+\varphi_1))],
\end{align}

\begin{align}\label{eq:K_Im_lemm_int2}
K_{\rho,\mu}^{Im}&(r,\varphi_1,\varphi_2,t, \theta) =\frac{\rho}{2\pi}\frac{(t r)^{\rho(1-\mu_R)}}{(r^2+2r\cos\varphi_1+1)(r^2+2r\cos\varphi_2+1)}\nonumber\\
\times &[f(r, \varphi_1-\pi,t, \theta) (r^2+2r\cos\varphi_2+1)\nonumber\\
\times &(r\cos(\xi(r,\varphi_1-\pi,t, \theta)) +\cos(\xi(r,\varphi_1-\pi,t, \theta)+\varphi_1))\nonumber\\
-&f(r,\varphi_2-\pi,t, \theta) (r^2+2r\cos\varphi_1+1)\nonumber\\
\times & (r\cos(\xi(r,\varphi_2-\pi,t, \theta)) +\cos(\xi(r,\varphi_2-\pi,t, \theta)+\varphi_2))],
\end{align}
\begin{align}
P_{\rho,\mu}^{Re}(r,\varphi,t, \theta) &=\frac{\rho}{2\pi}\frac{(t r)^{\rho(1-\mu_R)}r f(r,\varphi,t, \theta)} {(r^2-2r\cos\varphi+1)}\nonumber\\
& \times  \left(r\cos(\xi(r,\varphi,t, \theta))- \cos(\xi(r,\varphi,t, \theta)+\varphi)\right),\label{eq:P_Re_lemm_int2}\\
P_{\rho,\mu}^{Im}(r,\varphi,t, \theta)&=\frac{\rho}{2\pi}\frac{(t r)^{\rho(1-\mu_R)}r f(r,\varphi,t, \theta)} {(r^2-2r\cos\varphi+1)}\nonumber\\
& \times  \left(r\sin(\xi(r,\varphi,t, \theta))- \sin(\xi(r,\varphi,t, \theta)+\varphi)\right),\label{eq:P_Im_lemm_int2}
\end{align}
where
\begin{align}
  f(r,\varphi,t, \theta) &= \exp\left\{(t r)^\rho\cos(\rho(\theta+\varphi))+ \rho\mu_I(\theta+\varphi)\right\},\label{eq:f_int2} \\
  \xi(r,\varphi,t, \theta) &= (t r)^\rho\sin(\rho(\theta+\varphi))+ \rho(1-\mu_R)(\theta+\varphi)-\rho\mu_I\ln(t r).\label{eq:xi_int2}
\end{align}
\end{theorem}

\begin{proof}

 1)  In the article \cite{Saenko2020d} (see Theorem~2 in \cite{Saenko2020d}) it was shown that the Mittag-Leffler function can be represented in the form
\begin{equation}\label{eq:MLF_tmp_lemm_int2}
  E_{\rho,\mu}(z)=\int_{1+\epsilon}^{\infty}K_{\rho,\mu}(r,-\delta_{1\rho},\delta_{2\rho},z)dr+
  \int_{-\delta_{1\rho}-\pi}^{\delta_{2\rho}-\pi}P_{\rho,\mu}(1+\epsilon,\varphi,z)d\varphi.
\end{equation}
Here $K_{\rho,\mu}(r,\varphi_1,\varphi_2,z)$ and $P_{\rho,\mu}(r,\varphi,z)$ are defined by the expressions  (8)~and~(11) in~\cite{Saenko2020d}. We represent the complex number $z$ in the form $z=t e^{i\theta}$, and the complex parameter $\mu$ in the form $\mu=\mu_R+i\mu_I$ and put these representations in the kernels $K_{\rho,\mu}(r,\varphi_1,\varphi_2,z)$ and $P_{\rho,\mu}(r,\varphi,z)$.

First, we obtain the expression for the kernel $K_{\rho,\mu}(r,\varphi_1,\varphi_2,z)$.  From (8) in \cite{Saenko2020d} it is clear that this expression is composed of the function $\eta(r,\varphi,z)$, defined by the expression  (10) from \cite{Saenko2020d}. As a result, this function can be represented in the form
\begin{equation*}
  \eta(r,\varphi,z)=(t r)^\rho e^{i\rho(\theta-\pi)}\sin(\rho\varphi)+\rho(1-\mu_R-i\mu_I)\varphi.
\end{equation*}

Using this representation in $A_{\rho,\mu}(r,\omega_1,\omega_2,z)$, which is defined by the equation (9) from \cite{Saenko2020d}, we obtain
\begin{multline}\label{eq:A_lemm_int2_tmp0}
  A_{\rho,\mu}(r,\omega_1,\omega_2,z)=
   \exp\left\{(rt)^\rho e^{i\rho(\theta+\omega_1-\pi)}\right\}e^{i\rho(1-\mu_R)\omega_1+\rho\mu_I\omega_1} \\
   \times(r^2+2r\cos\omega_2+1) \left(r+e^{i\omega_1}\right).
\end{multline}
For the factor $(zr)^{\rho(1-\mu)} e^{-i\rho\pi(1-\mu)}$, included in the integrand $K_{\rho,\mu}(r,\varphi_1,\varphi_2,z)$ (see (8) in \cite{Saenko2020d}) we have
\begin{equation}\label{eq:zr_lemm_int2_tmp0}
  (zr)^{\rho(1-\mu)} e^{-i\rho\pi(1-\mu)}=(t r)^{\rho(1-\mu_R-i\mu_I)}e^{i\rho(\theta-\pi)(1-\mu_R-i\mu_I)}.
\end{equation}
Using (\ref{eq:A_lemm_int2_tmp0}) and (\ref{eq:zr_lemm_int2_tmp0}) for the factor $(zr)^{\rho(1-\mu)} e^{-i\rho\pi(1-\mu)} A_{\rho,\mu}(r,\varphi_1,\varphi_2,z)$, included in $K_{\rho,\mu}(r,\varphi_1,\varphi_2,z)$ (see (8) in \cite{Saenko2020d}) and making some transformations we obtain
\begin{multline}\label{eq:K_2nd_lemm_int2}
  (zr)^{\rho(1-\mu)} e^{-i\rho\pi(1-\mu)} A_{\rho,\mu}(r,\varphi_1,\varphi_2,z)= (t r)^{\rho(1-\mu_R-i\mu_I)} e^{i\rho(\theta-\pi) (1-\mu_R-i\mu_I)}\\
   \times \exp\left\{(rt)^\rho e^{i\rho(\theta+\varphi_1-\pi)}\right\}e^{i\rho(1-\mu_R)\varphi_1+\rho\mu_I\varphi_1} (r^2+2r\cos\varphi_2+1) \left(r+e^{i\varphi_1}\right)\\
   =(t r)^{\rho(1-\mu_R)} f(r,\varphi_1-\pi,t, \theta) (r^2+2r\cos\varphi_2+1)\\
    \times \left(r\cos(\xi(r,\varphi_1-\pi,t, \theta))+\cos(\xi(r,\varphi_1-\pi,t, \theta)+\varphi_1)\right. \\
   +\left.i\left\{r\sin(\xi(r,\varphi_1-\pi,t, \theta)) +\sin(\xi(r,\varphi_1-\pi,t, \theta)+\varphi_1)\right\} \right),
\end{multline}
where the notations were introduced
\begin{align}
  f(r,\varphi,t, \theta) &= \exp\{(t r)^\rho \cos(\rho(\theta+\varphi))+\rho\mu_I(\theta+\varphi)\}, \label{eq:f_tmp_lemm_int2}\\
  \xi(r,\varphi,t, \theta) &= (t r)^\rho \sin(\rho(\theta+\varphi)) +\rho(1-\mu_R)(\theta+\varphi) -\rho\mu_I\ln(t r).\label{eq:xi_tmp_lemm_int2}
\end{align}
Similarly, we obtain
\begin{multline}\label{eq:K_1st_lemm_int2}
   (zr)^{\rho(1-\mu)} e^{-i\rho\pi(1-\mu)} A_{\rho,\mu}(r,\varphi_2,\varphi_1,z)
   =(t r)^{\rho(1-\mu_R)} f(r,\varphi_2-\pi,t, \theta) (r^2+2r\cos\varphi_1+1)\\
    \times \left(r\cos(\xi(r,\varphi_2-\pi,t, \theta))+\cos(\xi(r,\varphi_2-\pi,t, \theta)+\varphi_2)\right.\\
   +i\left.\left\{r\sin(\xi(r,\varphi_2-\pi,t, \theta)) +\sin(\xi(r,\varphi_2-\pi,t, \theta)+\varphi_1)\right\} \right).
\end{multline}

Substituting now (\ref{eq:K_2nd_lemm_int2}) and (\ref{eq:K_1st_lemm_int2}) in  definition of the integrand $K_{\rho,\mu}(r,\varphi_1,\varphi_2,z)$ (see (8) in \cite{Saenko2020d}) and making some transformations we get
\begin{multline}\label{eq:KRe+KIm_lemm_int2}
  K_{\rho,\mu}(r,\varphi_1,\varphi_2,z)=\frac{\rho}{2\pi i}\frac{(t r)^{\rho(1-\mu_R)}}{(r^2+2r\cos\varphi_1+1)(r^2+2r\cos\varphi_2+1)}\\
   \times\left[ f(r,\varphi_2-\pi,t, \theta) (r^2+2r\cos\varphi_1+1)\right.\\
  \left(r\cos(\xi(r,\varphi_2-\pi,t, \theta))+\cos(\xi(r,\varphi_2-\pi,t, \theta)+\varphi_2)\right.\\
   +\left.i\left\{r\sin(\xi(r,\varphi_2-\pi,t, \theta)) +\sin(\xi(r,\varphi_2-\pi,t, \theta)+\varphi_2)\right\} \right)\\
   - f(r,\varphi_1-\pi,t, \theta) (r^2+2r\cos\varphi_2+1)\\
    \left(r\cos(\xi(r,\varphi_1-\pi,t, \theta))+\cos(\xi(r,\varphi_1-\pi,t, \theta)+\varphi_1)\right.\\
   +i\left.\left.\left\{r\sin(\xi(r,\varphi_1-\pi,t, \theta)) +\sin(\xi(r,\varphi_1-\pi,t, \theta)+\varphi_1)\right\} \right) \right]\\
   =K_{\rho,\mu}^{Re}(r,\varphi_1,\varphi_2,t, \theta) + i K_{\rho,\mu}^{Im}(r,\varphi_1,\varphi_2,t, \theta),
\end{multline}
where $K_{\rho,\mu}^{Re}(r,\varphi_1,\varphi_2,t, \theta)$ and $K_{\rho,\mu}^{Im}(r,\varphi_1,\varphi_2,t, \theta)$ are  real and imaginary parts of the obtained expressions and have the form (\ref{eq:K_Re_lemm_int2}) and (\ref{eq:K_Im_lemm_int2}) respectively.
Now we get the expression for $P_{\rho,\mu}(r,\varphi,z)$.  For this, we represent the complex numbers $z$ and $\mu$ in the form $z=t e^{i\theta}$ and $\mu=\mu_R+i\mu_I$ and put them in the definition of this function (see (11) in \cite{Saenko2020d}). As a result, for the function $\chi(r,\varphi,z)$ which is defined by the formula (12) in \cite{Saenko2020d}, we have
\begin{equation*}
  \chi(r,\varphi,z)=(t r)^\rho e^{i\rho\theta} \sin(\rho\varphi)+\rho(1-\mu_R)\varphi-i\rho\mu_I\varphi.
\end{equation*}
Using now this expression in the kernel $P_{\rho,\mu}(r,\varphi,z)$ (see~(11) in \cite{Saenko2020d}) and making some transformations, we obtain
\begin{multline}\label{eq:PRe+PIm_lemm_int2}
  P_{\rho,\mu}(r,\varphi,z)\equiv P_{\rho,\mu}(r,\varphi,t,\theta)=\frac{\rho}{2\pi}\frac{\left(t re^{i\theta}\right)^{\rho(1-\mu_R-i\mu_I)}} {r^2-2r\cos\varphi+1} \exp\left\{(t r)^\rho e^{i\rho\theta}\cos(\rho\varphi)\right\}\\
  \times\exp\left\{i\left[(t r)^\rho e^{i\rho\theta} \sin(\rho\varphi)+\rho(1-\mu_R)\varphi-i\rho\mu_I\varphi\right]\right\} r(r-e^{i\varphi})\\
   =\frac{\rho}{2\pi}\frac{(t r)^{\rho(1-\mu_R)}rf(r,\varphi,t, \theta)}{r^2-2r\cos\varphi+1} \Bigl[r\cos(\xi(r,\varphi,t, \theta))
   -\cos(\xi(r,\varphi,t, \theta)+\varphi)\\
   +i\left\{r\sin(\xi(r,\varphi,t, \theta))-\sin(\xi(r,\varphi,t, \theta)+\varphi)\right\}\Bigr]
   = P_{\rho,\mu}^{Re}(r,\varphi,t, \theta)+i P_{\rho,\mu}^{Im}(r,\varphi,t, \theta).
\end{multline}
Here $f(r,\varphi,t, \theta)$ and $\xi(r,\varphi,t, \theta)$ have the form (\ref{eq:f_tmp_lemm_int2}) and (\ref{eq:xi_tmp_lemm_int2}), and $P_{\rho,\mu}^{Re}(r,\varphi,t, \theta)$ and $P_{\rho,\mu}^{Im}(r,\varphi,t, \theta)$ are real and imaginary parts of the obtained expression and have the form (\ref{eq:P_Re_lemm_int2}) and (\ref{eq:P_Im_lemm_int2}) respectively. Now substituting (\ref{eq:KRe+KIm_lemm_int2}) and (\ref{eq:PRe+PIm_lemm_int2}) in the representation (\ref{eq:MLF_tmp_lemm_int2}) we obtain the representation (\ref{eq:MLF_int2}).
\begin{flushright}
  $\Box$
\end{flushright}
\end{proof}

The obtained representation allows us to calculate the value of the Mittag-Leffler function. As we can see,  the real and imaginary parts of the representation (\ref{eq:MLF_int2}) consist of two integrals: improper and definite ones.

To verify the correctness of integral representations obtained it is necessary to compare the value of the function calculated with the use of these formulas with the function value calculated by means of the known representations for the Mittag-Leffler function, for example, expressed in terms of elementary functions. It is known that at parameter values  $\rho=1$ and $\mu=0,\pm1,\pm2,\pm3,\dots$ the Mittag-Leffler function has the form
\begin{align}
  E_{1,n}(z)&=e^z z^{1-n},\quad n\leqslant1,\label{eq:MLF_mu<=1}\\
  E_{1,n}(z)&=z^{1-n}\left(e^z-\sum_{k=0}^{n-2}\frac{z^k}{k!}\right),\quad n\geqslant2.\label{eq:MLF_mu>=2}
\end{align}
These formulas are known. For example, the formula (\ref{eq:MLF_mu<=1}) can be found in the book \cite{Gorenflo2014} (see formula (4.6.1) in section 4.6) and formula (\ref{eq:MLF_mu>=2}) can be found in  \cite{Podlubny1999_ch1,Mathai2008,Gorenflo2014}. Derivation of the latter formula from the representation (\ref{eq:MLF_gen}) directly can be found in the works \cite{Podlubny1999_ch1,Mathai2008}.

Alternative derivation of the formulas (\ref{eq:MLF_mu<=1}) and (\ref{eq:MLF_mu>=2}) is given in the paper \cite{Saenko2020a}. In this paper it was shown that these formulas could be derived directly both from the integral representation for the function $E_{\rho,\mu}(z)$ introduced in the paper  \cite{Saenko2020} and from representation (\ref{eq:MLF_gen}). Therefore, to verify the correctness of integral representations obtained in this work for the function $E_{\rho,\mu}(z)$ we are going to compare the function values calculated by means of formulas (\ref{eq:MLF_mu<=1}), (\ref{eq:MLF_mu>=2}) with the function values calculated with the use of integral representations.

In general case, it is impossible to calculate the integrals in the representation (\ref{eq:MLF_int2}) analytically. However, with the use of numerical methods the calculation of these integrals is quite simple. For the numerical calculation of the integrals in (\ref{eq:MLF_int2}) the Gauss-Kronrod numerical integration method was used. In particular, implementations of this method were used in the GSL library \cite{gsl_lib}. The calculation results of the function $E_{\rho,\mu}(z)$ by the formula (\ref{eq:MLF_int2}) are given in Fig.~\ref{fig:MLF_int2_rho1_mu0}~and~\ref{fig:MLF_int2_rho1_mu3}.

\begin{figure}
  \centering
  \includegraphics[width=0.42\textwidth]{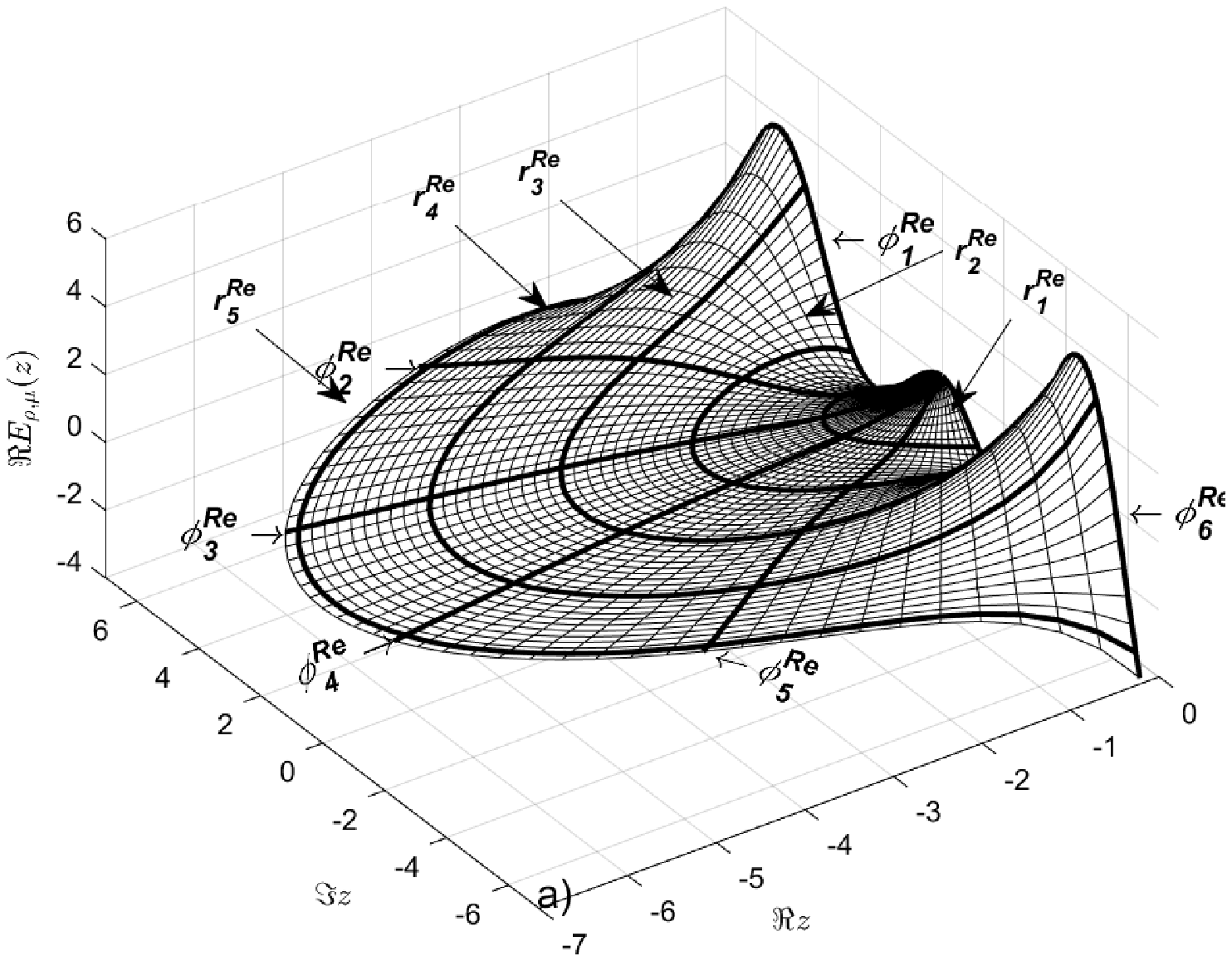}\hfill
  \includegraphics[width=0.42\textwidth]{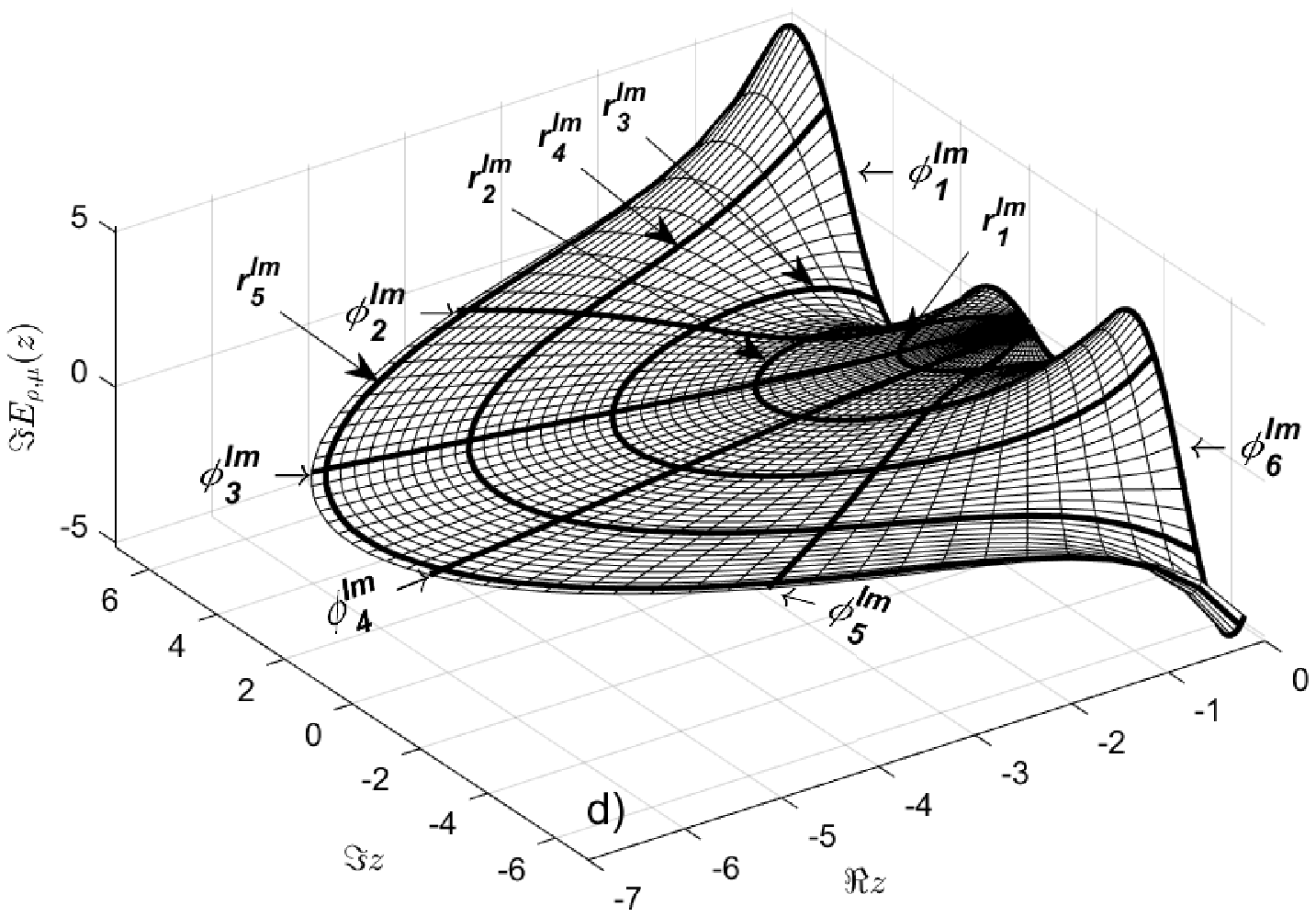}\\[4mm]
  \includegraphics[width=0.42\textwidth]{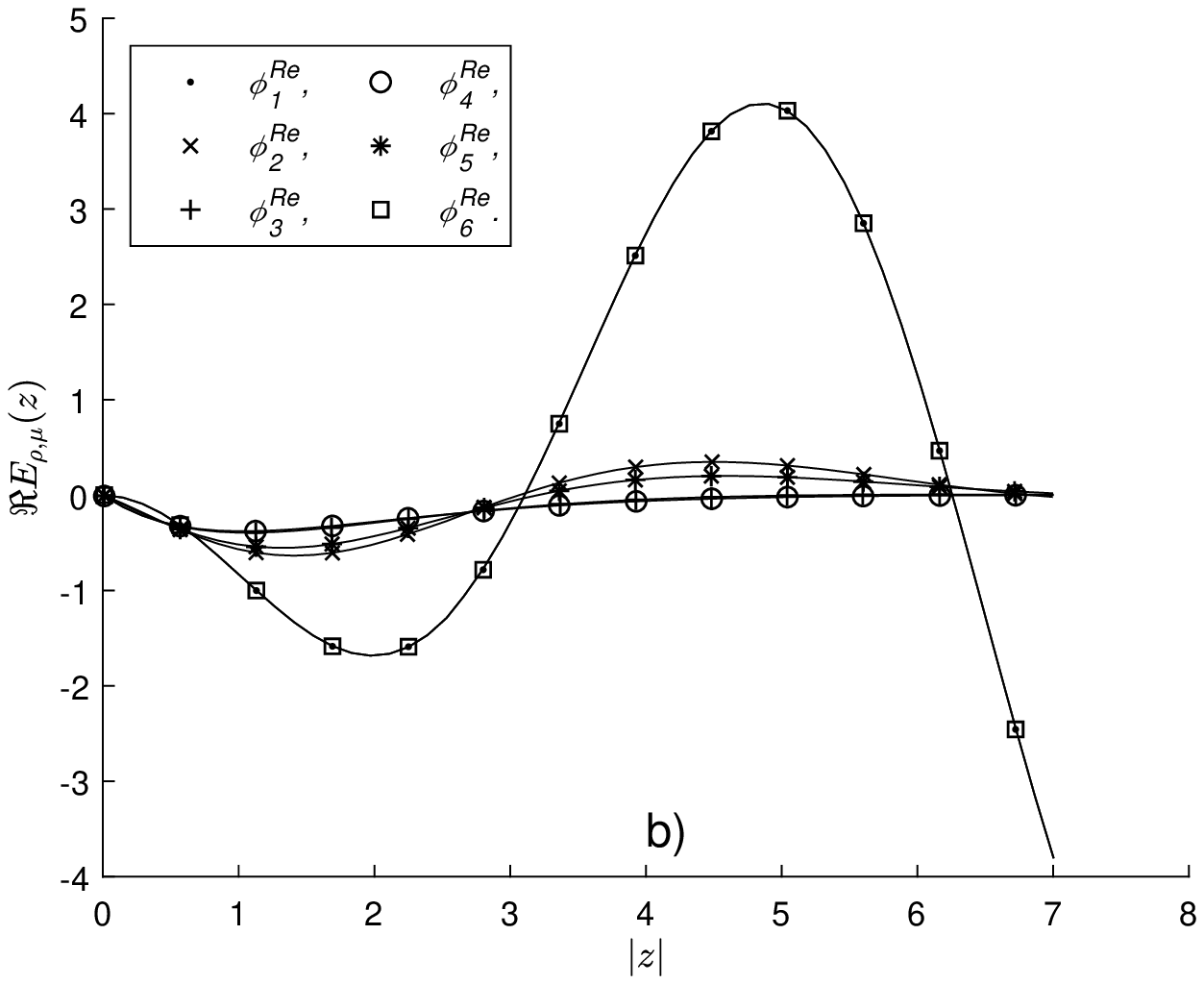}\hfill
  \includegraphics[width=0.42\textwidth]{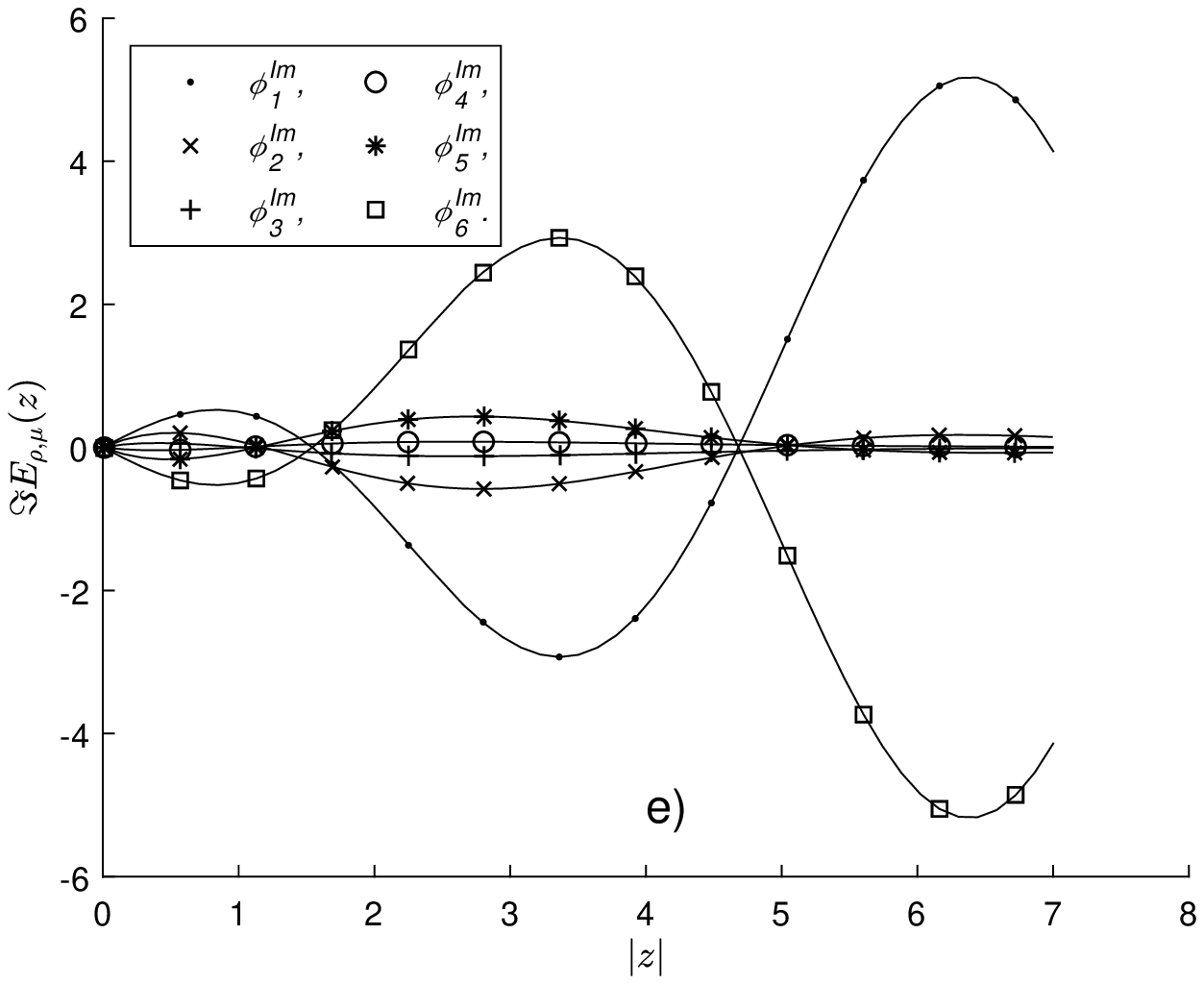}\\[4mm]
  \includegraphics[width=0.42\textwidth]{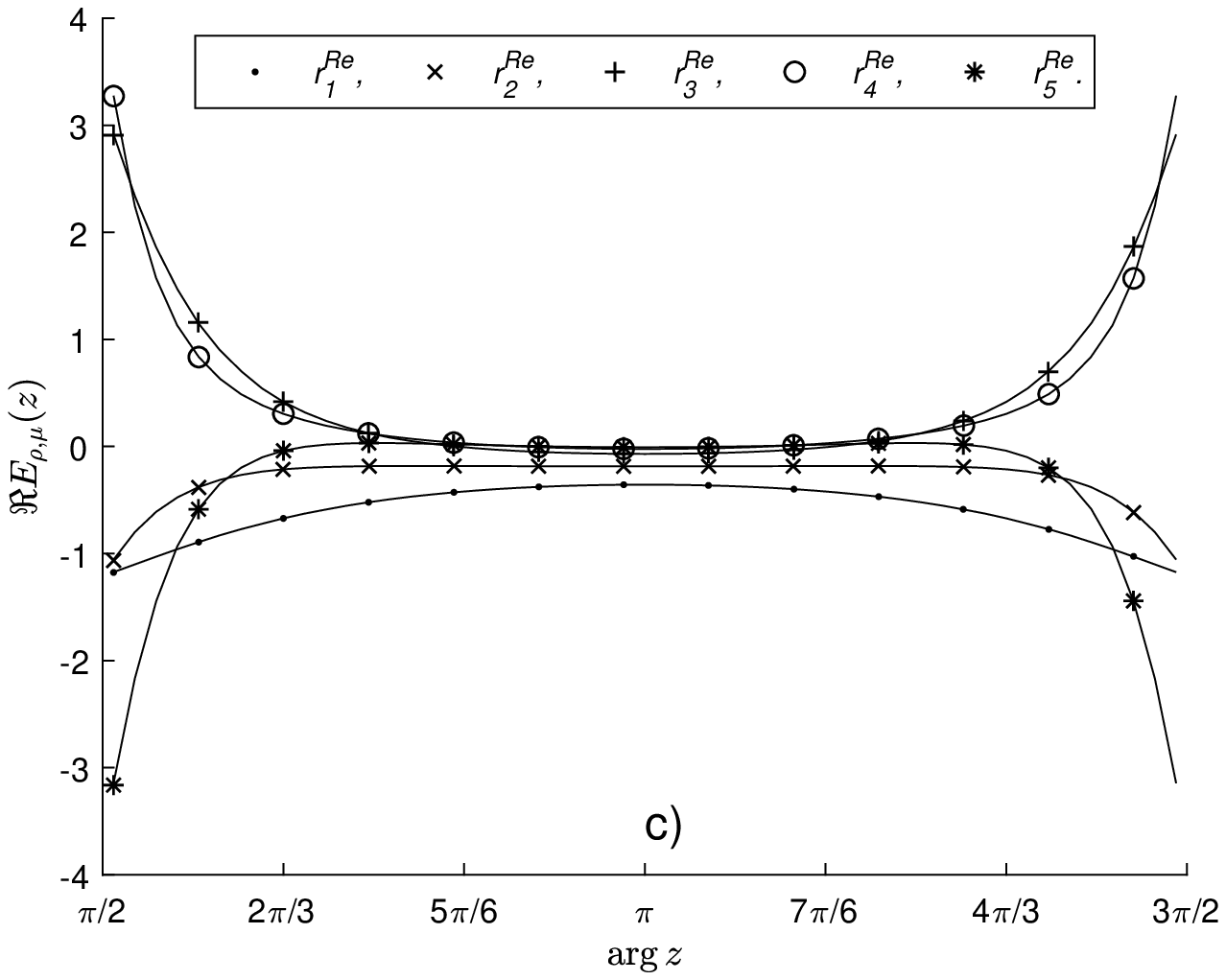}\hfill
  \includegraphics[width=0.42\textwidth]{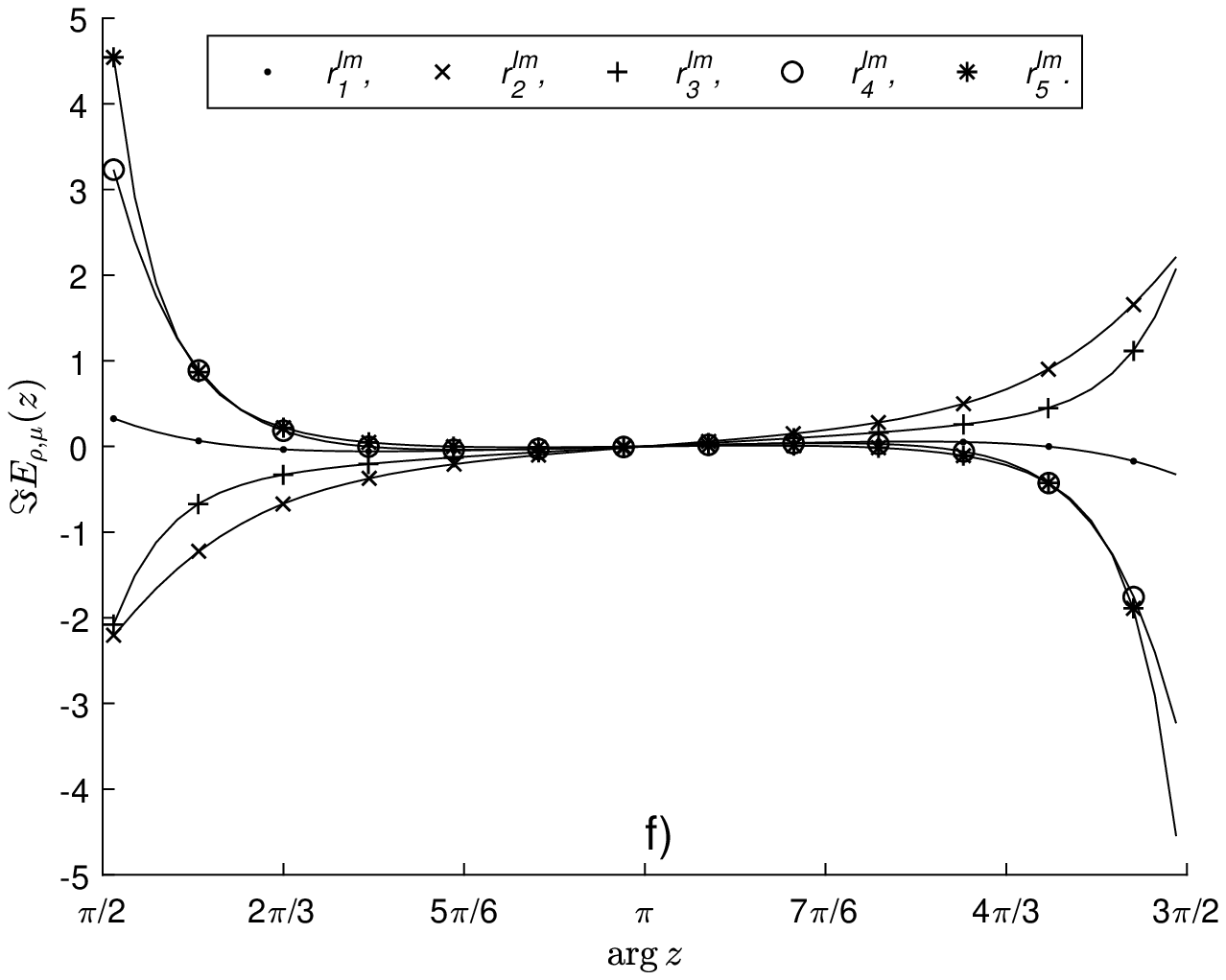}
  \caption{The function $E_{\rho,\mu}(z)$ for $\rho=1, \mu=0$ and $0.01\leqslant|z|\leqslant7, \pi/2<\arg z<3\pi/2$.  On the figures a) and d) the surfaces – the formula (\ref{eq:MLF_int2}), the curves – the formula (\ref{eq:MLF_mu<=1}). On the figures b), c), e) and f) the curves - the formula (\ref{eq:MLF_int2}), the points – the formula (\ref{eq:MLF_mu<=1})}\label{fig:MLF_int2_rho1_mu0}
\end{figure}

\begin{figure}
  \centering
  \includegraphics[width=0.43\textwidth]{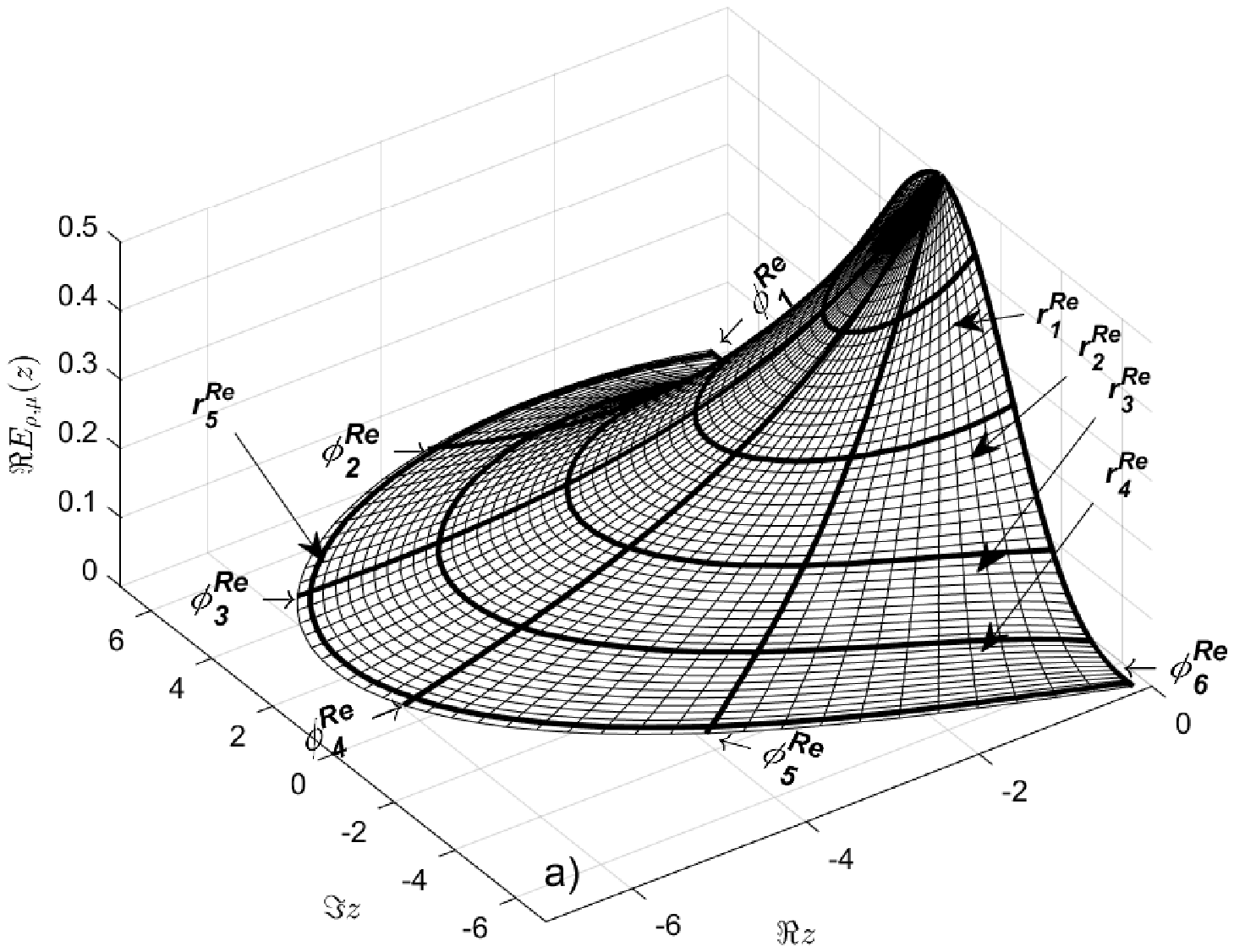}\hfill
  \includegraphics[width=0.43\textwidth]{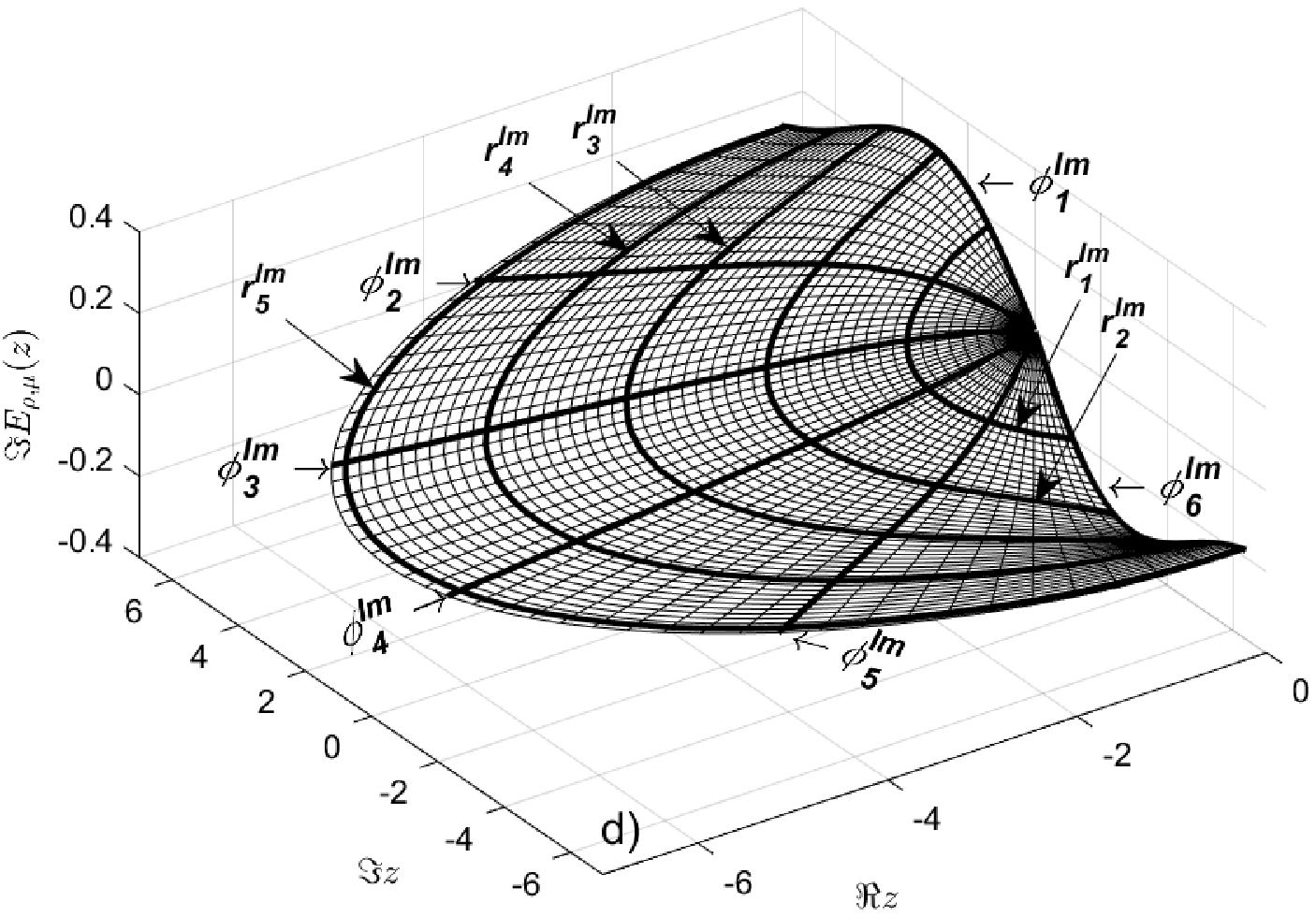}\\[4mm]
  \includegraphics[width=0.43\textwidth]{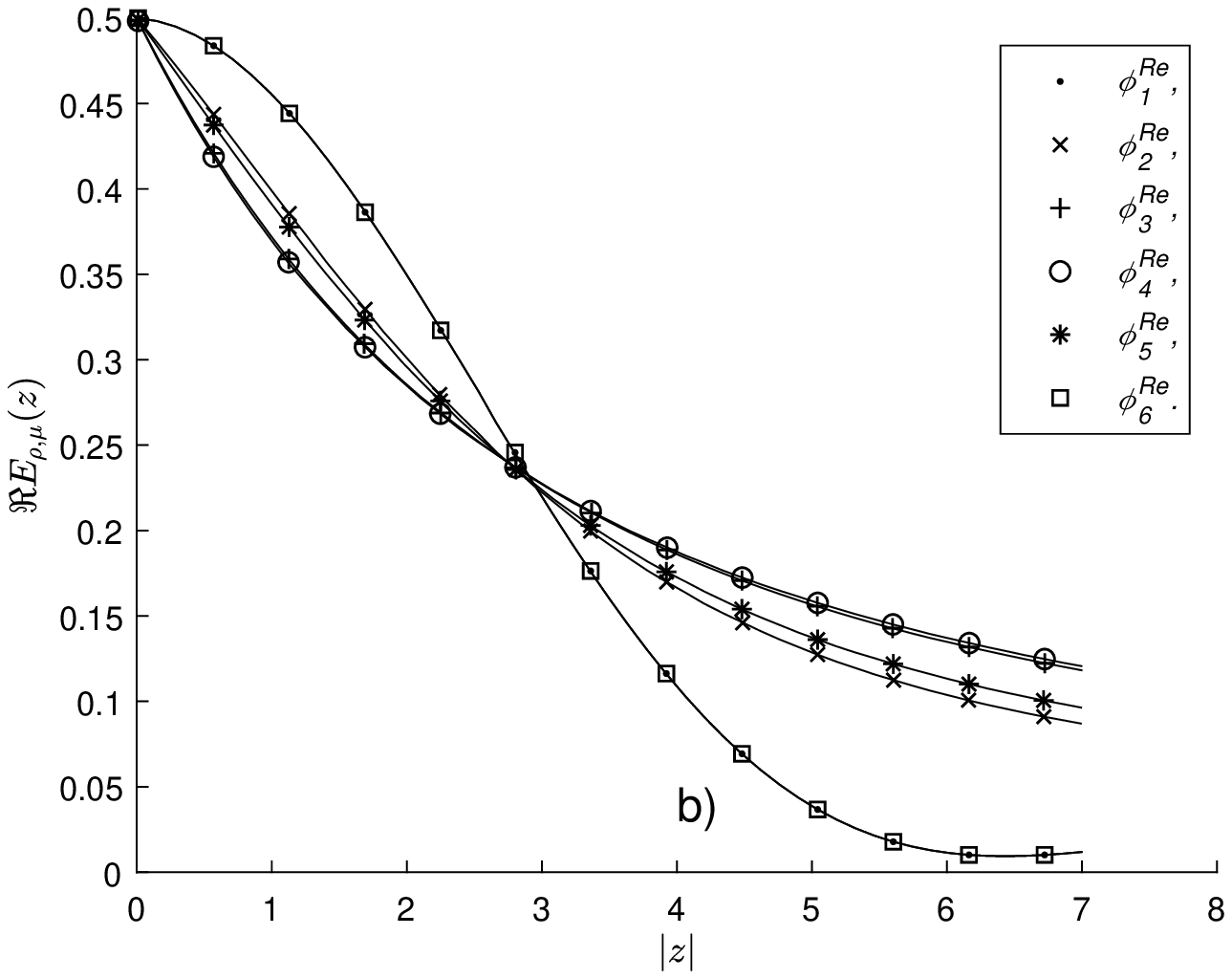}\hfill
  \includegraphics[width=0.43\textwidth]{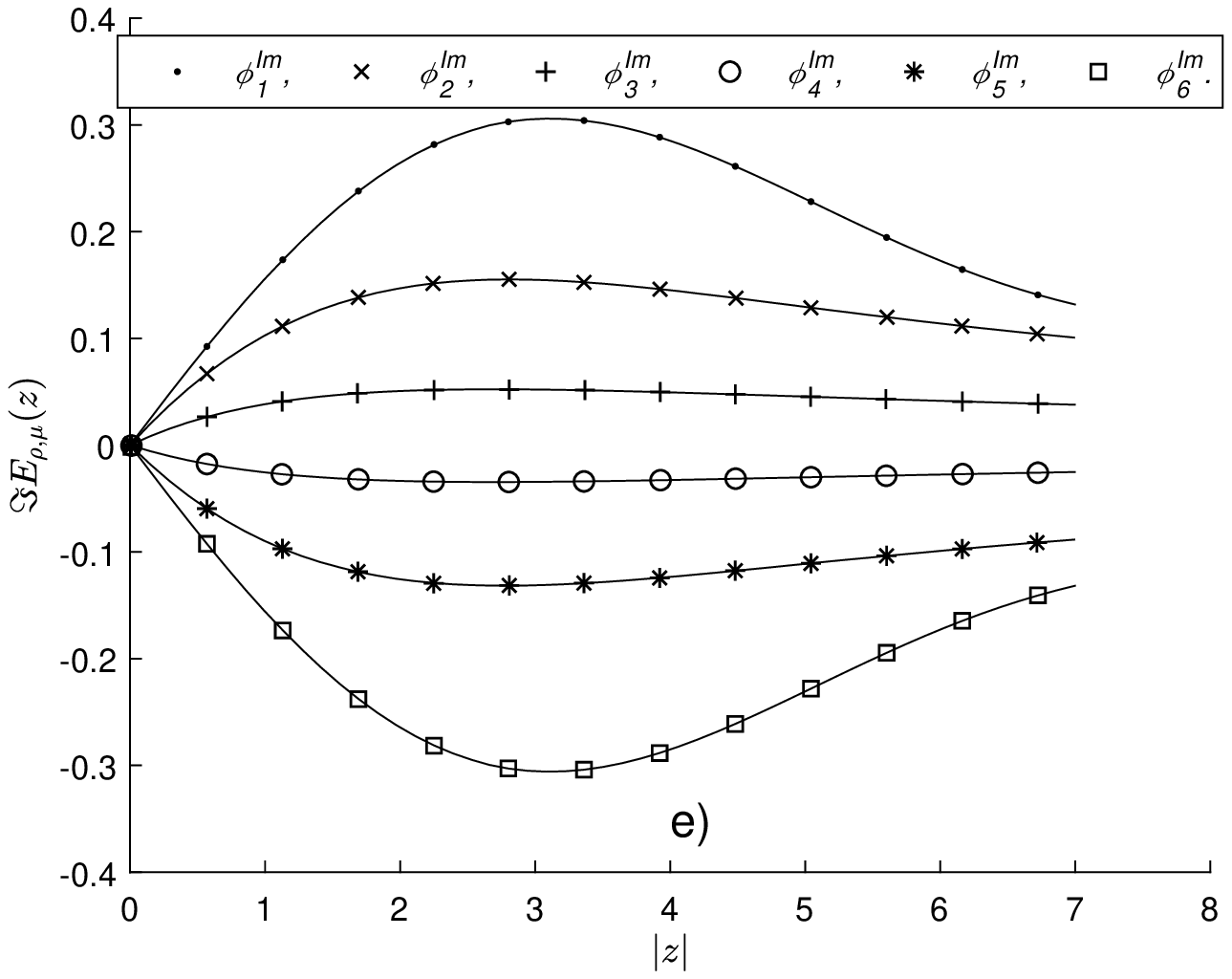}\\[4mm]
  \includegraphics[width=0.43\textwidth]{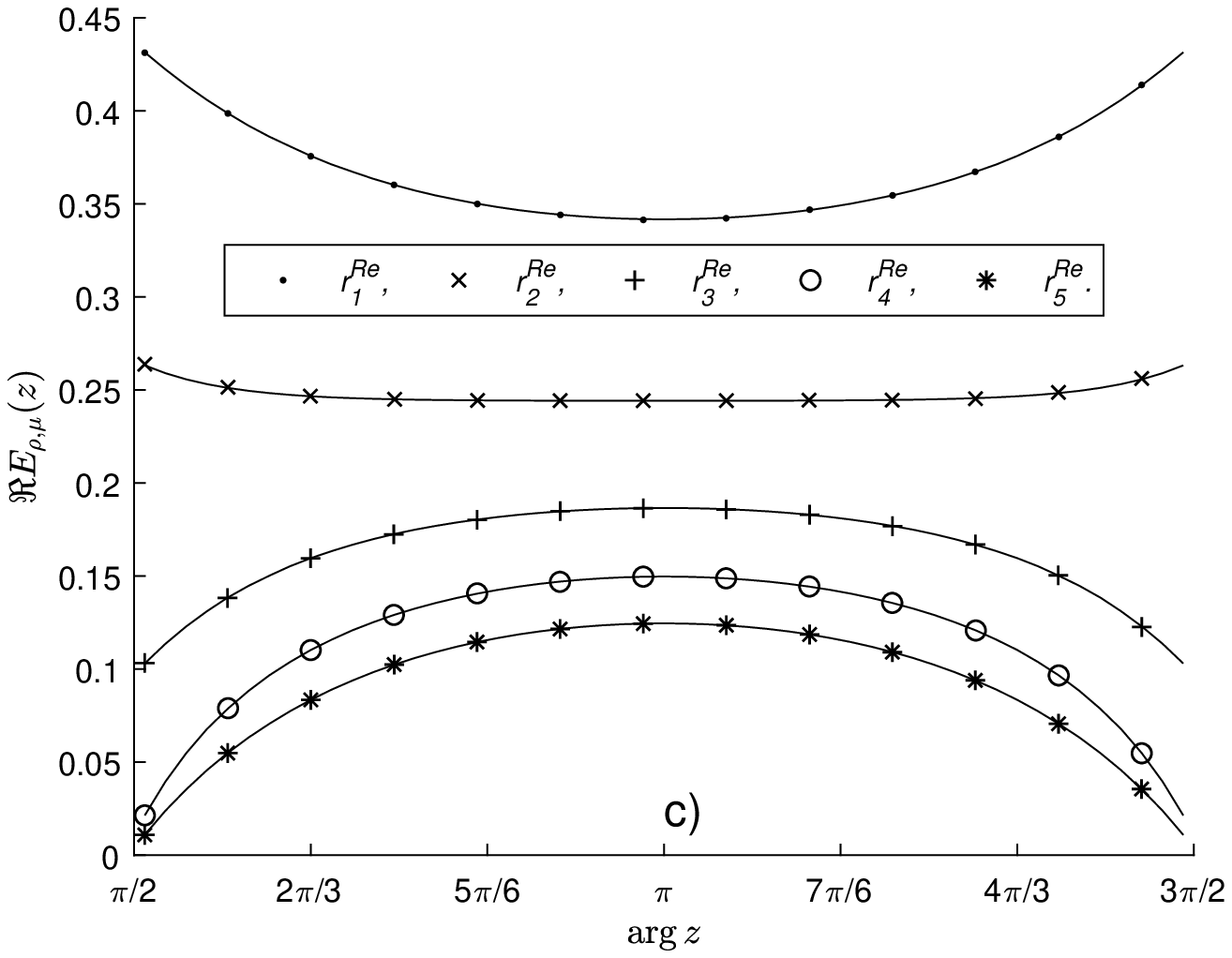}\hfill
  \includegraphics[width=0.43\textwidth]{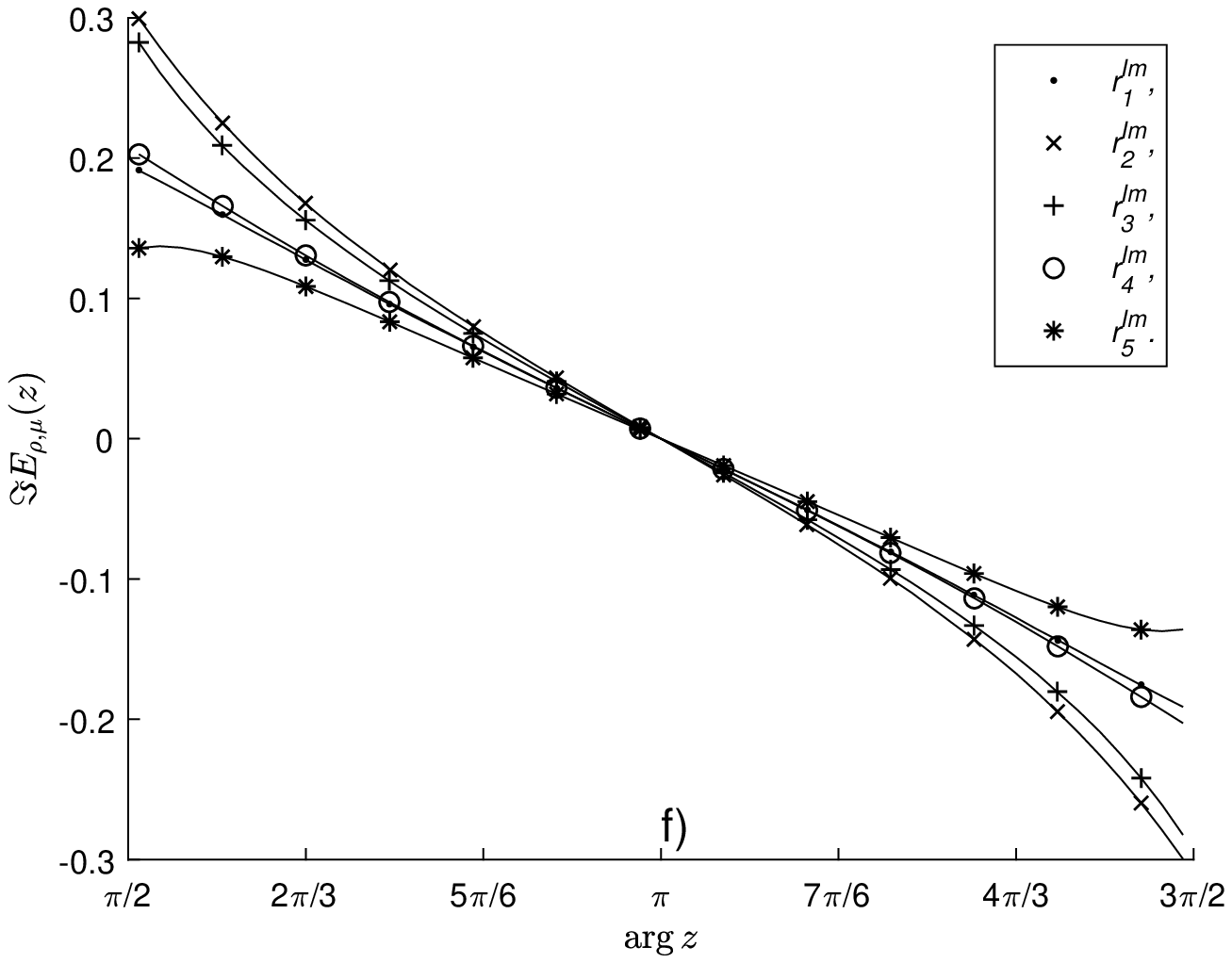}
  \caption{The function $E_{\rho,\mu}(z)$ for $\rho=1, \mu=3$ and $0.01\leqslant|z|\leqslant7, \pi/2<\arg z<3\pi/2$.  On the figures a) and d) the surfaces – the formula (\ref{eq:MLF_int2}), the curves – the formula (\ref{eq:MLF_mu>=2}). On the figures b), c), e) and f) the curves – the formula (\ref{eq:MLF_int2}), the points – the formula (\ref{eq:MLF_mu>=2}) }\label{fig:MLF_int2_rho1_mu3}
\end{figure}

In Fig.~\ref{fig:MLF_int2_rho1_mu0}~a) the real part of the representation (\ref{eq:MLF_int2}) is constructed and in Fig.~\ref{fig:MLF_int2_rho1_mu0}~d) the imaginary part of the representation (\ref{eq:MLF_int2}) for the values of the parameters $\rho=1,\mu=0$ and complex $z$ that changes in the limits $0.01\leqslant |z|\leqslant 7$, $\pi/2<\arg z<3\pi/2$. It is important to pay attention that in Theorem~\ref{lemm:MLF_int2} the argument $z$ was represented in the form $z=|z|e^{i\arg z}$, in this connection, in calculations it is convenient to use the polar coordinate system. This explains the concentric form of the calculation area. To verify the correctness of the representation (\ref{eq:MLF_int2}) and the correctness of the numerical calculation in Fig.~\ref{fig:MLF_int2_rho1_mu0}~a) and in Fig.~\ref{fig:MLF_int2_rho1_mu0}~d) the results of the calculation of the function $E_{1,0}(z)$ are given by the formula  (\ref{eq:MLF_mu<=1}). The values $\Re E_{1,0}(z)$ and $\Im E_{1,0}(z)$, calculated by the formula (\ref{eq:MLF_mu<=1}) are given in Fig.~1~a) and in Fig.~1~d) in thick curves. The fixed values $\arg z$ correspond to the thick curves $\phi_i^{Re},\ i=1,\dots,6$ in Fig.~\ref{fig:MLF_int2_rho1_mu0}~a) and the curves $\phi_i^{Im},\ i=1,\dots,6$ in Fig.~\ref{fig:MLF_int2_rho1_mu0}~d). The fixed values $|z|$ correspond to the thick curves $r_i^{Re},\ i=1,\dots,5$ in Fig.~\ref{fig:MLF_int2_rho1_mu0}~a) and the curves $r_i^{Im},\ i=1,\dots,6$ in Fig.~\ref{fig:MLF_int2_rho1_mu0}~d). As one can see from the figures, these curves lie exactly on the surface of the function $E_{1,0}(z)$, calculated by the formula (\ref{eq:MLF_int2}).  To verify the coincidence of the calculation results by the formulas (\ref{eq:MLF_int2}) and (\ref{eq:MLF_mu<=1}) in Fig.~\ref{fig:MLF_int2_rho1_mu0}~b) and Fig.~\ref{fig:MLF_int2_rho1_mu0}~e) the plots of $\Re E_{1,0}(z)$ and $\Im E_{1,0}(z)$ against $|z|$ are given at the fixed value $\arg z$. The values $\arg z$, for which these plots were constructed,  correspond to the curves $\phi_i^{Re},\ i=1,\dots,6$ in  Fig.~\ref{fig:MLF_int2_rho1_mu0}~a) and $\phi_i^{Im},\ i=1,\dots,6$  in Fig.~\ref{fig:MLF_int2_rho1_mu0}~d). In these figures, the solid curves correspond to the representation (\ref{eq:MLF_int2}), the points correspond to the formula~(\ref{eq:MLF_mu<=1}). As one can see, the calculation results by these two formulas coincide absolutely. In Fig.~\ref{fig:MLF_int2_rho1_mu0}~c) and Fig.~\ref{fig:MLF_int2_rho1_mu0}~f) the plots of $\Re E_{1,0}(z)$ and $\Im E_{1,0}(z)$ against $\arg z$ are given with the fixed value $|z|$. The values  $|z|$, for which these plots were constructed, correspond to the values $|z|$ the curves $r_i^{Re}$, $i=1,\dots,5$ in Fig.~\ref{fig:MLF_int2_rho1_mu0}~a) and $r_i^{Im}$, $i=1,\dots,5$, in Fig.~\ref{fig:MLF_int2_rho1_mu0}~d). As in the previous case, the solid curves are the calculation results by the formula  (\ref{eq:MLF_int2}), the points – the calculation results by the formula (\ref{eq:MLF_mu<=1}). As one can see from the graphs, the calculation results by these two formulas match absolutely which means that the representation  (\ref{eq:MLF_int2}) is correct.

In Fig.~\ref{fig:MLF_int2_rho1_mu3}~a) the real part of the representation  (\ref{eq:MLF_int2}) is constructed,  and in Fig.~\ref{fig:MLF_int2_rho1_mu3}~d) the imaginary part of the representation (\ref{eq:MLF_int2}) for values of the parameters $\rho=1,\mu=3$ and complex $z$ that changes in the limits  $0.01\leqslant |z|\leqslant 7$, $\pi/2<\arg z<3\pi/2$. The notation in Fig.~\ref{fig:MLF_int2_rho1_mu3} is similar to the notation in Fig.~\ref{fig:MLF_int2_rho1_mu0} with one exception that for verifying the calculation results the formula (\ref{eq:MLF_mu>=2}) was used. As we can see, in this case the calculation results with the help of the representation   (\ref{eq:MLF_int2}) coincide absolutely with the formula (\ref{eq:MLF_mu>=2}), which confirms the correctness of calculations made and validity of the integral representation for the function $E_{\rho,\mu}(z)$ formulated in Theorem~\ref{lemm:MLF_int2}.

We consider Parametrizations 2 and 3 for the representation “A”. In the work \cite{Saenko2020d} it was shown that in these two cases the form of the kernel  $K_{\rho,\mu}(r,-\delta_{1\rho},\delta_{2\rho},z)$ become simple (see corollary 1 in \cite{Saenko2020d}).  The similar situation takes place for the representation (\ref{eq:MLF_int2}). We formulate the result in the form of a corollary.

\begin{corollary}\label{coroll:MLF_int2_delatRho}
For any real $\rho>1/2$ and $\epsilon>0$, any complex $\mu=\mu_R+i\mu_I$ and any complex $z=te^{i\theta}$ satisfying the conditions $\frac{\pi}{2\rho}-\delta_\rho+\pi<\theta<-\frac{\pi}{2\rho}+\delta_\rho+\pi$,

1) at any real $\delta_\rho$ satisfying the condition
\begin{equation}\label{eq:deltaRho_cond_corol_int2_deltaRho}
\pi/(2\rho)<\delta_\rho\leqslant\min\left(\pi,\pi/\rho\right)
\end{equation}
the Mittag-Leffler function can be represented in the form:
\begin{multline}\label{eq:MLF_int2_deltaRho}
  E_{\rho,\mu}(z)=\int_{1+\epsilon}^{\infty} K_{\rho,\mu}^{Re}(r,\delta_\rho,t,\theta)dr+
  \int_{-\delta_\rho-\pi}^{\delta_\rho-\pi} P_{\rho,\mu}^{Re}(1+\epsilon,\varphi,t,\theta)d\varphi+\\
  +i\left\{\int_{1+\epsilon}^{\infty} K_{\rho,\mu}^{Im}(r,\delta_\rho,t,\theta)dr+
  \int_{-\delta_\rho-\pi}^{\delta_\rho-\pi} P_{\rho,\mu}^{Im}(1+\epsilon,\varphi,t,\theta)d\varphi\right\},
\end{multline}
where
\begin{multline}\label{eq:K_Re_deltaRho_corol_int2}
K_{\rho,\mu}^{Re}(r,\delta_\rho,t,\theta)=\frac{\rho}{2\pi}\frac{(tr)^{\rho(1-\mu_R)}}{(r^2+2r\cos\delta_\rho+1)}\\
\times [f(r,\delta_\rho-\pi,t,\theta)\left(r\sin(\xi(r,\delta_\rho-\pi,t,\theta))+ \sin(\xi(r,\delta_\rho-\pi,t,\theta)+\delta_\rho)\right)\\
-f(r,-\delta_\rho-\pi,t,\theta)\left.\left(r\sin(\xi(r,-\delta_\rho-\pi,t,\theta))+ \sin(\xi(r,-\delta_\rho-\pi,t,\theta)-\delta_\rho)\right)\right],
\end{multline}
\begin{multline}\label{eq:K_Im_deltaRho_corol_int2}
K_{\rho,\mu}^{Im}(r,\delta_\rho,t,\theta)=\frac{\rho}{2\pi}\frac{(tr)^{\rho(1-\mu_R)}}{(r^2+2r\cos\delta_\rho+1)}\\
\times [f(r,-\delta_\rho-\pi,t,\theta)\left(r\cos(\xi(r,-\delta_\rho-\pi,t,\theta))+ \cos(\xi(r,-\delta_\rho-\pi,t,\theta)-\delta_\rho)\right)\\
-f(r,\delta_\rho-\pi,t,\theta)\left.\left(r\cos(\xi(r,\delta_\rho-\pi,t,\theta))+ \cos(\xi(r,\delta_\rho-\pi,t,\theta)+\delta_\rho)\right)\right],
\end{multline}
$P_{\rho,\mu}^{Re}(r,\varphi,t,\theta)$ is defined by the expression  (\ref{eq:P_Re_lemm_int2}), $P_{\rho,\mu}^{Im}(r,\varphi,t,\theta)$ is defined by the expression (\ref{eq:P_Im_lemm_int2}), and the functions $f(r,\varphi,t,\theta)$ and $\xi(r,\varphi,t,\theta)$ are defined by the expressions  (\ref{eq:f_int2}) and (\ref{eq:xi_int2}), respectively.

2)  At $\rho\geqslant1$ and $\delta_\rho=\pi/\rho$, the Mittag-Leffler function can be represented in the form:
\begin{multline}\label{eq:MLF_int2_piRho}
  E_{\rho,\mu}(z)=\int_{1+\epsilon}^{\infty} K_{\rho,\mu}^{Re}(r,t,\theta)dr+
  \int_{-\tfrac{\pi}{\rho}-\pi}^{\tfrac{\pi}{\rho}-\pi} P_{\rho,\mu}^{Re}(1+\epsilon,\varphi,t,\theta)d\varphi+\\
  +i\left\{\int_{1+\epsilon}^{\infty} K_{\rho,\mu}^{Im}(r,t,\theta)dr+
  \int_{-\tfrac{\pi}{\rho}-\pi}^{\tfrac{\pi}{\rho}-\pi} P_{\rho,\mu}^{Im}(1+\epsilon,\varphi,t,\theta)d\varphi\right\},
\end{multline}
where
\begin{multline}\label{eq:K_Re_piRho_corol_int2}
  K_{\rho,\mu}^{Re}(r,t,\theta)  =\frac{\rho}{2\pi}\frac{(tr)^{\rho(1-\mu_R)} e^{\varrho(r,t,\theta)}}{(r^2+2r\cos(\pi/\rho)+1)}\\
  \times \left[e^{\mu_I\pi}(r\sin(\xi'(r,t,\theta)+(1-\mu_R)\pi)+\sin(\xi'(r,t,\theta)+(1-\mu_R)\pi+\pi/\rho))\right.\\
  -\left.  e^{-\mu_I\pi}(r\sin(\xi'(r,t,\theta)-(1-\mu_R)\pi)+\sin(\xi'(r,t,\theta)-(1-\mu_R)\pi-\pi/\rho))\right],
\end{multline}
\begin{multline}\label{eq:K_Im_piRho_corol_int2}
  K_{\rho,\mu}^{Im}(r,t,\theta)  =\frac{\rho}{2\pi}\frac{(tr)^{\rho(1-\mu_R)} e^{\varrho(r,t,\theta)}}{(r^2+2r\cos(\pi/\rho)+1)}\\
  \times \left[e^{-\mu_I\pi}(r\cos(\xi'(r,t,\theta)-(1-\mu_R)\pi)+\cos(\xi'(r,t,\theta)-(1-\mu_R)\pi-\pi/\rho))\right. \\
  -\left.  e^{\mu_I\pi}(r\cos(\xi'(r,t,\theta)+(1-\mu_R)\pi)+\cos(\xi'(r,t,\theta)+(1-\mu_R)\pi+\pi/\rho))\right],
\end{multline}
$P_{\rho,\mu}^{Re}(r,\varphi,t,\theta)$ is defined by the expression (\ref{eq:P_Re_lemm_int2}), $P_{\rho,\mu}^{Im}(r,\varphi,t,\theta)$ is defined by the expression (\ref{eq:P_Im_lemm_int2}), and the functions $\varrho(r,t,\theta)$ and $\xi'(r,t,\theta)$ have the form
\begin{align}
  \varrho(r,t,\theta) & =-(tr)^\rho\cos(\rho(\theta-\pi))+\rho\mu_I(\theta-\pi),\label{eq:varrho_int2} \\
  \xi'(r,t,\theta) & =-(tr)^\rho\sin(\rho(\theta-\pi))+\rho(1-\mu_R)(\theta-\pi)-\rho\mu_I\ln(tr)\label{eq:xi'_int2}.
\end{align}
\end{corollary}

\begin{proof}
1) According to Theorem~\ref{lemm:MLF_int2} the Mittag-Leffler function can be represented in the form~(\ref{eq:MLF_int2}). This representation is true for any real  $\delta_{1\rho}$ and $\delta_{2\rho}$ satisfying the conditions (\ref{eq:deltaRho_cond_lemm_int2}). In the case, if $\delta_{1\rho}=\delta_{2\rho}=\delta_\rho$ these two conditions are transformed into one condition imposed on the value $\delta_\rho$
\begin{equation*}
  \pi/(2\rho)<\delta_\rho\leqslant\min\left(\pi,\pi/\rho\right).
\end{equation*}
We represent the complex number $z$ in the form $z=te^{i\theta}$. As a result, in view of the foregoing, the condition (\ref{eq:argZ_cond_lemm_int2}) is written in the form $\pi/(2\rho)-\delta_{\rho}+\pi<\theta<-\pi/(2\rho)+\delta_{\rho}+\pi$. Substituting in (\ref{eq:MLF_int2}) the value  $\delta_\rho$, instead of $\delta_{1\rho}$ and $\delta_{2\rho}$ we obtain
\begin{multline}\label{eq:MLF_int2_corol_deltaRho_tmp}
  E_{\rho,\mu}(z)=\int_{1+\epsilon}^{\infty}K_{\rho,\mu}^{Re}(r,-\delta_{\rho},\delta_{\rho},t,\theta)dr
  + \int_{-\delta_{\rho}-\pi}^{\delta_{\rho}-\pi}P_{\rho,\mu}^{Re}(1+\epsilon,\varphi,t,\theta)d\varphi\\
  +i\left\{ \int_{1+\epsilon}^{\infty}K_{\rho,\mu}^{Im}(r,-\delta_{\rho},\delta_{\rho},t,\theta)dr
  + \int_{-\delta_{\rho}-\pi}^{\delta_{\rho}-\pi}P_{\rho,\mu}^{Im}(1+\epsilon,\varphi,t,\theta)d\varphi\right\}.
\end{multline}

It is clear from this expression that in the summands, c $P_{\rho,\mu}^{Re}(r,\varphi,t,\theta)$ and  $P_{\rho,\mu}^{Im}(r,\varphi,t,\theta)$ only the limits of integration have changed but the integrand expressions do not change.  However, in the summands  c $K_{\rho,\mu}^{Re}(r,-\delta_\rho,\delta_\rho,t,\theta)$ and $K_{\rho,\mu}^{Im}(r,-\delta_\rho,\delta_\rho,t,\theta)$ with this substitution the integrand functions will change. We consider how these functions are transformed. For brevity, we will introduce the notations
\begin{equation}\label{eq:KRe_deltaRho_KIm_deltaRho}
K_{\rho,\mu}^{Re}(r,\delta_\rho,t,\theta)\equiv K_{\rho,\mu}^{Re}(r,-\delta_\rho,\delta_\rho,t,\theta),\quad K_{\rho,\mu}^{Im}(r,\delta_\rho,t,\theta)\equiv K_{\rho,\mu}^{Im}(r,-\delta_\rho,\delta_\rho,t,\theta).
\end{equation}
Substituting in (\ref{eq:K_Re_lemm_int2}) $\delta_{1\rho}=\delta_{2\rho}=\delta_\rho$ we get
\begin{multline*}
  K_{\rho,\mu}^{Re}(r,-\delta_\rho,\delta_\rho,t,\theta)\equiv K_{\rho,\mu}^{Re}(r,\delta_\rho,t,\theta)
    = \frac{\rho}{2\pi}\frac{(tr)^{\rho(1-\mu_R)}} {(r^2+2r\cos\delta_\rho+1)}\\
  \times [f(r,\delta_\rho-\pi,t,\theta) \left(r\sin(\xi(r,\delta_\rho-\pi,t,\theta))+ \sin(\xi(r,\delta_\rho-\pi,t,\theta)+\delta_\rho)\right)\nonumber\\
-f(r,-\delta_\rho-\pi,t,\theta)\left.\left(r\sin(\xi(r,-\delta_\rho-\pi,t,\theta))+ \sin(\xi(r,-\delta_\rho-\pi,t,\theta)-\delta_\rho)\right)\right].
\end{multline*}

Similarly, after substitution $\delta_{1\rho}=\delta_{2\rho}=\delta_\rho$ in (\ref{eq:K_Im_lemm_int2}) we get the expression for  $K_{\rho,\mu}^{Im}(r,\delta_\rho,t,\theta)$:
\begin{multline*}
 K_{\rho,\mu}^{Im}(r,-\delta_\rho,\delta_\rho,t,\theta)\equiv K_{\rho,\mu}^{Im}(r,\delta_\rho,t,\theta)= \frac{\rho}{2\pi}\frac{(tr)^{\rho(1-\mu_R)}}{(r^2+2r\cos\delta_\rho+1)}\\
\times[f(r,-\delta_\rho-\pi,t,\theta)\left(r\cos(\xi(r,-\delta_\rho-\pi,t,\theta))+ \cos(\xi(r,-\delta_\rho-\pi,t,\theta)-\delta_\rho)\right)\\
-f(r,\delta_\rho-\pi,t,\theta)\left(r\cos(\xi(r,\delta_\rho-\pi,t,\theta))+ \cos(\xi(r,\delta_\rho-\pi,t,\theta)+\delta_\rho)\right)].
\end{multline*}
Using now $K_{\rho,\mu}^{Re}(r,\delta_\rho,t,\theta)$ and $K_{\rho,\mu}^{Im}(r,\delta_\rho,t,\theta)$ in (\ref{eq:MLF_int2_corol_deltaRho_tmp}) we come to the representation  (\ref{eq:MLF_int2_deltaRho}). Since no additional constraints were introduced for the values $\rho$ and $\mu$ in the proof, then the conditions for the values of these parameters pass from Theorem~\ref{lemm:MLF_int2} without change.  The first part of the corollary is proved.

2) We consider now the case $\delta_\rho=\pi/\rho$.  As one can see from the condition (\ref{eq:deltaRho_cond_corol_int2_deltaRho}) this case can be implemented only at values $\rho\geqslant1$.  Substituting now the value $\delta_\rho=\pi/\rho$ in (\ref{eq:MLF_int2_deltaRho}) we get
\begin{multline}\label{eq:MLF_int2_piRho_tmp}
  E_{\rho,\mu}(z)=\int_{1+\epsilon}^{\infty} K_{\rho,\mu}^{Re}(r,\pi/\rho,t,\theta)dr+
  \int_{-\tfrac{\pi}{\rho}-\pi}^{\tfrac{\pi}{\rho}-\pi} P_{\rho,\mu}^{Re}(1+\epsilon,\varphi,t,\theta)d\varphi+\\
  +i\left\{\int_{1+\epsilon}^{\infty} K_{\rho,\mu}^{Im}(r,\pi/\rho,t,\theta)dr+
  \int_{-\tfrac{\pi}{\rho}-\pi}^{\tfrac{\pi}{\rho}-\pi} P_{\rho,\mu}^{Im}(1+\epsilon,\varphi,t,\theta)d\varphi\right\}.
\end{multline}
In the same way as we did in the first part of the proof of this corollary, we  consider how the kernels $K_{\rho,\mu}^{Re}(r,\pi/\rho,t,\theta)$ and $K_{\rho,\mu}^{Im}(r,\pi/\rho,t,\theta)$ will change at $\delta_\rho=\pi/\rho$. Taking into consideration that $\cos(\pi+\varphi)=-\cos\varphi$ for the function  $f(r,\varphi,t,\theta)$ which is defined by the expression (\ref{eq:f_int2}), we have
\begin{multline}\label{eq:f1_corol_int2}
  f\left(r,\tfrac{\pi}{\rho}-\pi,t,\theta\right)=\exp\left\{(tr)^\rho \cos\left(\rho\left(\theta+\tfrac{\pi}{\rho}-\pi\right)\right)+ \rho\mu_I\left(\theta+\tfrac{\pi}{\rho}-\pi\right)\right\}\\
  =e^{\varrho(r,t,\theta)+\mu_I\pi},
\end{multline}
where $\varrho(r,t,\theta)$ has the form (\ref{eq:varrho_int2}).
Similarly, we obtain
\begin{equation}\label{eq:f2_corol_int2}
  f\left(r,-\tfrac{\pi}{\rho}-\pi,t,\theta\right)=e^{\varrho(r,t,\theta)-\mu_I\pi}.
\end{equation}
Taking into account that $\sin(\pi\pm\varphi)=\mp\sin\varphi$ for the function $\xi(r,\varphi,t,\theta)$ defined by the expressison (\ref{eq:xi_int2}) we obtain
\begin{equation}\label{eq:xi1_corol_int2}
  \xi\left(r,\tfrac{\pi}{\rho}-\pi,t,\theta\right) =\xi'(r,t,\theta)+(1-\mu_R)\pi,
\end{equation}
where $\xi'(r,t,\theta)$ has the form (\ref{eq:xi'_int2}). In the same way we obtain
\begin{equation}\label{eq:xi2_corol_int2}
  \xi\left(r,-\tfrac{\pi}{\rho}-\pi,t,\theta\right)=\xi'(r,t,\theta)-(1-\mu_R)\pi.
\end{equation}

Now using the expressions (\ref{eq:f1_corol_int2})~-~(\ref{eq:xi2_corol_int2}) in (\ref{eq:K_Re_deltaRho_corol_int2}) and introducing the notation
\begin{equation}\label{eq:KRe_piRho}
K_{\rho,\mu}^{Re}(r,t,\theta)\equiv K_{\rho,\mu}^{Re}(r,\pi/\rho,t,\theta),
\end{equation}
we get
\begin{multline} \label{eq:K_Re_piRho_corol_int2_tmp}
  K_{\rho,\mu}^{Re}(r,t,\theta)
  =\frac{\rho}{2\pi}\frac{(tr)^{\rho(1-\mu_R)}e^{\varrho(r,t,\theta)}}{(r^2+2r\cos(\pi/\rho)+1)}\\
  \times \left[e^{\mu_I\pi}\left(r\sin(\xi'(r,t,\theta)+(1-\mu_R)\pi)+\sin\left(\xi'(r,t,\theta)+ (1-\mu_R)\pi+\tfrac{\pi}{\rho}\right)\right)\right.\\
  -\left.e^{-\mu_I\pi}\left(r\sin(\xi'(r,t,\theta)-(1-\mu_R)\pi)+ \sin\left(\xi'(r,t,\theta)-(1-\mu_R)\pi-\tfrac{\pi}{\rho}\right)\right)\right].
\end{multline}
Similarly, substituting the expressions (\ref{eq:f1_corol_int2})~-~(\ref{eq:xi2_corol_int2}) in (\ref{eq:K_Re_deltaRho_corol_int2}) and by introducing the notation
\begin{equation}\label{eq:KIm_piRho}
K_{\rho,\mu}^{Im}(r,t,\theta)\equiv K_{\rho,\mu}^{Im}(r,\pi/\rho,t,\theta),
\end{equation}
we get
\begin{multline}\label{eq:K_Im_piRho_corol_int2_tmp}
  K_{\rho,\mu}^{Im}(r,t,\theta)=\frac{\rho}{2\pi}\frac{(tr)^{\rho(1-\mu_R)}e^{\varrho(r,t,\theta)}}{(r^2+2r\cos(\pi/\rho)+1)}\\
   \times\left[e^{-\mu_I\pi}\left(r\cos(\xi'(r,t,\theta)-(1-\mu_R)\pi)+ \cos\left(\xi'(r,t,\theta)-(1-\mu_R)\pi-\tfrac{\pi}{\rho}\right)\right)\right.\\
  -\left.e^{\mu_I\pi}\left(r\cos(\xi'(r,t,\theta)+(1-\mu_R)\pi)+\cos\left(\xi'(r,t,\theta)+(1-\mu_R)\pi+\tfrac{\pi}{\rho}\right)\right)\right].
\end{multline}
It should be pointed out that the integrands $P_{\rho,\mu}^{Re}(r,\varphi,t,\theta)$ and $P_{\rho,\mu}^{Re}(r,\varphi,t,\theta)$ do not change and are defined by the expressions (\ref{eq:P_Re_lemm_int2}) and (\ref{eq:P_Im_lemm_int2}) respectively. Now substituting (\ref{eq:K_Re_piRho_corol_int2_tmp}), (\ref{eq:K_Im_piRho_corol_int2_tmp}), (\ref{eq:P_Re_lemm_int2}) and (\ref{eq:P_Im_lemm_int2}) in (\ref{eq:MLF_int2_piRho_tmp}) we arrive at the representation (\ref{eq:MLF_int2_piRho}).
\begin{flushright}
  $\Box$
\end{flushright}
\end{proof}

The proved corollary shows that in the case when the parameters $\delta_{1\rho}$ and $\delta_{2\rho}$ are equal to each other (Parameterization~2)  the functions $K_{\rho,\mu}^{Re}(r,\varphi,t,\theta)$ and  $K_{\rho,\mu}^{Im}(r,\varphi,t,\theta)$ become simple. These functions take a simpler form if $\delta_\rho=\pi/\rho$ (parameterization~3). However, in this case the parameter $\rho$ can take only values $\rho\geqslant1$. It should be noted that the representation (\ref{eq:MLF_int2_piRho}) is also true for the value $\rho=1$ because in this case the contour of integration $\gamma_\zeta$  in the integral representation of the function $E_{\rho,\mu}(z)$ (see \cite{Saenko2020}) will not pass through the singular point $\zeta=1$.

In Fig.~\ref{fig:MLF_int2_deltaRho_rho1_mu2_AB} the calculation results of the function $E_{\rho,\mu}(z)$ are given for the value parameters $\rho=1$, $\mu=2$, $\delta_\rho=\pi$ with the use of the representation (\ref{eq:MLF_int2_deltaRho}), and in Fig.~\ref{fig:MLF_int2_piRho_rho1_mu4_AC} the calculation results are given for the parameters $\rho=1$, $\mu=4$ with the use of the representation (\ref{eq:MLF_int2_piRho}). In these two figures we can see the following: in Figure~a) the real part  $E_{\rho,\mu}(z)$ is given, and in Figure~d) the imaginary part of the function $E_{\rho,\mu}(z)$ calculated by the formulas (\ref{eq:MLF_int2_deltaRho}) and (\ref{eq:MLF_int2_piRho}) respectively. The argument $z$ of the function $E_{\rho,\mu}(z)$ changes within the limits  $0.01\leqslant |z|\leqslant 7$, $\pi/2<\arg z<3\pi/2$. To verify the validity of the calculations made in figures the calculation results of the function $E_{\rho,\mu}(z)$ are given according to the formula (\ref{eq:MLF_mu>=2}). The values $\Re E_{\rho,\mu}(z)$ and $\Im E_{\rho,\mu}(z)$, calculated according to the formula (\ref{eq:MLF_mu>=2}) are given in figures a) and d) in the thick curves. The fixed values of $\arg z$ correspond to the thick curves $\phi_i^{Re},\ i=1,\dots,6$ in Fig.~a) and the curves $\phi_i^{Im},\ i=1,\dots,6$ in Fig.~d). The fixed values of $|z|$ correspond to the thick curves  $r_i^{Re},\ i=1,\dots,5$ in Fig.~a) and the curves $r_i^{Im},\ i=1,\dots,6$ in Fig.~d). As one can see from these figures, these curves lie exactly on the surface of the function $E_{\rho,\mu}(z)$ calculated by the formulas (\ref{eq:MLF_int2_deltaRho}) and (\ref{eq:MLF_int2_piRho}). To be convinced that the calculation results by the  formulas (\ref{eq:MLF_int2_deltaRho}) and (\ref{eq:MLF_int2_piRho}) coincide absolutely with the calculation results by the formula (\ref{eq:MLF_mu>=2}) in Fig.~b) and Fig.~f) the plots of $\Re E_{\rho,\mu}(z)$ and $\Im E_{\rho,\mu}(z)$ against $|z|$ are given at the fixed value of $\arg z$. The values of $\arg z$, for which these plots are constructed, correspond to the curves $\phi_i^{Re},\ i=1,\dots,6$ in Fig.~a) and $\phi_i^{Im},\ i=1,\dots,6$ in Fig.~d). In Fig.~\ref{fig:MLF_int2_deltaRho_rho1_mu2_AB}~b) and Fig.~\ref{fig:MLF_int2_deltaRho_rho1_mu2_AB}~e) solid curves correspond to the representation (\ref{eq:MLF_int2_deltaRho}), and in Fig.~\ref{fig:MLF_int2_piRho_rho1_mu4_AC}~b) and in Fig.~\ref{fig:MLF_int2_piRho_rho1_mu4_AC}~e) solid curves correspond to the representation (\ref{eq:MLF_int2_piRho}). The points in these figures correspond to the formula (\ref{eq:MLF_mu>=2}). As one can see, the results of calculations by these two  formulas coincide absolutely. In Fig.~c) and Fig.~f) the plots of  $\Re E_{\rho,\mu}(z)$ and $\Im E_{\rho,\mu}(z)$ against $\arg z$ are given at the fixed value of  $|z|$. The values of $|z|$, for which these plots were constructed, correspond to the values  of $|z|$ of the curves $r_i^{Re}$, $i=1,\dots,5$ in Fig.~a) and $r_i^{Im}$, $i=1,\dots,5$ in Fig.~d). As in the previous case, solid curves are the calculation results by the formulas (\ref{eq:MLF_int2_deltaRho}) in Fig.~\ref{fig:MLF_int2_deltaRho_rho1_mu2_AB} and by the formula (\ref{eq:MLF_int2_piRho}) in Fig.~\ref{fig:MLF_int2_piRho_rho1_mu4_AC}, the points are the calculation results by the formula (\ref{eq:MLF_mu>=2}). As we can see from these graphs, the calculation results according to these two formulas coincide absolutely, which confirms the correctness of the integral representation for the function $E_{\rho,\mu}(z)$ formulated in the corollary~\ref{coroll:MLF_int2_delatRho}.

\begin{figure}
  \centering
  \includegraphics[width=0.43\textwidth]{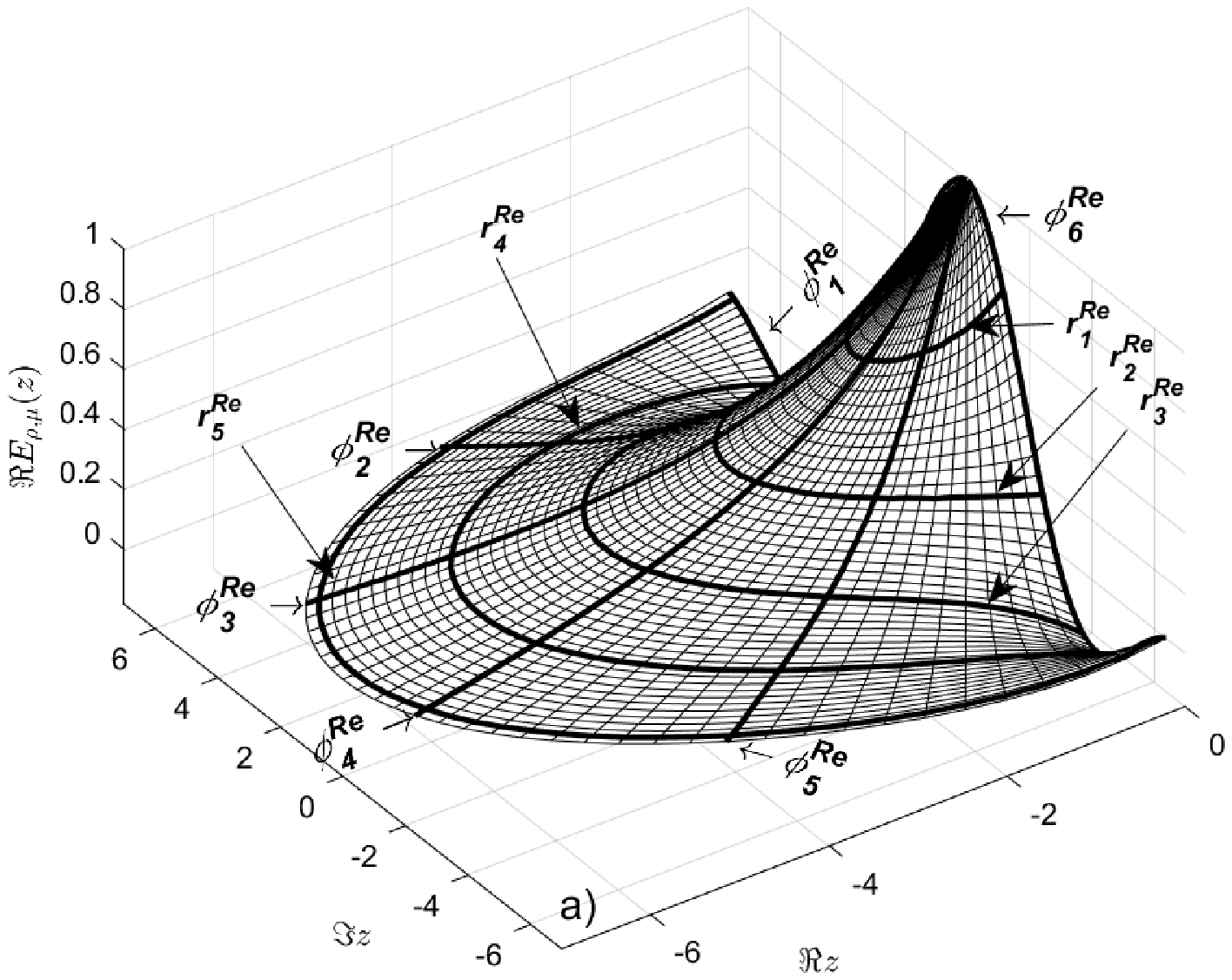}\hfill
  \includegraphics[width=0.43\textwidth]{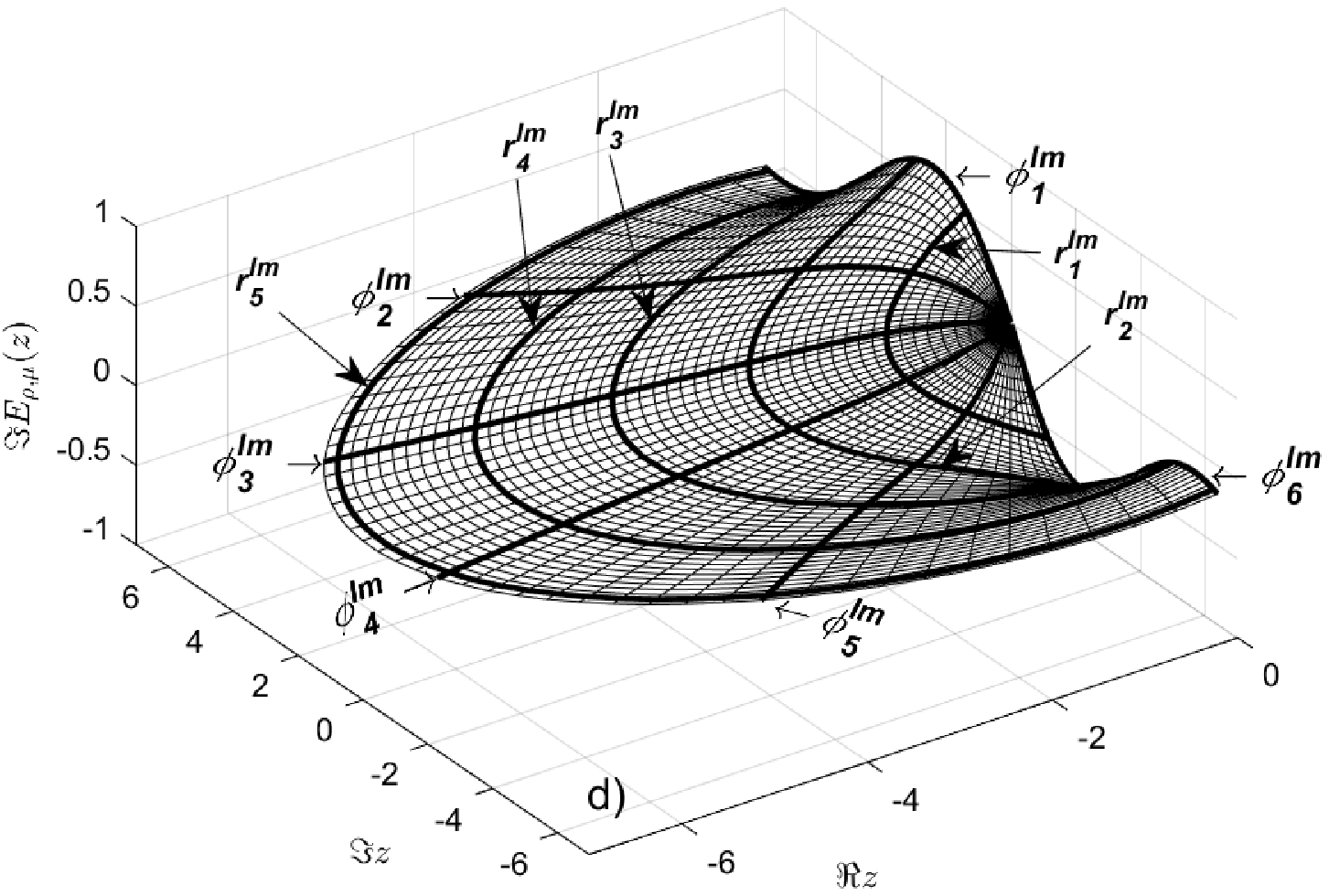}\\[4mm]
  \includegraphics[width=0.43\textwidth]{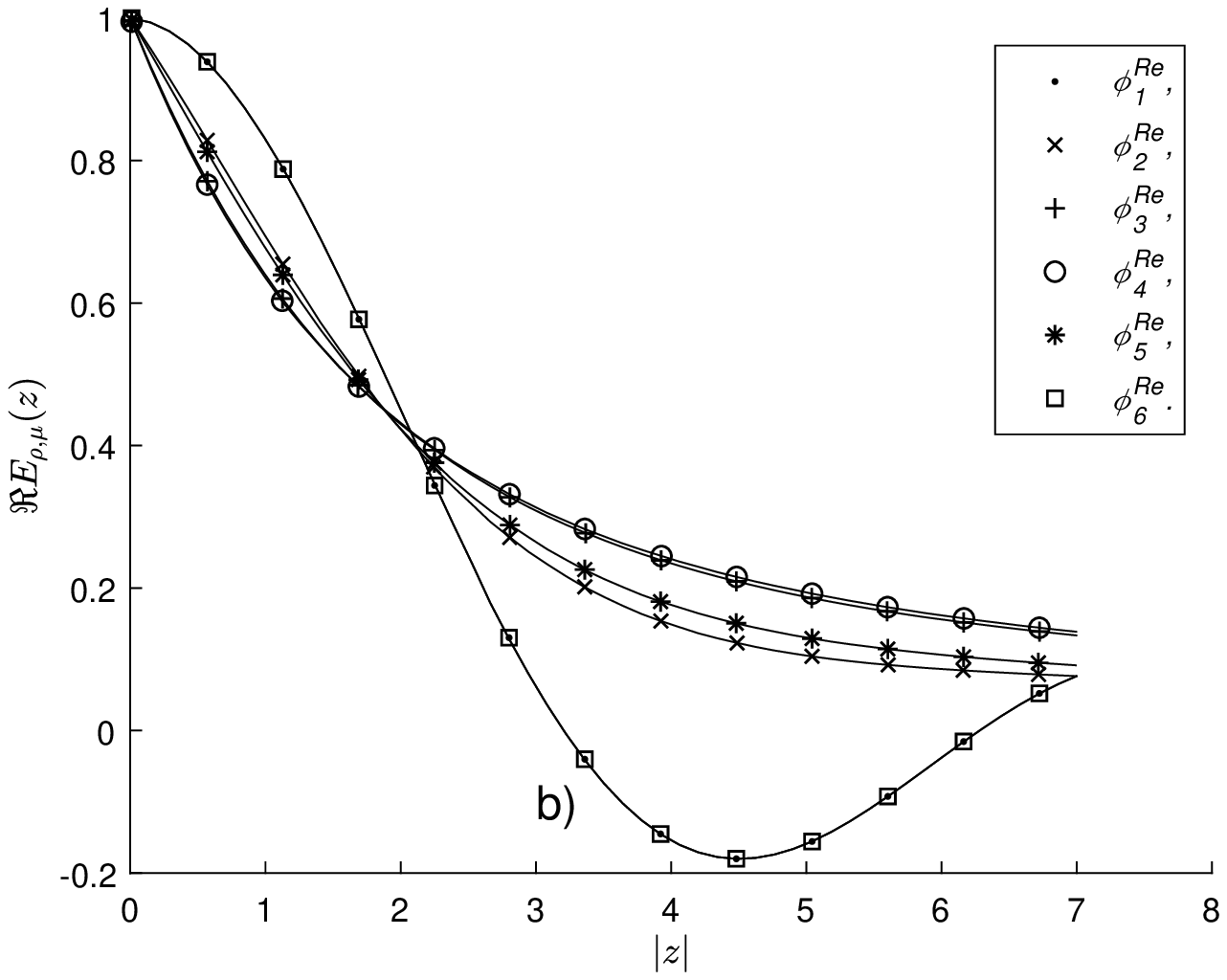}\hfill
  \includegraphics[width=0.43\textwidth]{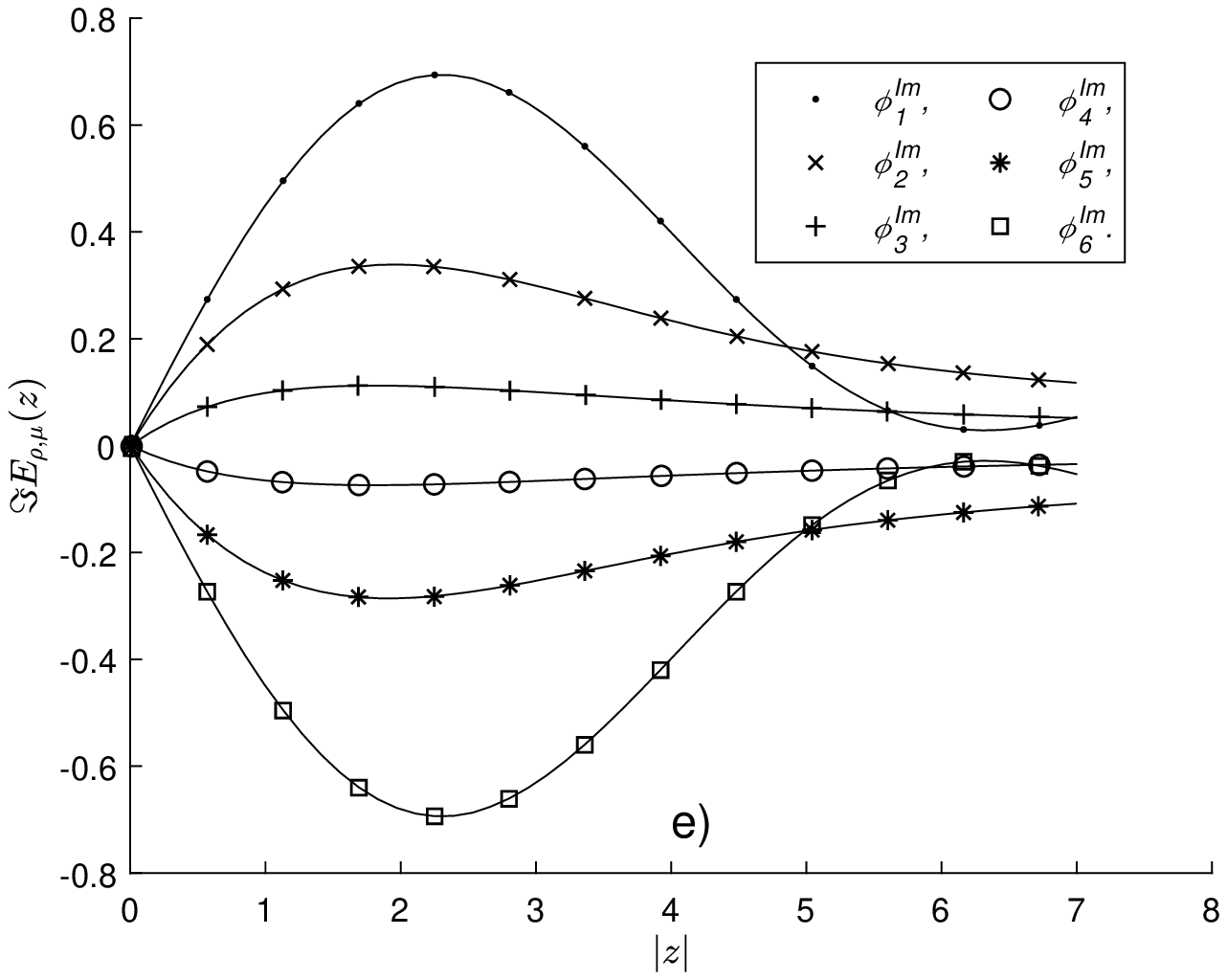}\\[4mm]
  \includegraphics[width=0.43\textwidth]{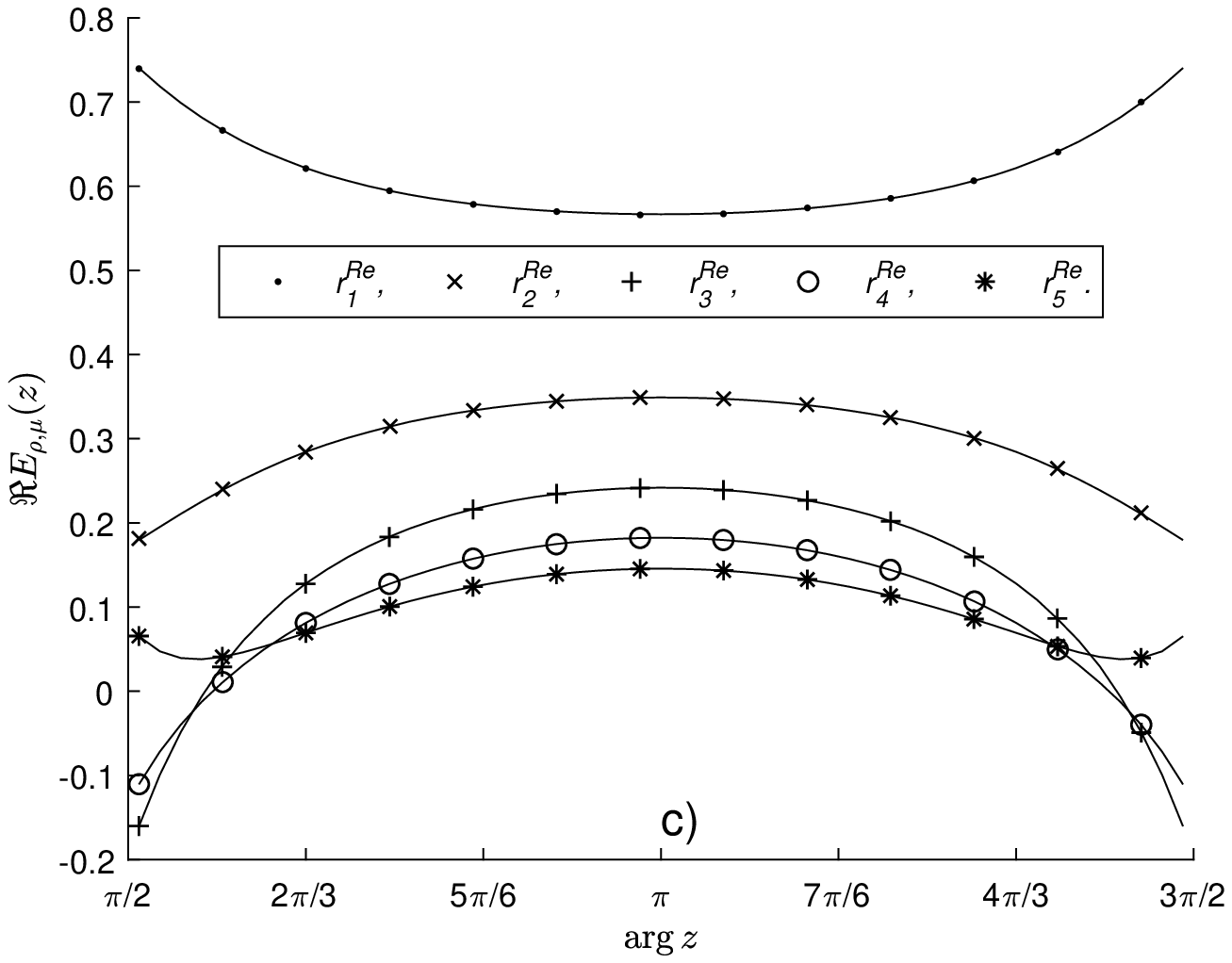}\hfill
  \includegraphics[width=0.43\textwidth]{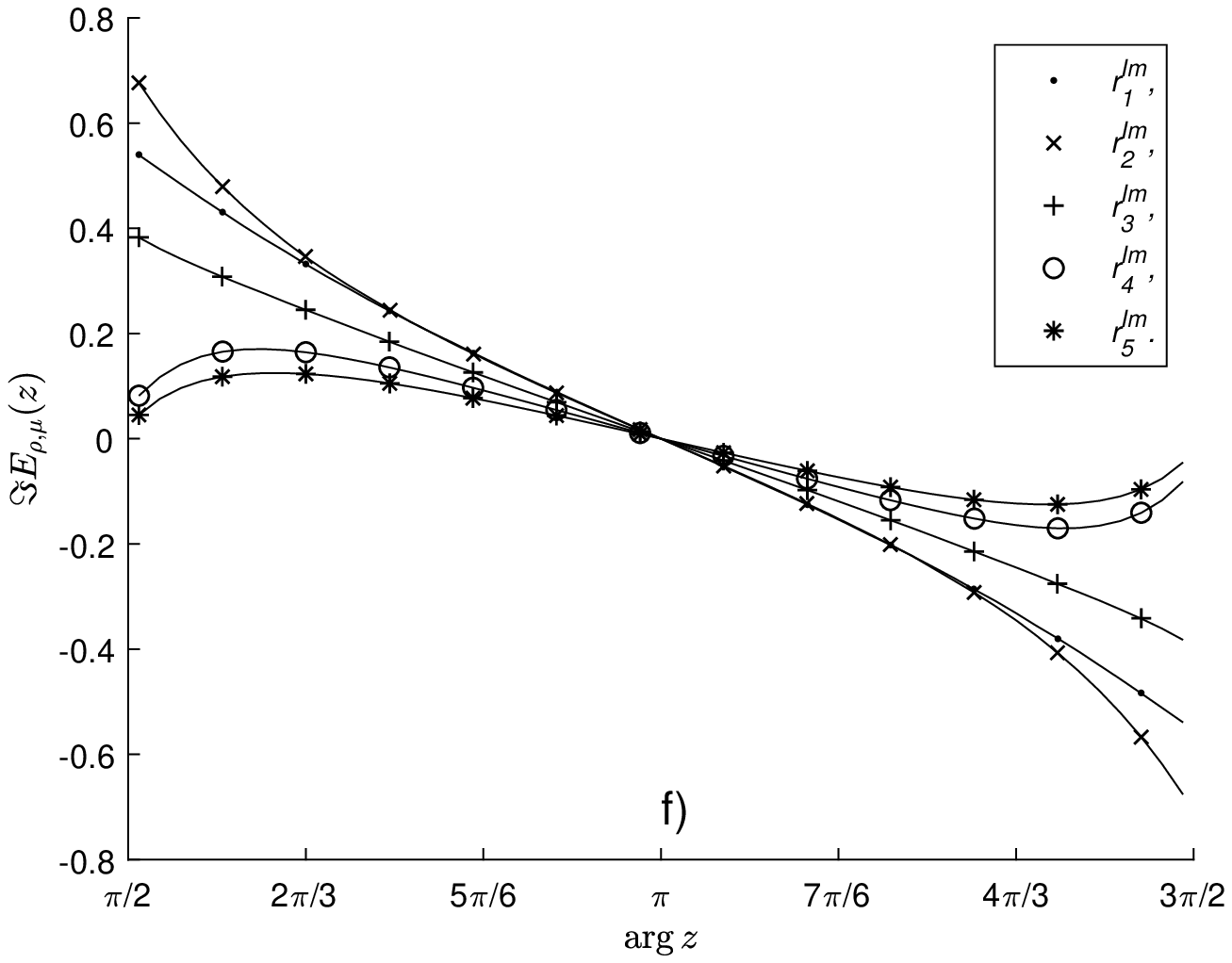}
  \caption{The function $E_{\rho,\mu}(z)$ for $\rho=1, \mu=2,\delta_\rho=\pi$ and $0.01\leqslant|z|\leqslant7, \pi/2<\arg z<3\pi/2$. On the figures a) and d) the surfaces - the formula (\ref{eq:MLF_int2_deltaRho}), the curves - the formula (\ref{eq:MLF_mu>=2}). On the figures b), c), e) and f) the curves - the formula (\ref{eq:MLF_int2_deltaRho}), the points - the formula (\ref{eq:MLF_mu>=2})
  }\label{fig:MLF_int2_deltaRho_rho1_mu2_AB}
\end{figure}

\begin{figure}
  \centering
  \includegraphics[width=0.43\textwidth]{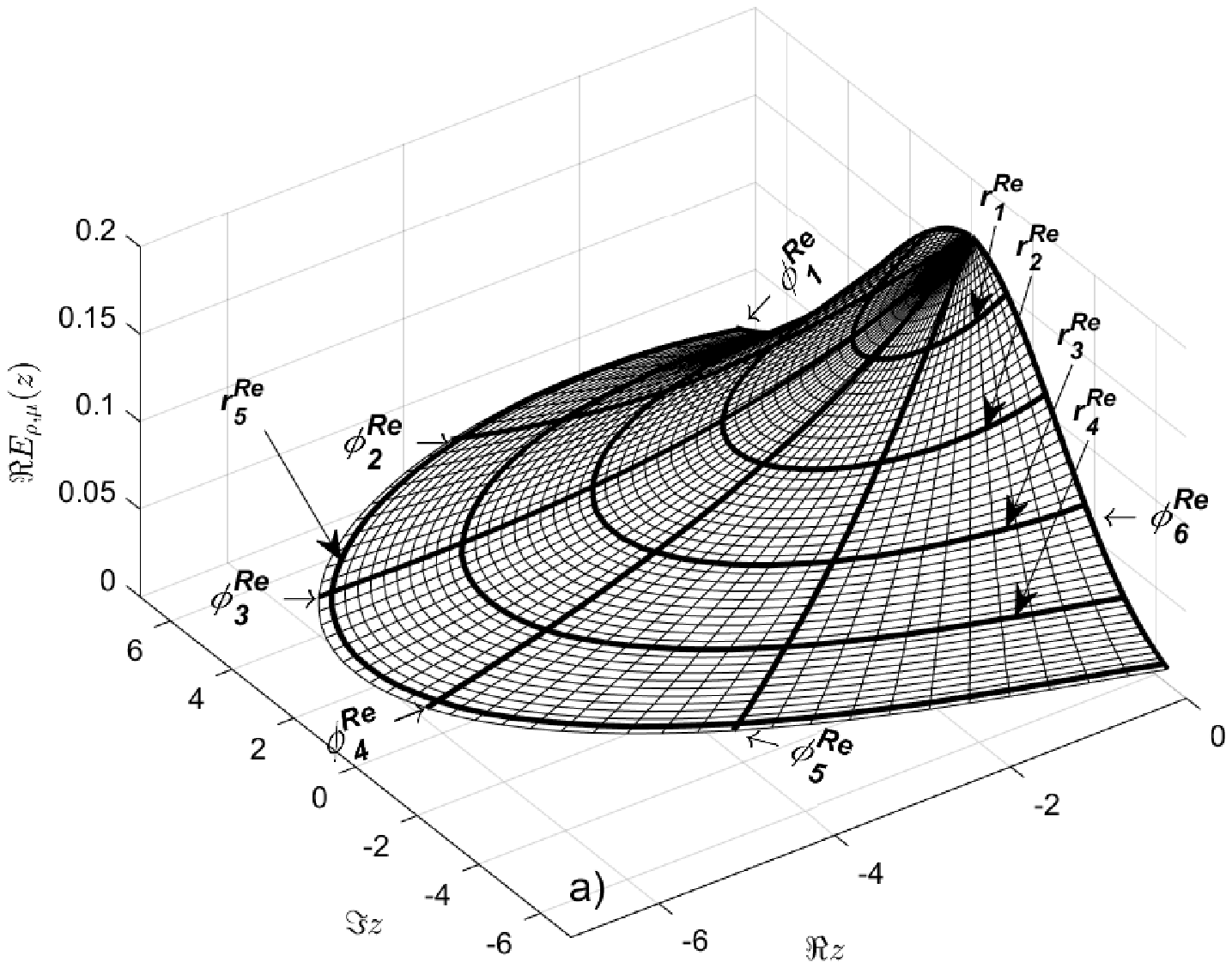}\hfill
  \includegraphics[width=0.43\textwidth]{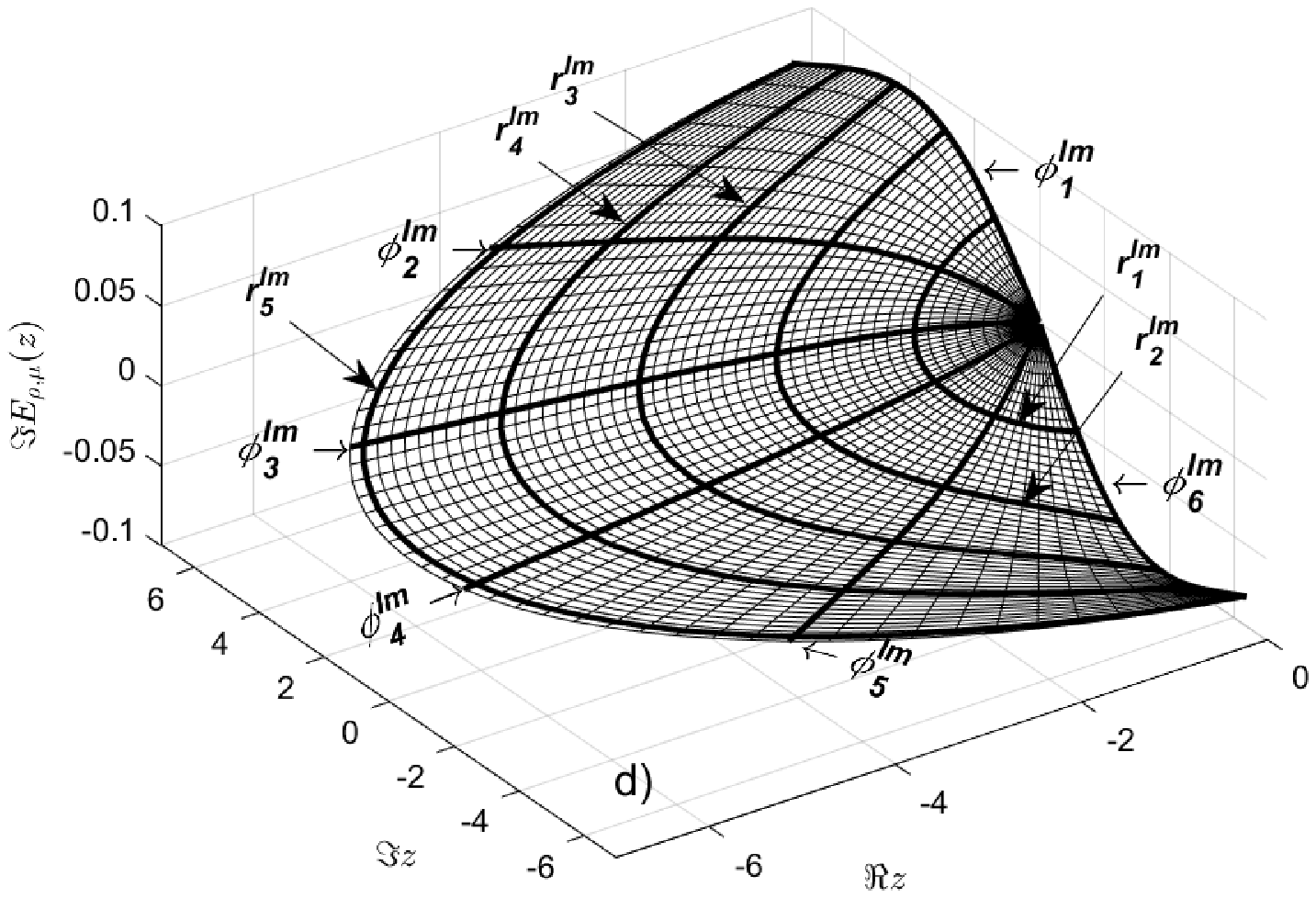}\\[4mm]
  \includegraphics[width=0.43\textwidth]{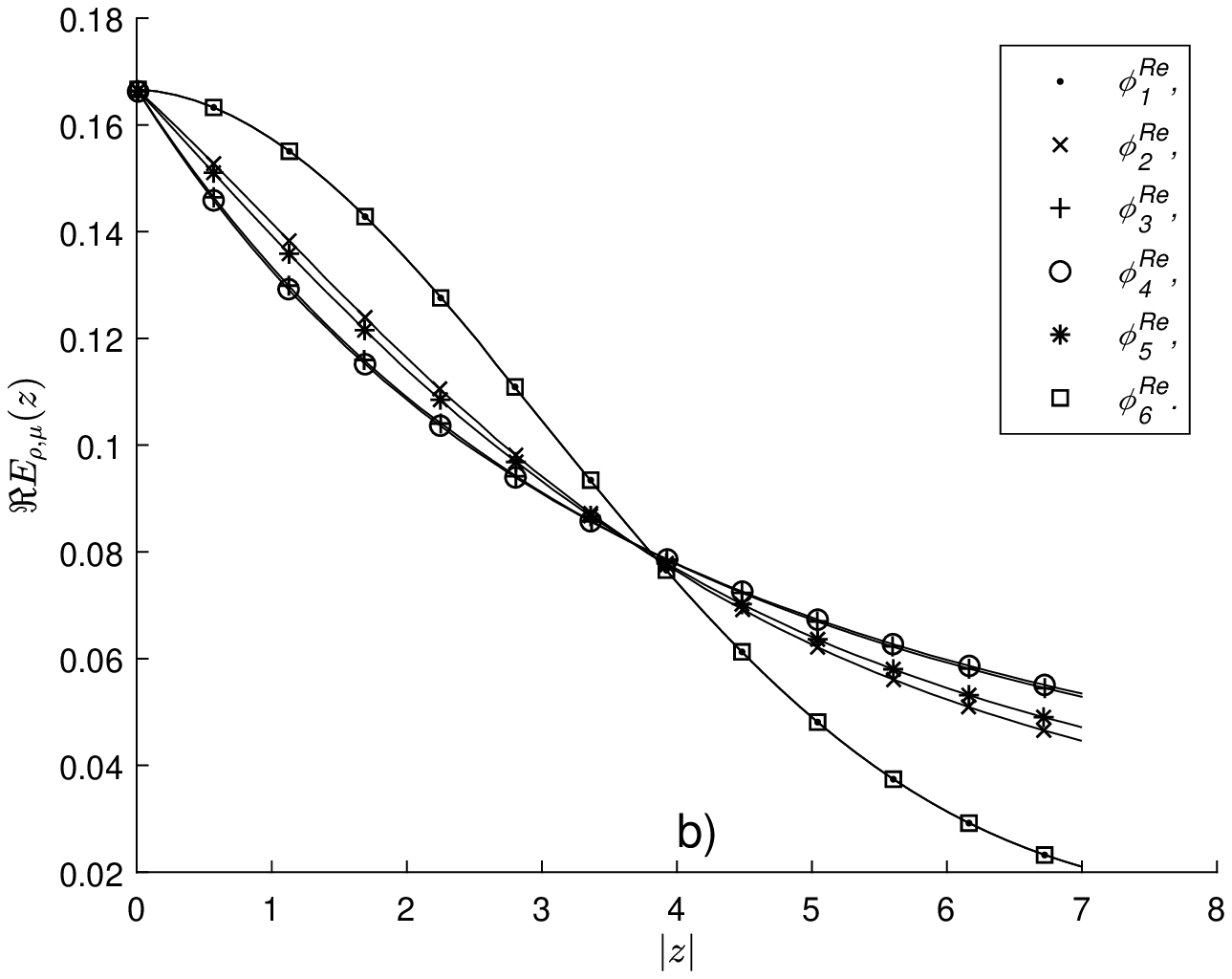}\hfill
  \includegraphics[width=0.43\textwidth]{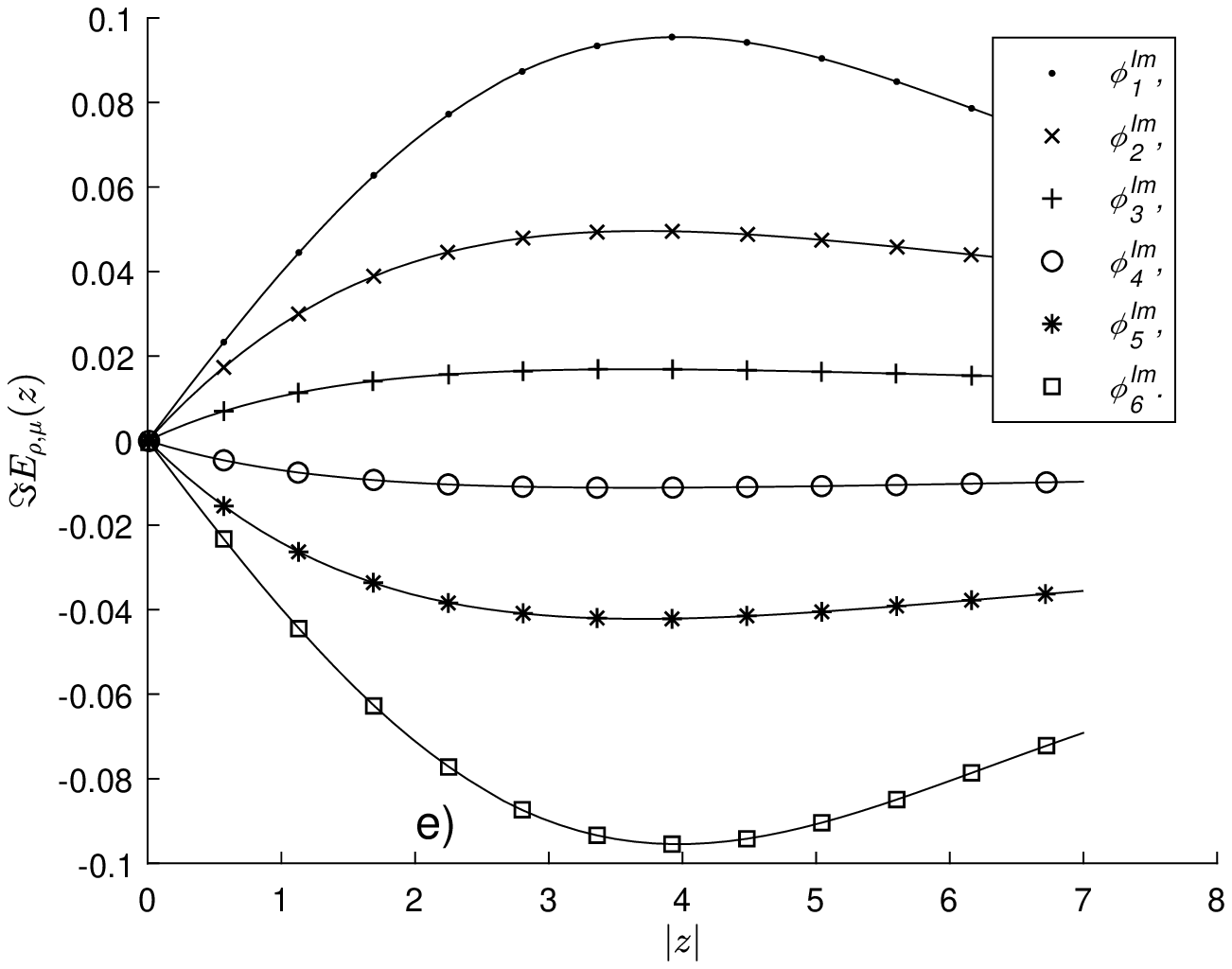}\\[4mm]
  \includegraphics[width=0.43\textwidth]{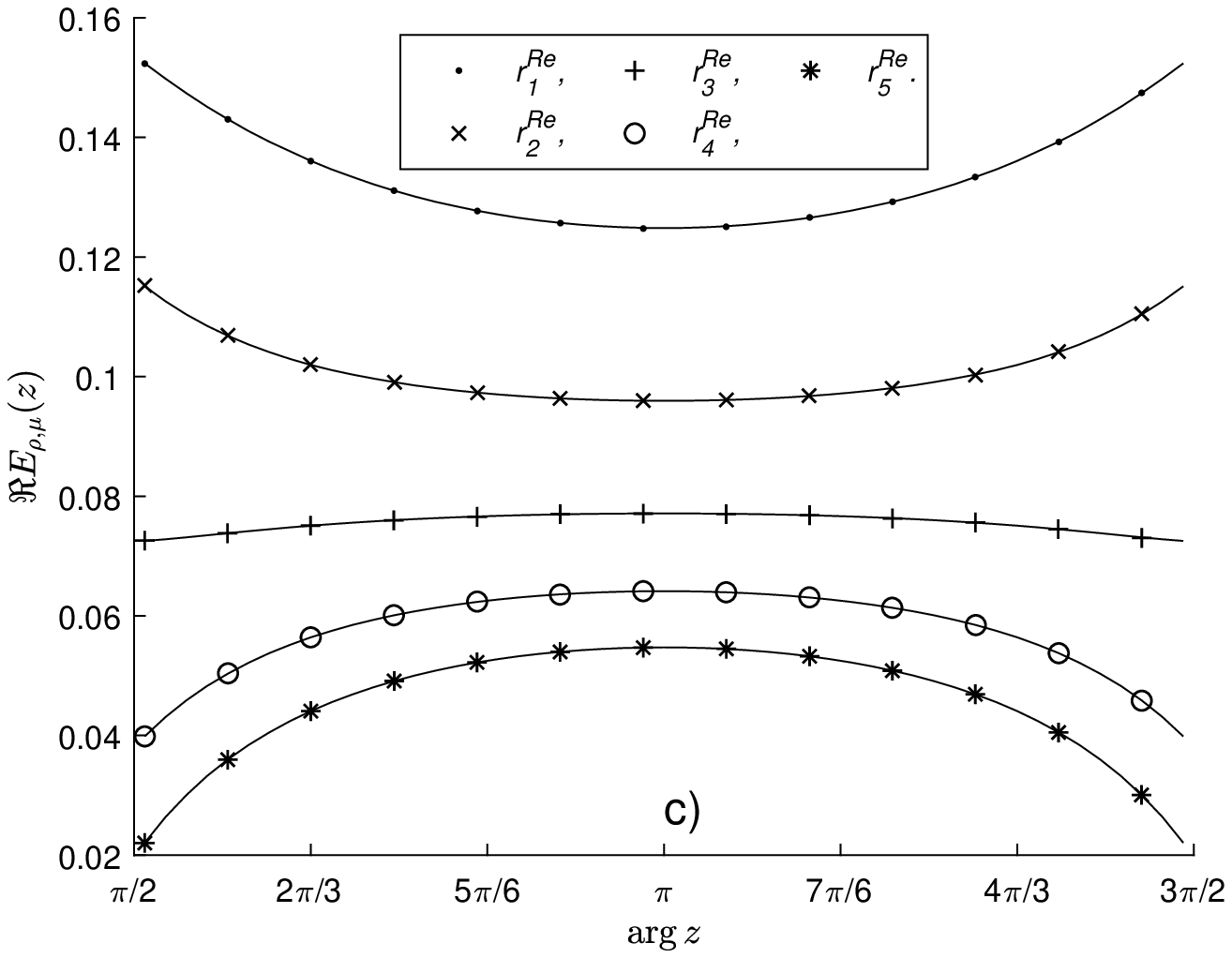}\hfill
  \includegraphics[width=0.43\textwidth]{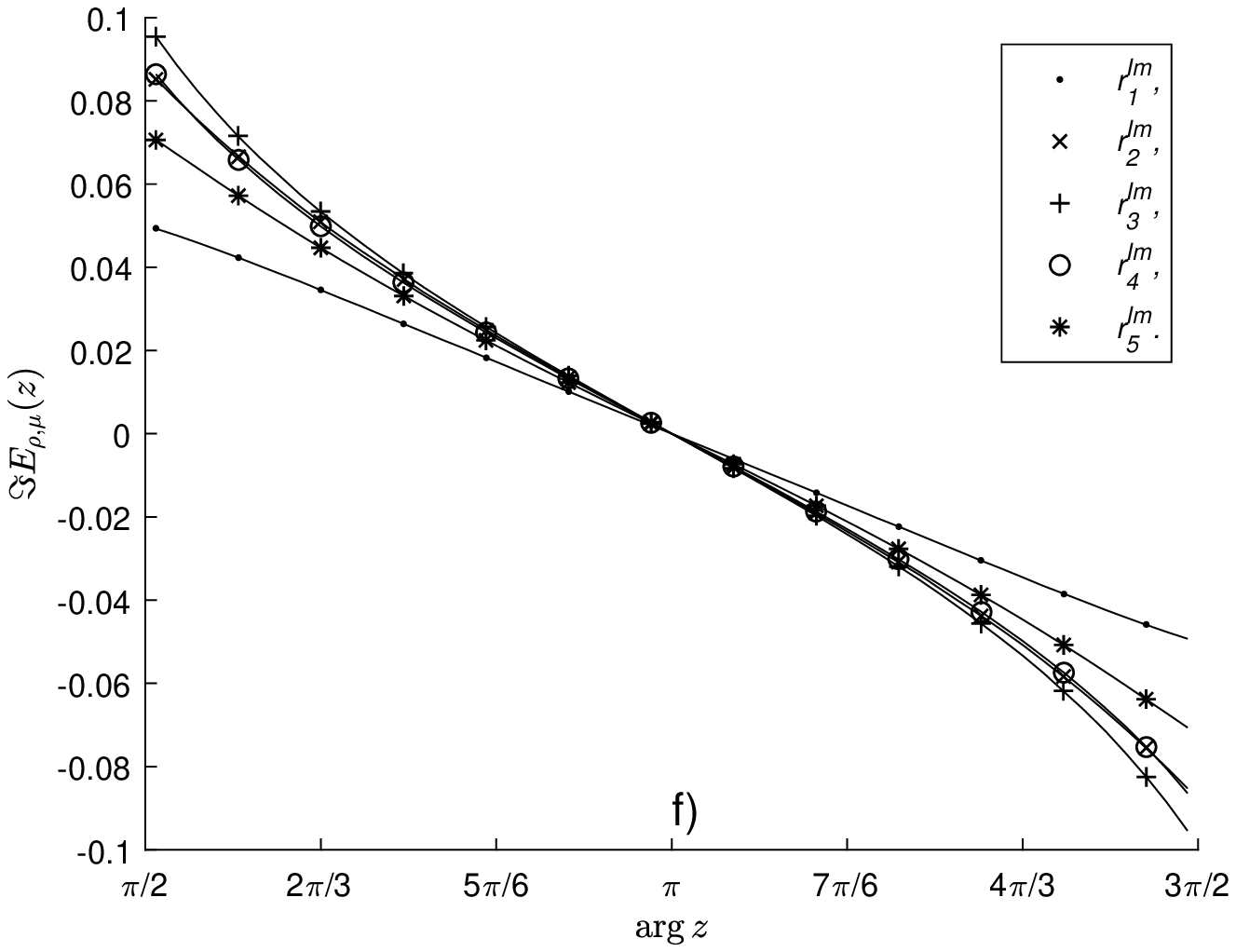}
  \caption{The function $E_{\rho,\mu}(z)$ for $\rho=1, \mu=4$ and $0.01\leqslant|z|\leqslant7, \pi/2<\arg z<3\pi/2$.  On the figures a) and d) the surfaces – the formula (\ref{eq:MLF_int2_piRho}), the curves – the formula (\ref{eq:MLF_mu>=2}). On the figures b), c), e) and f) the curves – the formula (\ref{eq:MLF_int2_piRho}), the points – the formula (\ref{eq:MLF_mu>=2})
  }\label{fig:MLF_int2_piRho_rho1_mu4_AC}
\end{figure}

Consider now the integral representation “B”. This integral representation  was obtained in the article \cite{Saenko2020d} (see Theorem~3 in \cite{Saenko2020d} for Parameterization~1 and Corollary~3 in \cite{Saenko2020d} for Parameterizations~2~and~3). As a result, for the integral representation “B” of the Mittag-Leffler function and Parameterizations~1,~2,~3 the following theorem turns out to be valid

\begin{theorem}\label{lemm:MLF_int3}
For any real $\rho>1/2$, any complex $\mu=\mu_R+i\mu_I$ satisfying the condition $\mu_R<1+\tfrac{1}{\rho}$
\begin{enumerate}
  \item at any real $\delta_{1\rho}$ and $\delta_{2\rho}$ satisfying the conditions
\begin{equation}\label{eq:deltaRho_cond_corol_int3}
\begin{array}{ll}
  \frac{\pi}{2\rho}<\delta_{1\rho}\leqslant\frac{\pi}{\rho},\quad \frac{\pi}{2\rho}<\delta_{2\rho}\leqslant\frac{\pi}{\rho}, &\quad \mbox{if}\quad\rho>1,  \\
  \frac{\pi}{2\rho}<\delta_{1\rho}<\pi,\quad \frac{\pi}{2\rho}<\delta_{2\rho}<\pi, &\quad \mbox{if}\quad 1/2<\rho\leqslant1
  \end{array}
\end{equation}
and any complex $z=te^{i\theta}$ satisfying the condition
\begin{equation}\label{eq:argZ_cond_corol_int3}
  \pi/(2\rho)-\delta_{2\rho}+\pi<\theta<-\pi/(2\rho)+\delta_{1\rho}+\pi
\end{equation}
the Mittag-Leffler function can be represented in the form
\begin{equation}\label{eq:MLF_int3}
E_{\rho,\mu}(z)=\int_{0}^{\infty}K_{\rho,\mu}^{Re}(r, -\delta_{1\rho},\delta_{2\rho},t,\theta) dr
+i \int_{0}^{\infty}K_{\rho,\mu}^{Im}(r,-\delta_{1\rho},\delta_{2\rho},t,\theta) dr,
\end{equation}
where $K_{\rho,\mu}^{Re}(r,\psi_1,\psi_2,t,\theta)$ and $K_{\rho,\mu}^{Im}(r,\psi_1,\psi_2,t,\theta)$ are defined by the expressions  (\ref{eq:K_Re_lemm_int2}) and (\ref{eq:K_Im_lemm_int2}) respectively.

\item If $1/2<\rho\leqslant1$, $\delta_{1\rho}=\pi$, $\pi/(2\rho)<\delta_{2\rho}<\pi$, then for any complex $z=t e^{i\theta}$, satisfying the condition $\frac{\pi}{2\rho}-\delta_{2\rho}+\pi<\theta<-\frac{\pi}{2\rho}+2\pi$,  the Mittag-Leffler function can be represented in the form
    \begin{multline}\label{eq:MLF_int3_case2}
      E_{\rho,\mu}(z)= \int_{0}^{\infty}K_{\rho,\mu}^{\prime Re}(r,\delta_{2\rho},t,\theta)dr
      -\int_{0}^{1-\varepsilon_1}K_{\rho,\mu}^{\prime Re}(r,-\pi,t,\theta)dr\\
      -\int_{1+\varepsilon_1}^{\infty}K_{\rho,\mu}^{\prime Re}(r,-\pi,t,\theta)dr
      + \int_{-2\pi}^{-\pi} P_{\rho,\mu}^{\prime Re}(\varepsilon_1,\psi,-2,t,\theta)d\psi\\
      + i\left\{\int_{0}^{\infty}K_{\rho,\mu}^{\prime Im}(r,\delta_{2\rho},t,\theta)dr -
      \int_{0}^{1-\varepsilon_1}K_{\rho,\mu}^{\prime Im}(r,-\pi,t,\theta)dr\right.\\
      -\left. \int_{1+\varepsilon_1}^{\infty}K_{\rho,\mu}^{\prime Im}(r,-\pi,t,\theta)dr
      + \int_{-2\pi}^{-\pi} P_{\rho,\mu}^{\prime Im}(\varepsilon_1,\psi,-2,t,\theta)d\psi\right\}.
    \end{multline}
Here $\varepsilon_1$ is an arbitrary real number satisfying the condition $0<\varepsilon_1<1$.
\begin{align}
  K_{\rho,\mu}^{\prime Re}(r,\delta,t,\theta) =& \frac{\rho}{2\pi} \frac{f(r,\delta-\pi,t,\theta) (rt)^{\rho(1-\mu_R) }}{r^2+2r\cos\delta+1}\nonumber\\
  \times&(r\sin(\xi(r,\delta-\pi,t,\theta))+\sin(\xi(r,\delta-\pi,t,\theta)+\delta)), \label{eq:MLF_int3_K'Re}\\
  K_{\rho,\mu}^{\prime Im}(r,\delta,t,\theta) =& -\frac{\rho}{2\pi} \frac{f(r,\delta-\pi,t,\theta) (rt)^{\rho(1-\mu_R) }}{r^2+2r\cos\delta+1}\nonumber\\
   \times&(r\cos(\xi(r,\delta-\pi,t,\theta))+\cos(\xi(r,\delta-\pi,t,\theta)+\delta)), \label{eq:MLF_int3_K'Im}
\end{align}
where the function $f(r,\varphi,t,\theta)$ has the form (\ref{eq:f_int2}), and the function $\xi(r,\varphi,t,\theta)$ has the form (\ref{eq:xi_int2}).

\begin{multline}\label{eq:MLF_int3_P'Re}
  P_{\rho,\mu}^{\prime Re}(\tau,\psi,k,t,\theta) = \frac{\rho\tau}{2\pi} \frac{(t r(\tau,\psi))^{\rho(1-\mu_R)} f'(\tau,\psi,k,t,\theta)}{(r(\tau,\psi))^2-2r(\tau,\psi)\cos(\varphi(\tau,\psi,k))+1}\\
  \times[r(\tau,\psi)\cos(\xi'(\tau,\psi,k,t,\theta)-\varphi(\tau,\psi,k))-\cos(\xi'(\tau,\psi,k,t,\theta))],
  \end{multline}
  \begin{multline}\label{eq:MLF_int3_P'Im}
  P_{\rho,\mu}^{\prime Im}(\tau,\psi,k,t,\theta) = \frac{\rho\tau}{2\pi} \frac{(t r(\tau,\psi))^{\rho(1-\mu_R)} f'(\tau,\psi,k,t,\theta)}{(r(\tau,\psi))^2-2r(\tau,\psi)\cos(\varphi(\tau,\psi,k))+1}\\
  \times[r(\tau,\psi)\sin(\xi'(\tau,\psi,k,t,\theta)-\varphi(\tau,\psi,k))-\sin(\xi'(\tau,\psi,k,t,\theta))],
\end{multline}
where $r(\tau,\psi)=\sqrt{\tau^2+2\tau\cos\psi+1}$, $\varphi(\tau,\psi,k)=\arctan\left(\frac{\tau\sin\psi}{\tau\cos\psi+1}\right)+k\pi$ and
\begin{align}
  f'(\tau,\psi,k,t,\theta) &= \exp\left\{(tr(\tau,\psi))^\rho\cos(\rho(\theta+\varphi(\tau,\psi,k)))
  + \rho\mu_I(\theta+\varphi(\tau,\psi,k))\right\},\label{eq:MLF_int3_f'}\\
  \xi'(\tau,\psi,k,t,\theta)&= (t r(\tau,\psi))^\rho\sin(\rho(\theta+\varphi(\tau,\psi,k)))\nonumber\\
  &- \rho\mu_I\ln(r(\tau,\psi)t) +\rho(1-\mu_R)(\theta+\varphi(\tau,\psi,k))+\psi.\label{eq:MLF_int3_xi'}
\end{align}

\item If $1/2<\rho\leqslant1$, $\frac{\pi}{2\rho}<\delta_{1\rho}<\pi$, $\delta_{2\rho}=\pi$, then for any complex $z=t e^{i\theta}$, satisfying the condition $\frac{\pi}{2\rho}<\theta<-\frac{\pi}{2\rho}+\delta_{1\rho}+\pi$, the Mittag-Leffler function can be represented in the form
    \begin{multline}\label{eq:MLF_int3_case3}
      E_{\rho,\mu}(z)=\int_{0}^{1-\varepsilon_1}K_{\rho,\mu}^{\prime Re}(r,\pi,t,\theta) dr
      + \int_{-\pi}^{0} P_{\rho,\mu}^{\prime Re}(\varepsilon_1,\psi,0,t,\theta) d\psi\\
      + \int_{1+\varepsilon_1}^{\infty}K_{\rho,\mu}^{\prime Re}(r,\pi,t,\theta) dr
      - \int_{0}^{\infty} K_{\rho,\mu}^{\prime Re}(r,-\delta_{1\rho},t,\theta) dr\\
      +i\left\{ \int_{0}^{1-\varepsilon_1}K_{\rho,\mu}^{\prime Im}(r,\pi,t,\theta) dr
      + \int_{-\pi}^{0} P_{\rho,\mu}^{\prime Im}(\varepsilon_1,\psi,0,t,\theta) d\psi\right.\\
     \left. + \int_{1+\varepsilon_1}^{\infty}K_{\rho,\mu}^{\prime Im}(r,\pi,t,\theta) dr
      - \int_{0}^{\infty} K_{\rho,\mu}^{\prime Im}(r,-\delta_{1\rho},t,\theta) dr\right\}.
    \end{multline}

\item If $1/2<\rho\leqslant1$ and $\delta_{1\rho}=\delta_{2\rho}=\pi$, then for any complex $z=t e^{i\theta}$ satisfying the condition $\frac{\pi}{2\rho}<\theta<-\frac{\pi}{2\rho}+2\pi$, the Mittag-Leffler function can be represented in the form
\begin{multline}\label{eq:MLF_int3_case4}
  E_{\rho,\mu}(z)= \int_{0}^{1-\varepsilon_1}K_{\rho,\mu}^{Re}(r,t,\theta) dr
  + \int_{-\pi}^{0} P_{\rho,\mu}^{\prime Re}(\varepsilon_1,\psi,0,t,\theta) d\psi\\
  +\int_{-2\pi}^{-\pi} P_{\rho,\mu}^{\prime Re}(\varepsilon_1,\psi,-2,t,\theta) d\psi
  + \int_{1+\varepsilon_1}^{\infty}K_{\rho,\mu}^{Re}(r,t,\theta) dr\\
  +i\left\{ \int_{0}^{1-\varepsilon_1}K_{\rho,\mu}^{Im}(r,t,\theta) dr
  + \int_{-\pi}^{0} P_{\rho,\mu}^{\prime Im}(\varepsilon_1,\psi,0,t,\theta) d\psi\right.\\
  \left.+\int_{-2\pi}^{-\pi} P_{\rho,\mu}^{\prime Im}(\varepsilon_1,\psi,-2,t,\theta) d\psi
  + \int_{1+\varepsilon_1}^{\infty}K_{\rho,\mu}^{Im}(r,t,\theta) dr\right\},
\end{multline}
where $K_{\rho,\mu}^{Re}(r,t,\theta)$ and $K_{\rho,\mu}^{Im}(r,t,\theta)$ are defined by the expressions  (\ref{eq:K_Re_piRho_corol_int2}) and (\ref{eq:K_Im_piRho_corol_int2}) respectively.

\item If $\delta_{1\rho}=\delta_{2\rho}=\delta_\rho$, then for any $\delta_\rho$ satisfying the conditions $\frac{\pi}{2\rho}<\delta_\rho\leqslant\frac{\pi}{\rho}$, if $\rho>1$  and  $\frac{\pi}{2\rho}<\delta_\rho<\pi$, if $1/2<\rho\leqslant1$  and any complex $z=t e^{i\theta}$ satisfying the condition $\frac{\pi}{2\rho}-\delta_{\rho}+\pi<\theta<-\frac{\pi}{2\rho}+\delta_{\rho}+\pi$, the Mittag-Leffler function can be represented in the form
\begin{equation}\label{eq:MLF_int3_deltaRho}
E_{\rho,\mu}(z)=\int_{0}^{\infty}K_{\rho,\mu}^{Re}(r, \delta_{\rho},t,\theta) dr
+i \int_{0}^{\infty}K_{\rho,\mu}^{Im}(r,\delta_{\rho},t,\theta) dr,
\end{equation}
where $K_{\rho,\mu}^{Re}(r, \delta_{\rho},t,\theta)$ and $K_{\rho,\mu}^{Im}(r, \delta_{\rho},t,\theta)$ are defined by the expressions (\ref{eq:K_Re_deltaRho_corol_int2}) and (\ref{eq:K_Im_deltaRho_corol_int2}) respectively.

\item If $\delta_\rho=\pi/\rho$ and $\rho>1$ then for any complex $z=te^{i\theta}$ satisfying the condition $-\frac{\pi}{2\rho}+\pi<\theta<\frac{\pi}{2\rho}+\pi$, the Mittag-Leffler function can be represented in the form
\begin{equation}\label{eq:MLF_int3_piRho}
E_{\rho,\mu}(z)=\int_{0}^{\infty}K_{\rho,\mu}^{Re}(r,t,\theta) dr
+i \int_{0}^{\infty}K_{\rho,\mu}^{Im}(r,t,\theta) dr,
\end{equation}
where $K_{\rho,\mu}^{Re}(r,t,\theta)$ and $K_{\rho,\mu}^{Im}(r,t,\theta)$ are defined by the expressions (\ref{eq:K_Re_piRho_corol_int2}) and (\ref{eq:K_Im_piRho_corol_int2}) respectively.
\end{enumerate}
\end{theorem}

\begin{proof}
\emph{Case 1.} The integral representation (\ref{eq:MLF_int3}) directly follows  from  Theorem~3 and Corollary~3 formulated in the article \cite{Saenko2020d}.  In fact,  in the article \cite{Saenko2020d} (see Theorem~3 in \cite{Saenko2020d})  it was shown that when the conditions are met
\begin{equation*}
\begin{array}{ll}
  \frac{\pi}{2\rho}<\delta_{1\rho}\leqslant\frac{\pi}{\rho},\quad \frac{\pi}{2\rho}<\delta_{2\rho}\leqslant\frac{\pi}{\rho}, &\quad \mbox{if}\quad\rho>1,  \\
  \frac{\pi}{2\rho}<\delta_{1\rho}<\pi,\quad \frac{\pi}{2\rho}<\delta_{2\rho}<\pi, &\quad \mbox{if}\quad 1/2<\rho\leqslant1
  \end{array}
\end{equation*}
and
\begin{equation}\label{eq:MLF_int3_argZCond_tmp}
  \tfrac{\pi}{2\rho}-\delta_{2\rho}+\pi<\arg z<-\tfrac{\pi}{2\rho}+\delta_{1\rho}+\pi
\end{equation}
for the Mittag-Leffler function the representation  is true
\begin{equation}\label{eq:MLF_int3_tmp}
  E_{\rho,\mu}(z)=\int_{0}^{\infty}K_{\rho,\mu}(r,-\delta_{1\rho},\delta_{2\rho},z)dr,
\end{equation}
where $K_{\rho,\mu}(r,\varphi_1,\varphi_2,z)$ is defined by the expression  (8) in \cite{Saenko2020d}. To obtain the representation (\ref{eq:MLF_int3}) we write the complex number  $z$ in the form $z=t e^{i\theta}$, and the complex  parameter  $\mu$ in the form $\mu=\mu_R+i\mu_i$. As a result, the condition  (\ref{eq:MLF_int3_argZCond_tmp}) takes the form
$ \tfrac{\pi}{2\rho}-\delta_{2\rho}+\pi<\theta<-\tfrac{\pi}{2\rho}+\delta_{1\rho}+\pi$.

Next, we put these representations in the integrand  $K_{\rho,\mu}(r,\varphi_1,\varphi_2,z)$ and transform it. It should be noted that it has already been done when proving Theorem~\ref{lemm:MLF_int2}, where it was found that the kernel $K_{\rho,\mu}(r,\varphi_1,\varphi_2,z)$ breaks down into the sum of real and imaginary parts of this kernel (see~(\ref{eq:KRe+KIm_lemm_int2}))
\begin{equation}\label{eq:KRe+KIm_corol_int3}
  K_{\rho,\mu}(r,\varphi_1,\varphi_2,z)=K_{\rho,\mu}^{Re}(r,\varphi_1,\varphi_2,t,\theta) + i K_{\rho,\mu}^{Im}(r,\varphi_1,\varphi_2,t,\theta),
\end{equation}
where $K_{\rho,\mu}^{Re}(r,\varphi_1,\varphi_2,t,\theta)$ and $K_{\rho,\mu}^{Im}(r,\varphi_1,\varphi_2,t,\theta)$ are defined by the expressions (\ref{eq:K_Re_lemm_int2}) and (\ref{eq:K_Im_lemm_int2}) respectively.  Now by putting (\ref{eq:KRe+KIm_corol_int3}) in (\ref{eq:MLF_int3_tmp}) we get the representation (\ref{eq:MLF_int3}). The first part of the lemma is proved.

\emph{Case 2.} Now we consider the case $1/2<\rho\leqslant1, \delta_{1\rho}=\pi, \frac{\pi}{2\rho}<\delta_{2\rho}<\pi$.
According to  the article \cite{Saenko2020d}, in the case considered the Mittag-Leffler function has the form (see Theorem~3 in~\cite{Saenko2020d})
\begin{multline}\label{eq:MLF_int3_MLF_case2_tmp}
    E_{\rho,\mu}(z)=\int_{0}^{\infty}K_{\rho,\mu}^\prime(r,\delta_{2\rho},z)dr- \int_{0}^{1-\varepsilon_1}K_{\rho,\mu}^\prime(r,-\pi,z)dr -\\
    \int_{1+\varepsilon_1}^{\infty}K_{\rho,\mu}^\prime(r,-\pi,z)dr+
    \int_{-2\pi}^{-\pi}P_{\rho,\mu}^\prime(\varepsilon_1,\psi,-2,z)d\psi.
\end{multline}

  The representations for the functions $K'_{\rho,\mu}(r,\delta,z)$ and $P'_{\rho,\mu}(r,\psi,k,z)$  introduced in the paper \cite{Saenko2020d} when proving item~2 of Theorem~3 \cite{Saenko2020d} are more convenient for further proof. At the beginning we consider the function $K'_{\rho,\mu}(r,\delta,z)$. This function can be defined by following (see~(56) in~\cite{Saenko2020d})
\begin{equation}\label{eq:lemm_MLF_int1_K'_def}
  K^\prime_{\rho,\mu}(r,\delta,z)=\frac{\rho}{2\pi i}\frac{\exp\left\{\left(zre^{i(\delta-\pi)^\rho}\right)\right\} \left(zre^{i(\delta-\pi)}\right)^{\rho(1-\mu)}}{re^{i(\delta-\pi)}-1}e^{i(\delta-\pi)}.
\end{equation}
We will introduce for the function $K^\prime_{\rho,\mu}(r,\delta,z)$ the following notation $K'_{\rho,\mu}(r,\delta,t,\theta)\equiv K'_{\rho,\mu}(r,\delta,z)$ which is going to be used subsequently. We represent the complex number $z$ in the form $z=te^{i\theta}$ and $\mu=\mu_R+i\mu_I$ and in view of this we transform the function  $K'_{\rho,\mu}(r,\delta,z)$. Getting rid of the complexity in the denominator  (\ref{eq:lemm_MLF_int1_K'_def})  and making some transformations we get
\begin{multline}\label{eq:MLF_int3_K'_sum}
  K'_{\rho,\mu}(r,\delta,t,\theta)
  = \frac{\rho}{2\pi i} \frac{\exp\left\{\left((rt) e^{i(\delta-\pi+\theta)}\right)^\rho \right\}}{r^2+2r\cos\delta+1}
   \left((rt) e^{i(\delta-\pi+\theta)}\right)^{\rho(1-\mu_R-i\mu_I)} \left(r+e^{i\delta}\right)\\
  =\frac{\rho}{2\pi} \frac{(rt)^{\rho(1-\mu_R)} f(r,\delta-\pi,t,\theta)}{r^2+2r\cos\delta+1} \left(r\sin(\xi(r,\delta-\pi,t,\theta))+\sin(\xi(r,\delta-\pi,t,\theta)+\delta))\right.\\
  \left.- i\left\{r\cos(\xi(r,\delta-\pi,t,\theta))+\cos(\xi(r,\delta-\pi,t,\theta)+\delta))\right\}\right)\\
  = K_{\rho,\mu}^{\prime Re}(r,\delta,t,\theta) +i K_{\rho,\mu}^{\prime Im}(r,\delta,t,\theta),
\end{multline}
where the functions $f(r,\varphi,t,\theta)$ and $\xi(r,\varphi,t,\theta)$ have the form (\ref{eq:f_int2}) and (\ref{eq:xi_int2}) and $K_{\rho,\mu}^{\prime Re}(r,\delta,t,\theta)$ and $K_{\rho,\mu}^{\prime Im}(r,\delta,t,\theta)$ have the form (\ref{eq:MLF_int3_K'Re}) and (\ref{eq:MLF_int3_K'Im}) respectively.

Now consider the integrand $P'_{\rho,\mu}(r,\psi,k,z)$. The representation for this function introduced in  \cite{Saenko2020d} turns out to be more convenient when proving item~2 of Theorem~3 (see~(65) in~\cite{Saenko2020d})
\begin{equation}\label{eq:P'_fun}
  P_{\rho,\mu}^\prime(\tau,\psi,k,z)=i\tau\frac{\phi_{\rho,\mu}\left(r(\tau,\psi) e^{i\varphi(\tau,\psi,k)},z\right)} {r(\tau,\psi) e^{i\varphi(\tau,\psi,k)}-1} e^{i\psi},
\end{equation}
where
$\phi_{\rho,\mu}(\zeta,z)=\frac{\rho}{2\pi i}\exp\left\{(\zeta z)^\rho\right\}(\zeta z)^{\rho(1-\mu)}$,
$r(\tau,\psi)  =\sqrt{\tau^2+2\tau\cos\psi+1}$,
$\varphi(\tau,\psi,k)=\arctan\left(\frac{\tau\sin\psi}{\tau\cos\psi +1}\right)+k\pi$.

We represent in the expression (\ref{eq:P'_fun}) the complex numbers $z$ and $\mu$ in the form $z=te^{i\theta}$, $\mu=\mu_R+i\mu_I$. Then, we will get rid of the complexity in the denominator and use the definitions of the functions $r(\tau,\psi)$, $\varphi(\tau,\psi,k)$ and $\phi_{\rho,\mu}(\zeta,z)$. Making some transformations, we obtain
\begin{multline}\label{eq:MLF_int3_P'_sum}
  P_{\rho,\mu}^\prime(\tau,\psi,k,z)=i\tau\frac{\phi_{\rho,\mu}\left(r(\tau,\psi) e^{i\varphi(\tau,\psi,k)},te^{i\theta}\right) \left(r(\tau,\psi) e^{-i\varphi(\tau,\psi,k)}-1\right)} {\left(r(\tau,\psi) e^{i\varphi(\tau,\psi,k)}-1\right)\left(r(\tau,\psi) e^{-i\varphi(\tau,\psi,k)}-1\right)}  e^{i\psi}\\
   =\frac{\rho\tau}{2\pi} \frac{(t r(\tau,\psi))^{\rho(1-\mu_R)} f'(\tau,\psi,k,t,\theta)}{(r(\tau,\psi))^2- 2r(\tau,\psi)\cos(\varphi(\tau,\psi,k))+1}\\
   \times \Bigl( r(\tau,\psi)\cos(\xi'(\tau,\psi,k,t,\theta)-\varphi(\tau,\psi,k))
   -\cos(\xi'(\tau,\psi,k,t,\theta))\\
   +i\left\{r(\tau,\psi)\sin(\xi'(\tau,\psi,k,t,\theta)-\varphi(\tau,\psi,k))- \sin(\xi'(\tau,\psi,k,t,\theta))\right\}\Bigr)\\
    = P_{\rho,\mu}^{\prime Re}(\tau,\psi,k,t,\theta) + i P_{\rho,\mu}^{\prime Im}(\tau,\psi,k,t,\theta),
\end{multline}
where the functions $f'(\tau,\psi,k,t,\theta)$ and $\xi'(\tau,\psi,k,\theta)$ have the form (\ref{eq:MLF_int3_f'}) and (\ref{eq:MLF_int3_xi'})  the functions $P_{\rho,\mu}^{\prime Re}(\tau,\psi,k,t,\theta)$ and $P_{\rho,\mu}^{\prime Im}(\tau,\psi,k,t,\theta)$ have the form (\ref{eq:MLF_int3_P'Re}) and (\ref{eq:MLF_int3_P'Im}) respectively.
Now using (\ref{eq:MLF_int3_K'_sum}) and (\ref{eq:MLF_int3_P'_sum}) in (\ref{eq:MLF_int3_MLF_case2_tmp}) we arrive at the representation (\ref{eq:MLF_int3_case2}). Thus, the second point of the lemma is proved.

\emph{Case 3.} We consider now the case $1/2<\rho\leqslant1$, $\frac{\pi}{2\rho}<\delta_{1\rho}<\pi$, $\delta_{2\rho}=\pi$. It was shown in the article \cite{Saenko2020d} that at such values of parameters the function $E_{\rho,\mu}(z)$ can be represented in the form (see theorem~3 case~3 in \cite{Saenko2020d})
\begin{multline}\label{eq:lemm_MLF_int1_case3}
    E_{\rho,\mu}(z)=\int_{0}^{1-\varepsilon_1}K_{\rho,\mu}^\prime(r,\pi,z)dr- \int_{1+\varepsilon_1}^{\infty}K_{\rho,\mu}^\prime(r,\pi,z)dr +\\
    +\int_{-\pi}^{0}P_{\rho,\mu}^\prime(\varepsilon_1,\psi,0,z)d\psi-
    \int_{0}^{\infty}K_{\rho,\mu}^\prime(r,-\delta_{1\rho},z)dr,
  \end{multline}
where $\varepsilon_1$ is an arbitrary real number such that $0<\varepsilon_1<1$.

We will write down the complex numbers  $z$ and $\mu$ in the form$z=te^{i\theta}$, $\mu=\mu_R+i\mu_I$ and substitute  them in  (\ref{eq:lemm_MLF_int1_case3}). When proving the previous item of the theorem it was shown that such  a substitution the functions  $K'_{\rho,\mu}(r,\varphi,z)$ and $P'_{\rho,\mu}(\tau,\psi,k,z)$ take the form (\ref{eq:MLF_int3_K'_sum}) and (\ref{eq:MLF_int3_P'_sum}) respectively. If we use these expressions now in (\ref{eq:lemm_MLF_int1_case3}), then we will obtain (\ref{eq:MLF_int3_case3}). The third item of the lemma is proved.

\emph{Case 4.} Now we consider the case $1/2<\rho\leqslant1, \delta_{1\rho}=\delta_{2\rho}=\pi$. In the article \cite{Saenko2020d} (see Theorem 3 case 4 in \cite{Saenko2020d}) it was shown that in this case the function $E_{\rho,\mu}(z)$ can be represented in the form
\begin{multline}\label{eq:lemm_MLF_int1_case4}
  E_{\rho,\mu}(z)=\int_{0}^{1-\varepsilon_1}K_{\rho,\mu}(r,\pi,z)dr-
  \int_{1+\varepsilon_1}^{\infty}K_{\rho,\mu}(r,\pi,z)dr +\\
  +\int_{-\pi}^{0}P_{\rho,\mu}^\prime(\varepsilon_1,\psi,0,z)d\psi+
  \int_{-2\pi}^{-\pi}P_{\rho,\mu}^\prime(\varepsilon_1,\psi,-2,z)d\psi,
\end{multline}
where $0<\varepsilon_1<1$.

As before, we will represent the complex numbers $z$ and $\mu$ in the form $z=te^{i\theta}$, $\mu=\mu_R+i\mu_I$ and substitute these expressions in (\ref{eq:lemm_MLF_int1_case4}). We consider at first how the integrand $K_{\rho,\mu}(r,\delta,z)$ is transformed with such a substitution. This function is defined by the expression  (19) in \cite{Saenko2020d}. It should be pointed out that the function $K_{\rho,\mu}(r,\delta,z)$ is the abbreviated notation of the function $K_{\rho,\mu}(r,-\delta, \delta,z)$ (see~Corollary~1~in~\cite{Saenko2020d} and the expression~(24)~in~\cite{Saenko2020d}). Consequently, using that fact that  $K_{\rho,\mu}(r,\delta_\rho,z)\equiv K_{\rho,\mu}(r,-\delta_\rho,\delta_\rho,z)$, we obtain
$ K_{\rho,\mu}(r,-\pi, \pi,z)\equiv K_{\rho,\mu}(r,\pi,z).$  Then, from the expressions (\ref{eq:KRe+KIm_lemm_int2}) and (\ref{eq:KRe_deltaRho_KIm_deltaRho}) we obtain
\begin{equation}\label{eq:MLF_int3_K_piRho_sum}
  K_{\rho,\mu}(r,\pi,z)\equiv K_{\rho,\mu}(r,\pi,t,\theta)=K_{\rho,\mu}^{Re}(r,\pi,t,\theta)+iK_{\rho,\mu}^{Im}(r,\pi,t,\theta),
\end{equation}
where $K_{\rho,\mu}^{Re}(r,\delta,t,\theta)$ and $K_{\rho,\mu}^{Im}(r,\delta,t,\theta)$ are defined by the expressions (\ref{eq:K_Re_deltaRho_corol_int2}) and (\ref{eq:K_Im_deltaRho_corol_int2}), respectively.

It remains to consider how the integrand $P'_{\rho,\mu}(\tau,\psi,k,z)$ is transformed.  However, this issue has already been considered by us when dealing with \emph{Case 2}. It was obtained that this function can be represented  in the form  (\ref{eq:MLF_int3_P'_sum}). Now using  (\ref{eq:MLF_int3_K_piRho_sum}) and (\ref{eq:MLF_int3_P'_sum}) in (\ref{eq:lemm_MLF_int1_case4}) we get the representation (\ref{eq:MLF_int3_case4}). The fourth item of the lemma is proved.

\emph{Case 5.} Consider now the case when $\delta_{1\rho}=\delta_{2\rho}=\delta_\rho$.  This case is a particular case of the representation (\ref{eq:MLF_int3}). However, as it was mentioned earlier, at the same values of the parameters $\delta_{1\rho}$ and  $\delta_{2\rho}$ the integrand $K_{\rho,\mu}(r,\varphi_1,\varphi_2,t,\theta)$ becomes significantly simple. Therefore, the case of equal values of these parameters and the case $\delta_{1\rho}=\delta_{2\rho}=\pi/\rho$ are considered separately.

If $\delta_{1\rho}=\delta_{2\rho}=\delta_\rho$, then the inequalities (\ref{eq:deltaRho_cond_corol_int3}) become simple and take the form:
$\frac{\pi}{2\rho}<\delta_\rho\leqslant\frac{\pi}{\rho}$, if $\rho>1$, and  $\frac{\pi}{2\rho}<\delta_\rho<\pi$, if $1/2<\rho\leqslant1$, and the condition (\ref{eq:argZ_cond_corol_int3}) is written in the form $\frac{\pi}{2\rho}-\delta_{\rho}+\pi<\theta<-\frac{\pi}{2\rho}+\delta_{\rho}+\pi$. As a result, the representation (\ref{eq:MLF_int3}) takes the form
\begin{equation*}
  E_{\rho,\mu}(z)=\int_{0}^{\infty}K_{\rho,\mu}^{Re}(r, -\delta_{\rho},\delta_{\rho},t,\theta) dr
+i \int_{0}^{\infty}K_{\rho,\mu}^{Im}(r,-\delta_{\rho},\delta_{\rho},t,\theta) dr,
\end{equation*}
where $K_{\rho,\mu}^{Re}(r,\psi_1,\psi_2,t,\theta)$ and $K_{\rho,\mu}^{Im}(r,\psi_1,\psi_2,t,\theta)$ are defined by the expressions  (\ref{eq:K_Re_lemm_int2}) and (\ref{eq:K_Im_lemm_int2}) respectively. Now using in this expression (\ref{eq:KRe_deltaRho_KIm_deltaRho}) we obtain the representation (\ref{eq:MLF_int3_deltaRho}). The fifth item of the lemma is proved.

\emph{Case 6.} Consider the case $\delta_\rho=\pi/\rho$. The integral representation for the Mittag-Leffler function for this case is directly obtained from the case considered earlier. However, as it was mentioned before, the parameter $\delta_\rho$ could not take values more than the value $\pi$. From this it follows  $\pi/\rho\leqslant\pi$  only if  $\rho\geqslant1$. It should be noted that the value $\rho=1$ should be excluded, since in this case the segments $\Gamma_1$ and $\Gamma_2$ of the auxiliary contour $\Gamma$ (see~Theorem~3 and Fig.~2 in~\cite{Saenko2020d}) will pass through the singular point $\zeta=1$. As a result, the integrals in (\ref{eq:MLF_int3_deltaRho}) will diverge. In addition, this case was considered in item 4 of the current theorem. Thus, we arrive at the condition that the case $\delta_\rho=\pi/\rho$ in the representation ``B'' is possible if $\rho>1$.

Now substituting in (\ref{eq:MLF_int3_deltaRho}) the value  $\delta_\rho=\pi/\rho$ we obtain
\begin{equation*}
  E_{\rho,\mu}(z)=\int_{0}^{\infty}K_{\rho,\mu}^{Re}(r,\pi/\rho,z) dr
+i \int_{0}^{\infty}K_{\rho,\mu}^{Im}(r,\pi/\rho,z) dr.
\end{equation*}
The expressions for $K_{\rho,\mu}^{Re}(r,\pi/\rho,z)$ and $K_{\rho,\mu}^{Im}(r,\pi/\rho,z)$ were obtained earlier when proving the second part of corollary~\ref{coroll:MLF_int2_delatRho}. Using the notation (\ref{eq:KRe_piRho}) and (\ref{eq:KIm_piRho}) in this expression we arrive at the representation (\ref{eq:MLF_int3_piRho}).
\begin{flushright}
  $\Box$
\end{flushright}
\end{proof}

It should be noted that in this theorem the results for Parameterizations 1, 2 and 3 are combined. Items 1-4 are related to Parameterization~1. At the same time, Item 4 can be attributed to both Parameterization~1 and Parameterization~2. Item~5 is related to Parameterization~2, and Item~6 to Parameterization~3.

\begin{figure}
  \centering
  \includegraphics[width=0.43\textwidth]{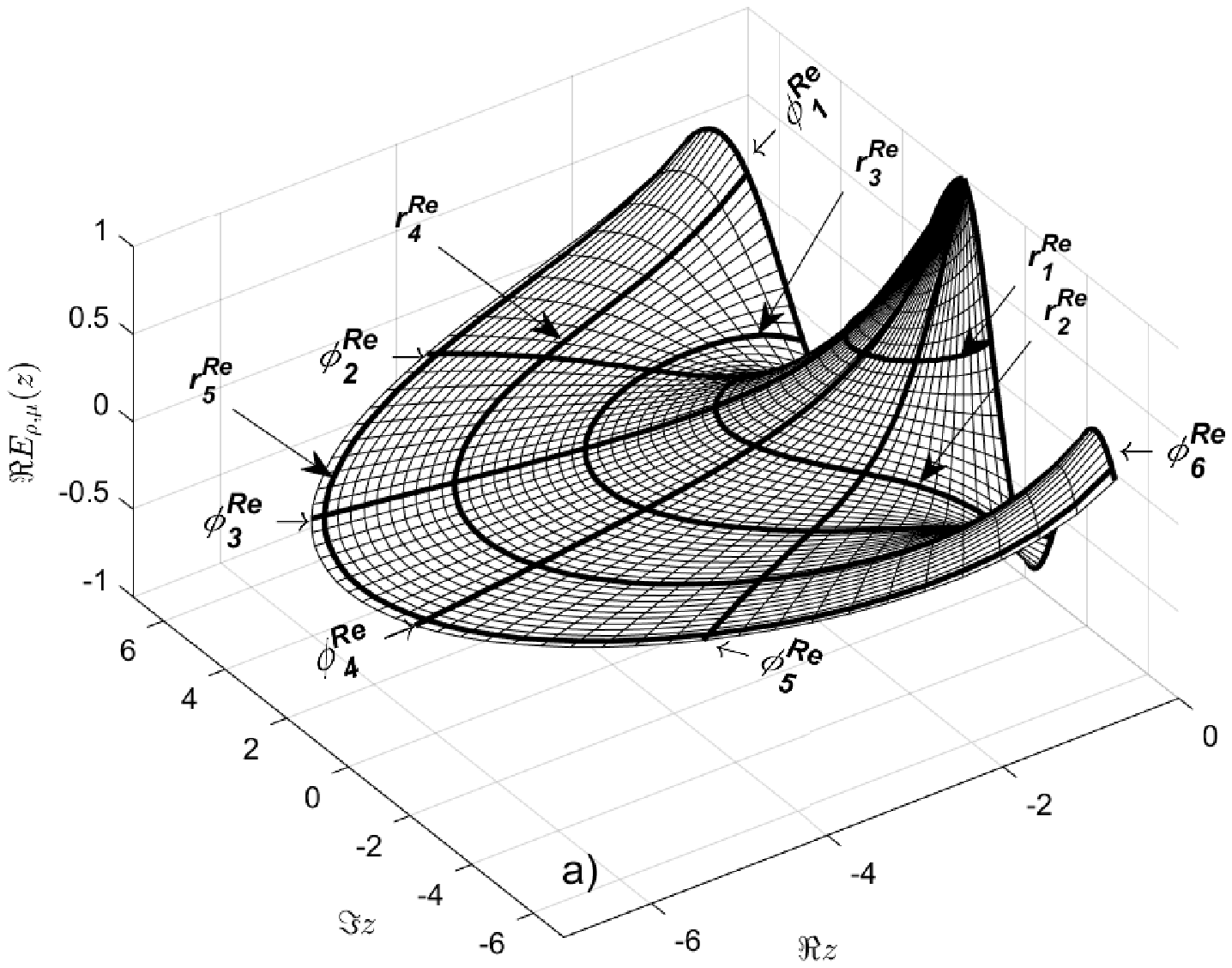}\hfill
  \includegraphics[width=0.43\textwidth]{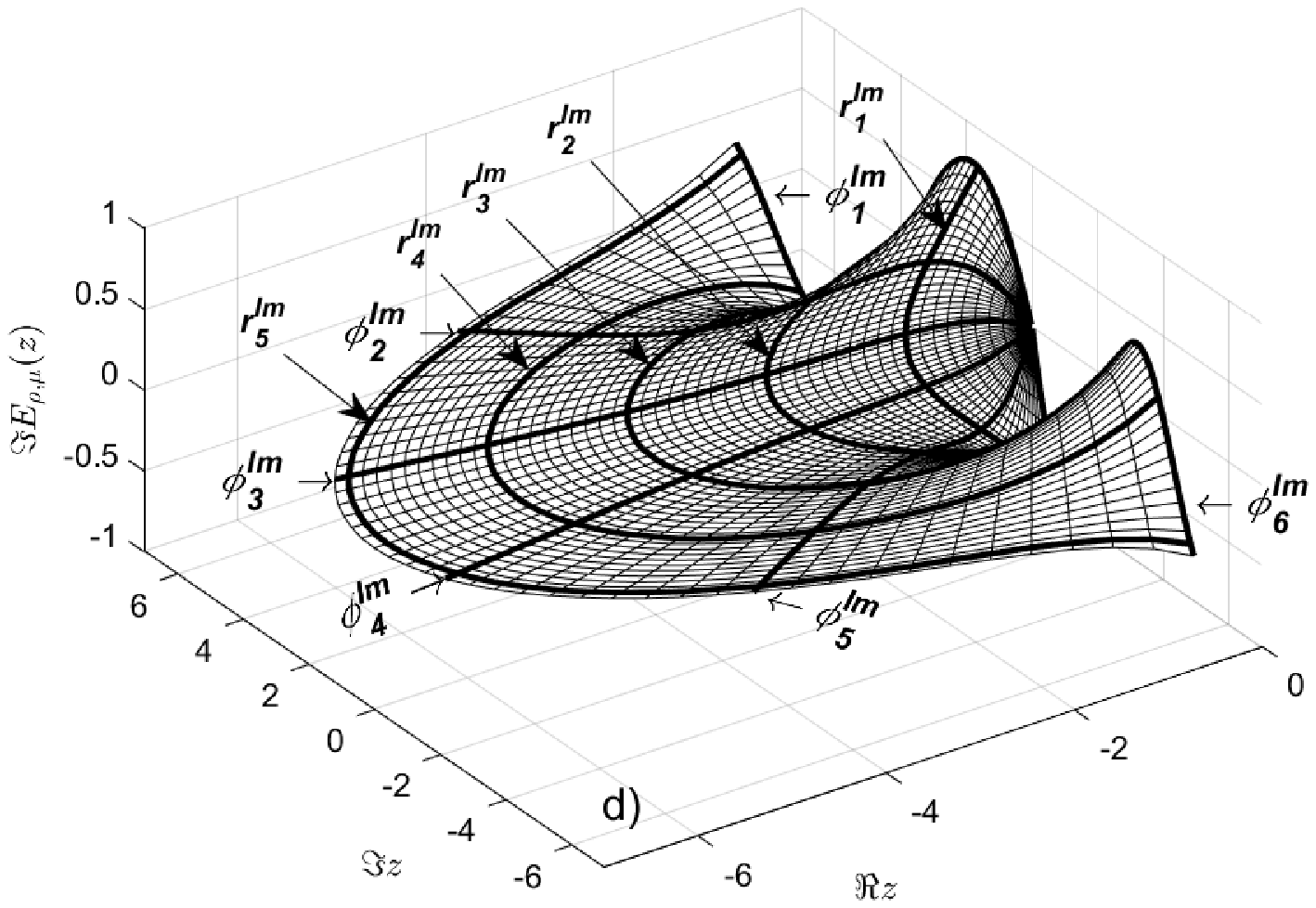}\\[4mm]
  \includegraphics[width=0.43\textwidth]{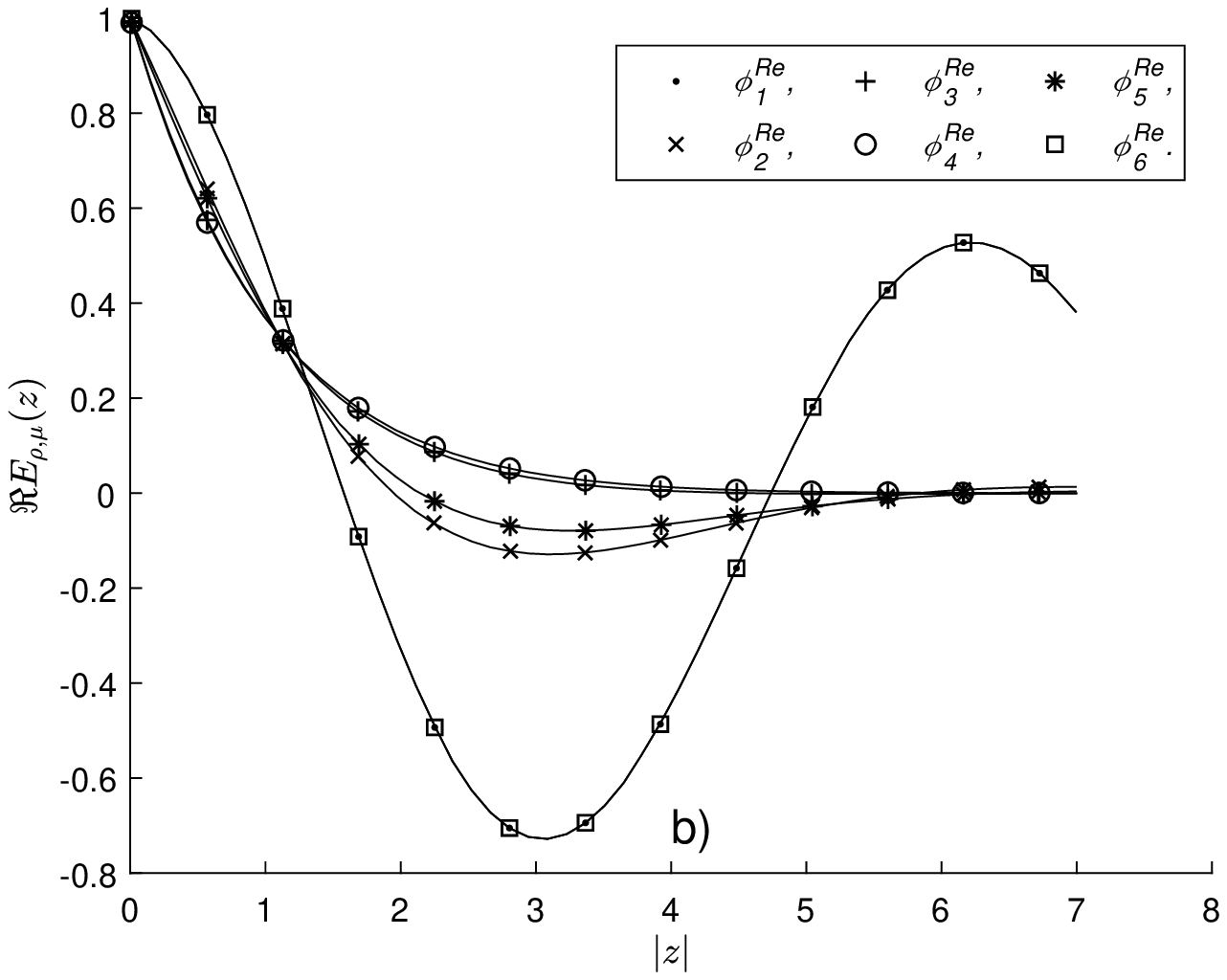}\hfill
  \includegraphics[width=0.43\textwidth]{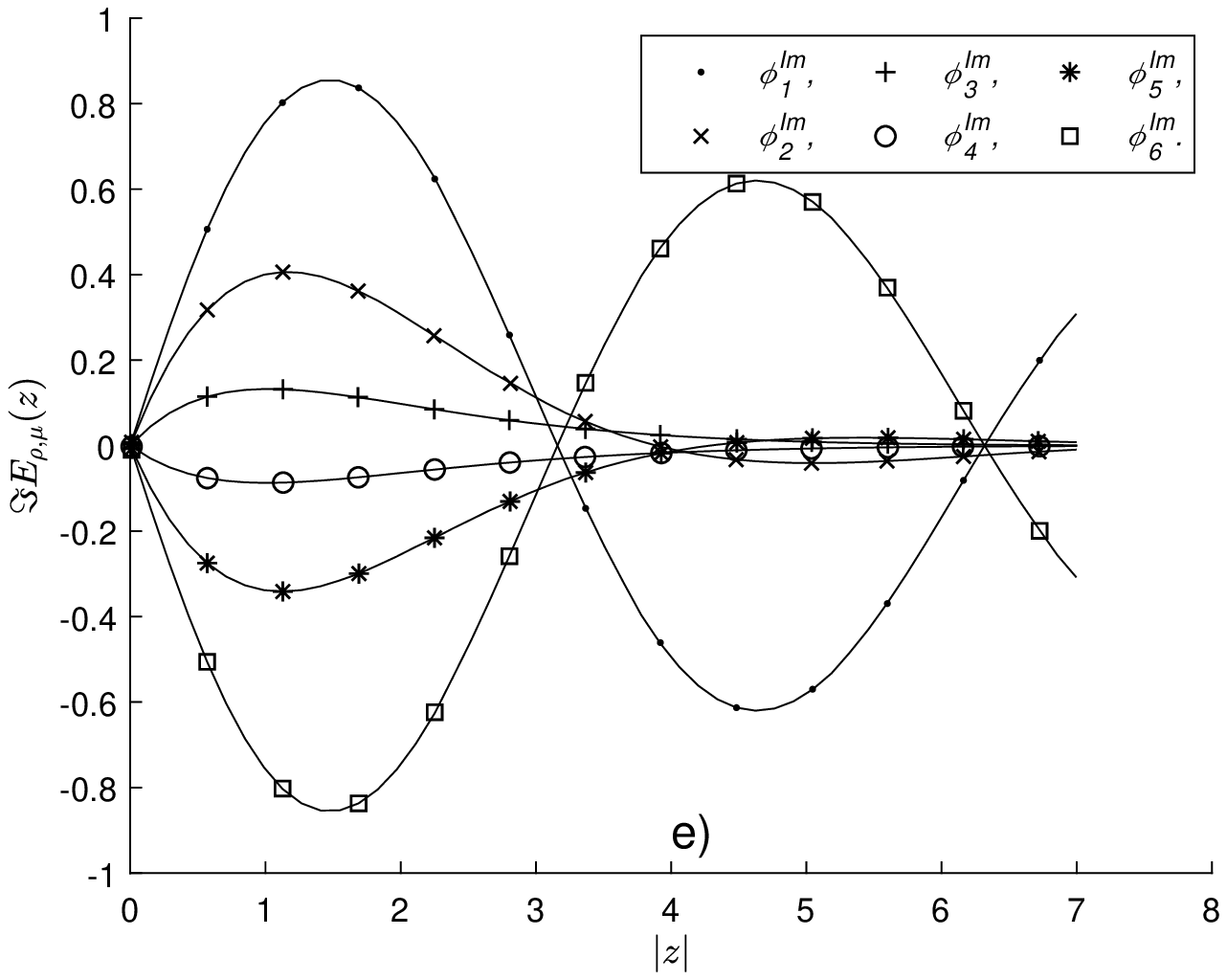}\\[4mm]
  \includegraphics[width=0.43\textwidth]{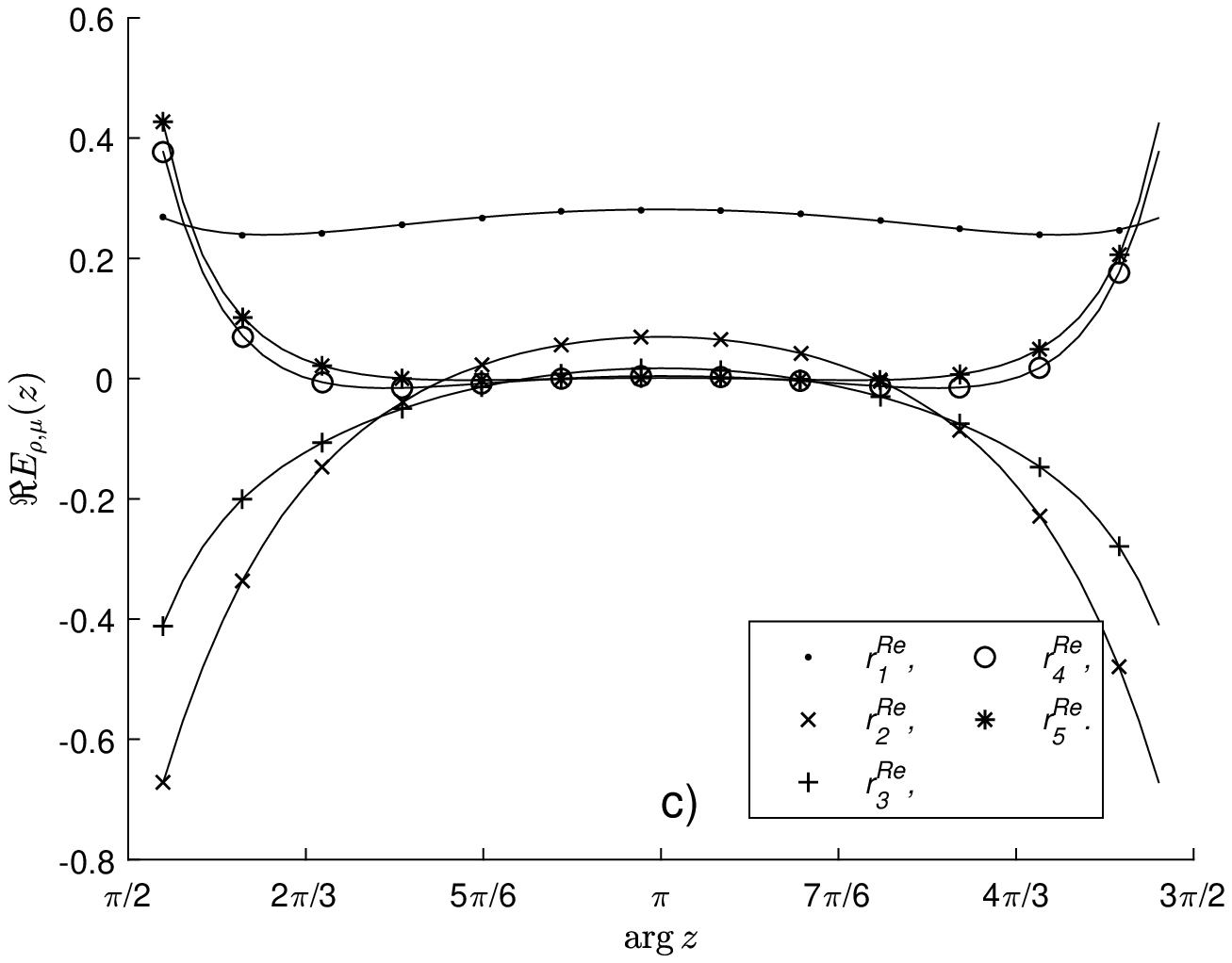}\hfill
  \includegraphics[width=0.43\textwidth]{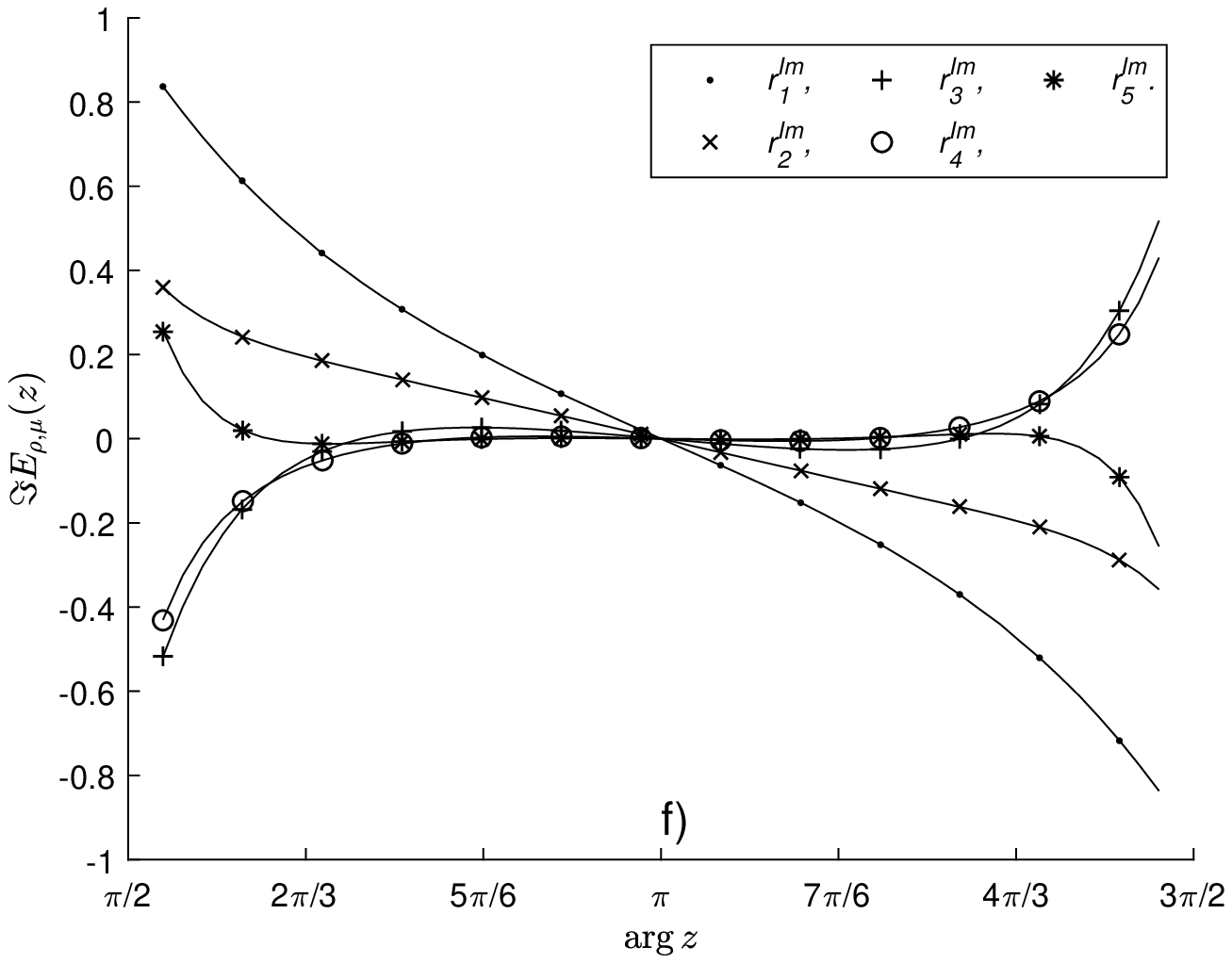}
  \caption{The function $E_{\rho,\mu}(z)$ for $\rho=1, \mu=1$ and $\delta_{1\rho}=35\pi/36$, $\delta_{2\rho}=35\pi/36$, $0.01\leqslant|z|\leqslant7, 3\pi/2-35\pi/36<\arg z<\pi/2+35\pi/36$.  On the figures a) and d) the surfaces – the formula (\ref{eq:MLF_int3}), the curves – the formula (\ref{eq:MLF_mu<=1}). On the figures b), c), e) and f) the curves – the formula (\ref{eq:MLF_int3}), the points – the formula (\ref{eq:MLF_mu<=1})
  }\label{fig:MLF_int3_case1_rho1_mu1_BA}
\end{figure}

\begin{figure}
  \centering
  \includegraphics[width=0.43\textwidth]{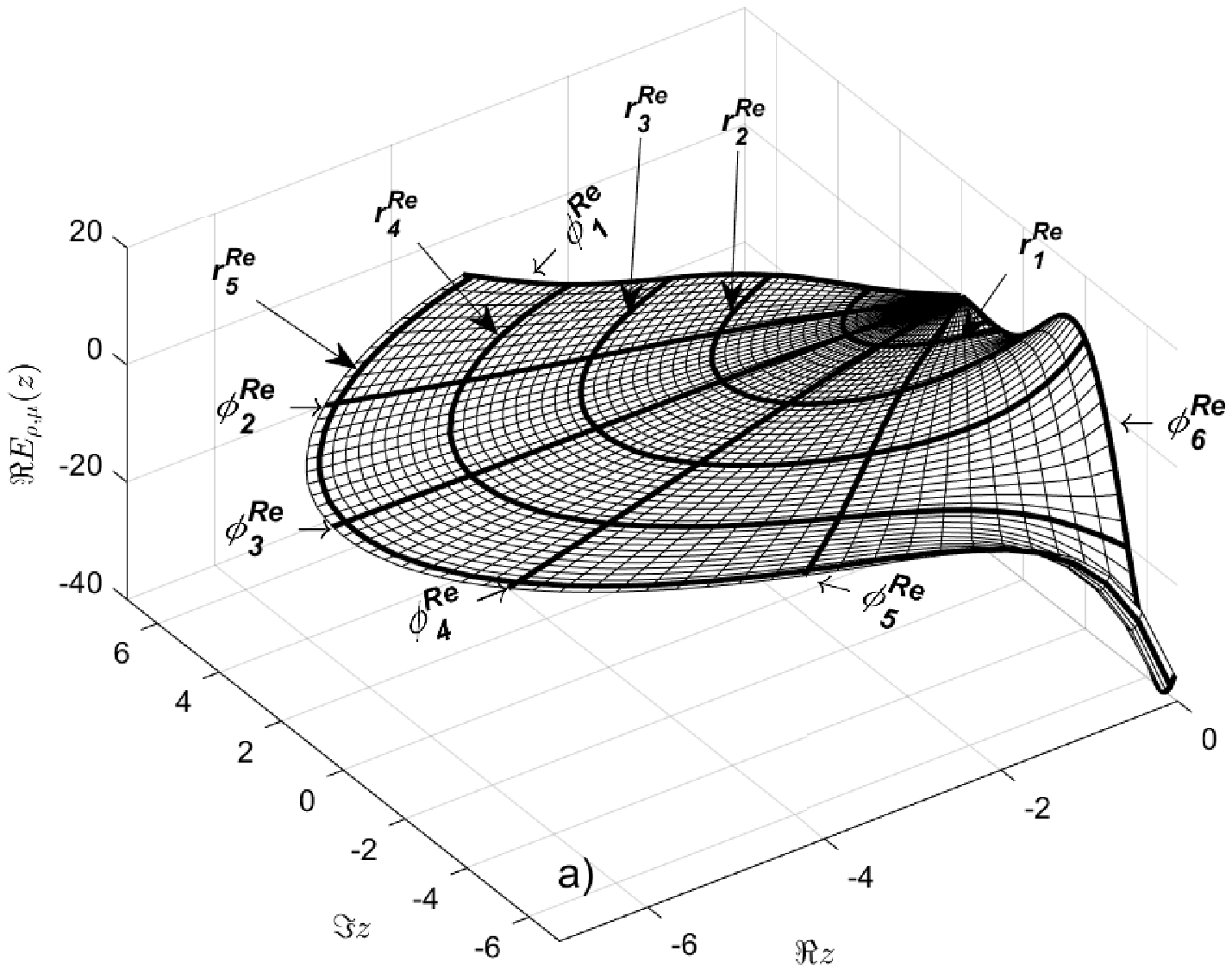}\hfill
  \includegraphics[width=0.43\textwidth]{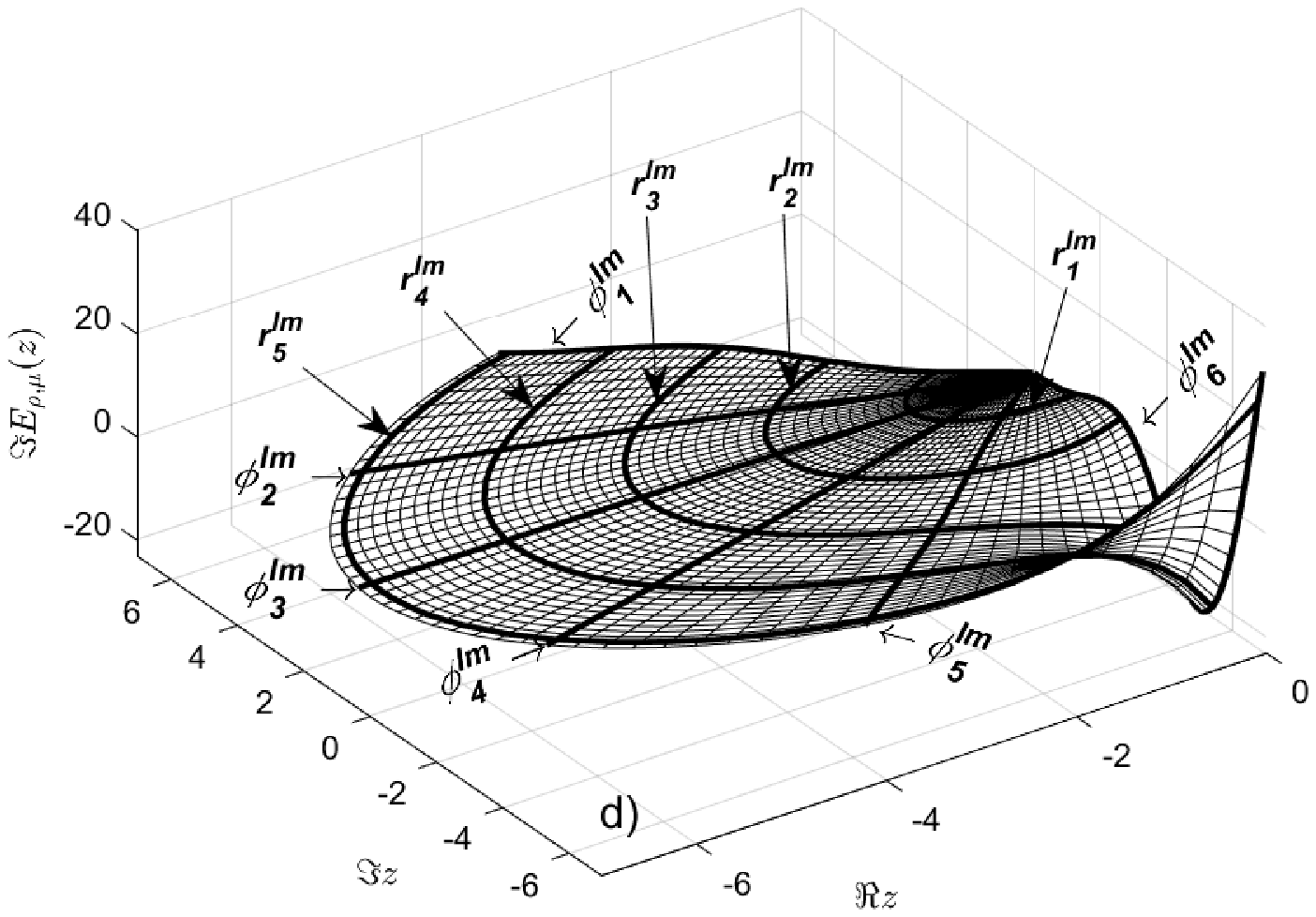}\\[4mm]
  \includegraphics[width=0.43\textwidth]{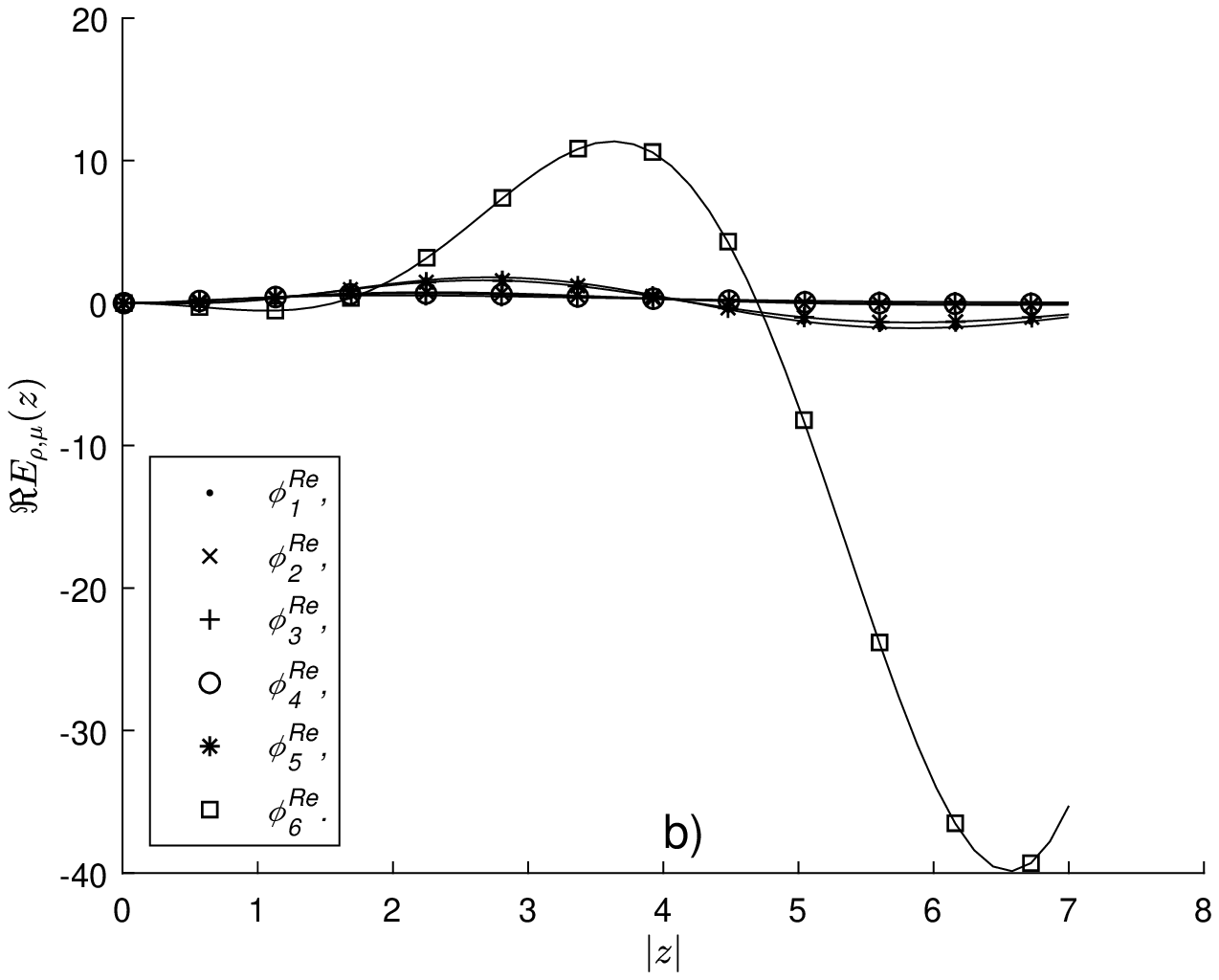}\hfill
  \includegraphics[width=0.43\textwidth]{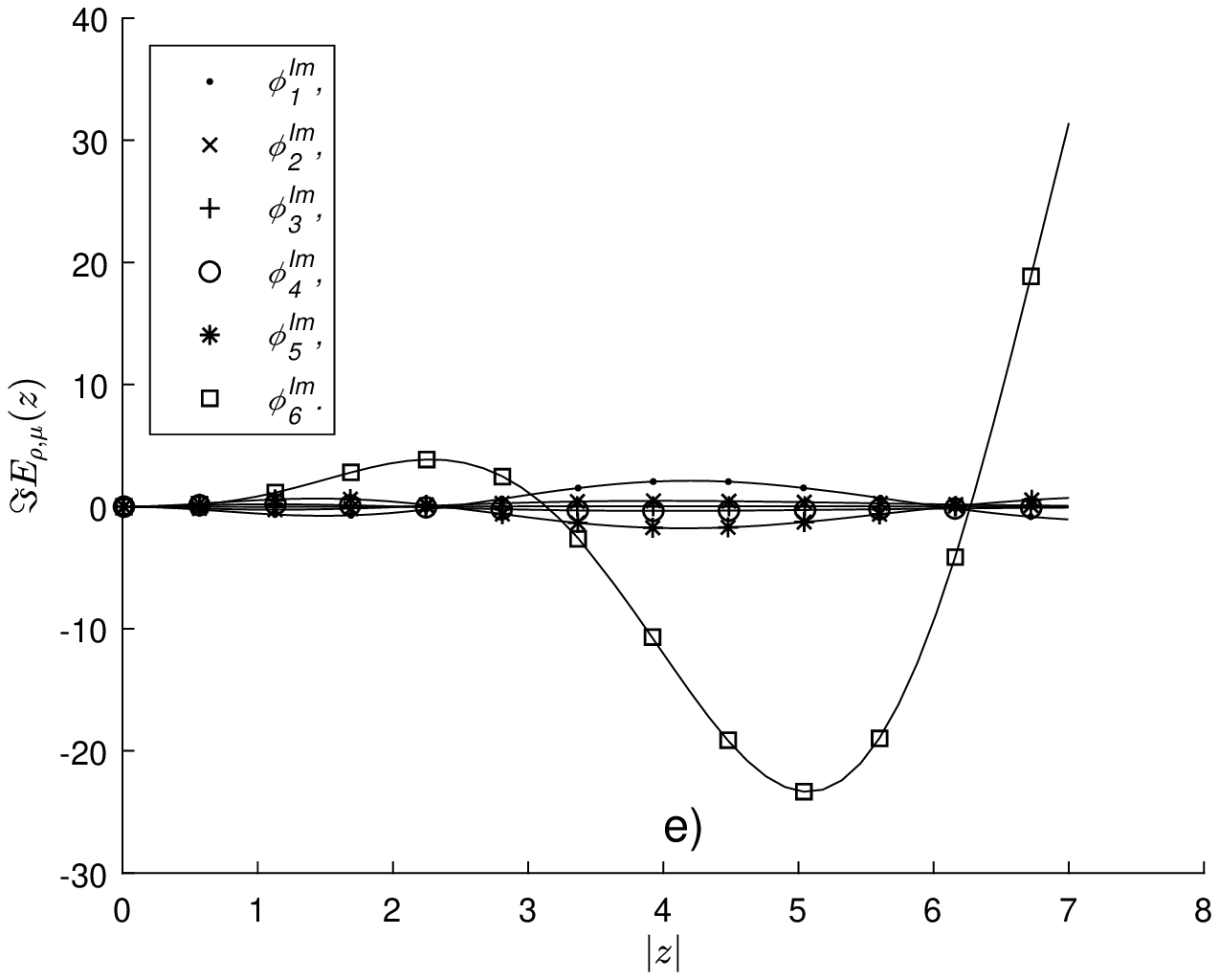}\\[4mm]
  \includegraphics[width=0.43\textwidth]{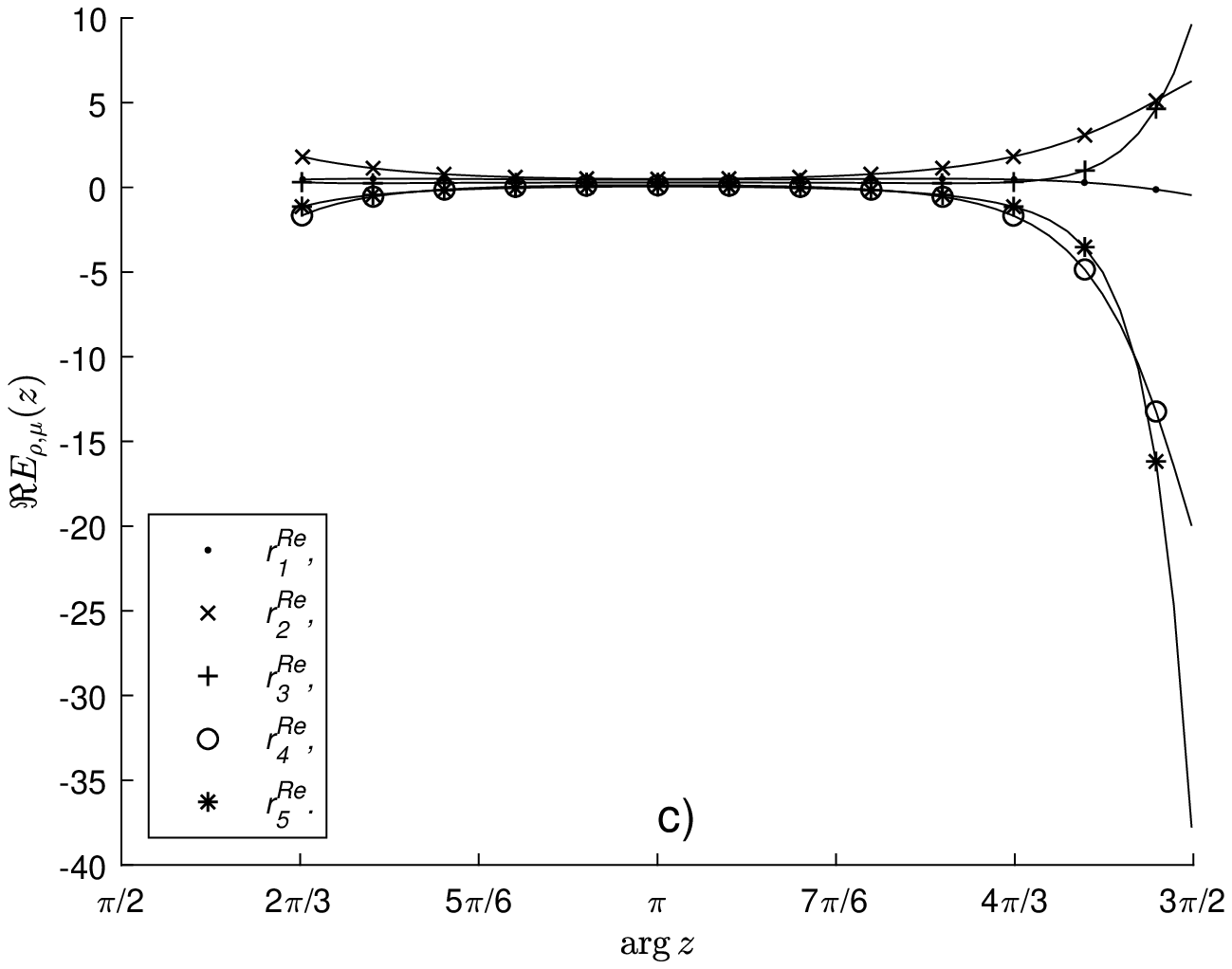}\hfill
  \includegraphics[width=0.43\textwidth]{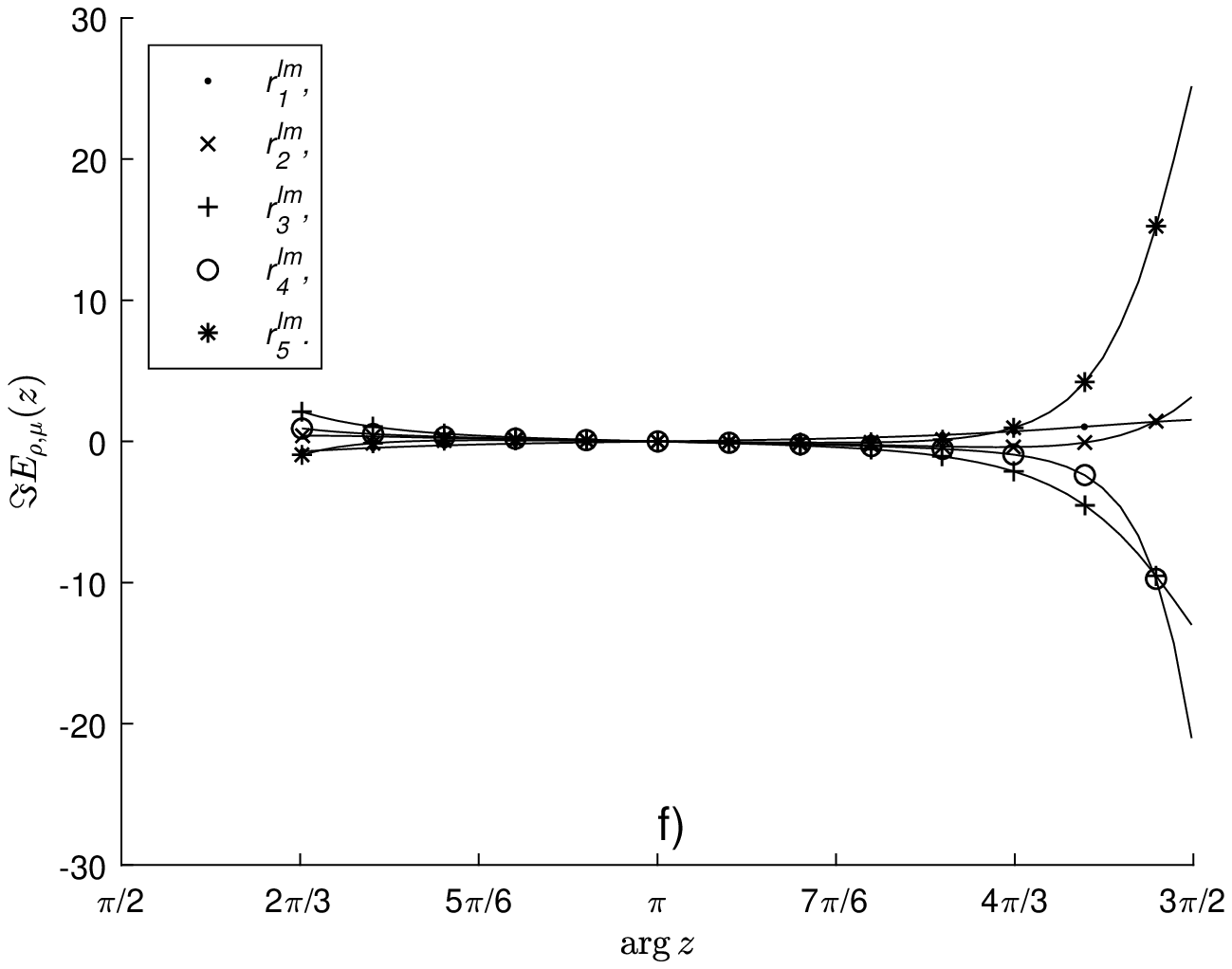}
  \caption{The function $E_{\rho,\mu}(z)$ for $\rho=1, \mu=-1$ and $\delta_{1\rho}=\pi$, $\delta_{2\rho}=5\pi/6$, $0.01\leqslant|z|\leqslant7, 2\pi/3<\arg z<3\pi/2$.  On the figures a) and d) the surfaces – the formula (\ref{eq:MLF_int3_case2}), the curves – the formula (\ref{eq:MLF_mu<=1}). On the figures b), c), e) and f) the curves – the formula (\ref{eq:MLF_int3_case2}), the points – the formula (\ref{eq:MLF_mu<=1})
  }\label{fig:MLF_int3_case2_rho1_mu_1_BA}
\end{figure}

\begin{figure}
  \centering
  \includegraphics[width=0.43\textwidth]{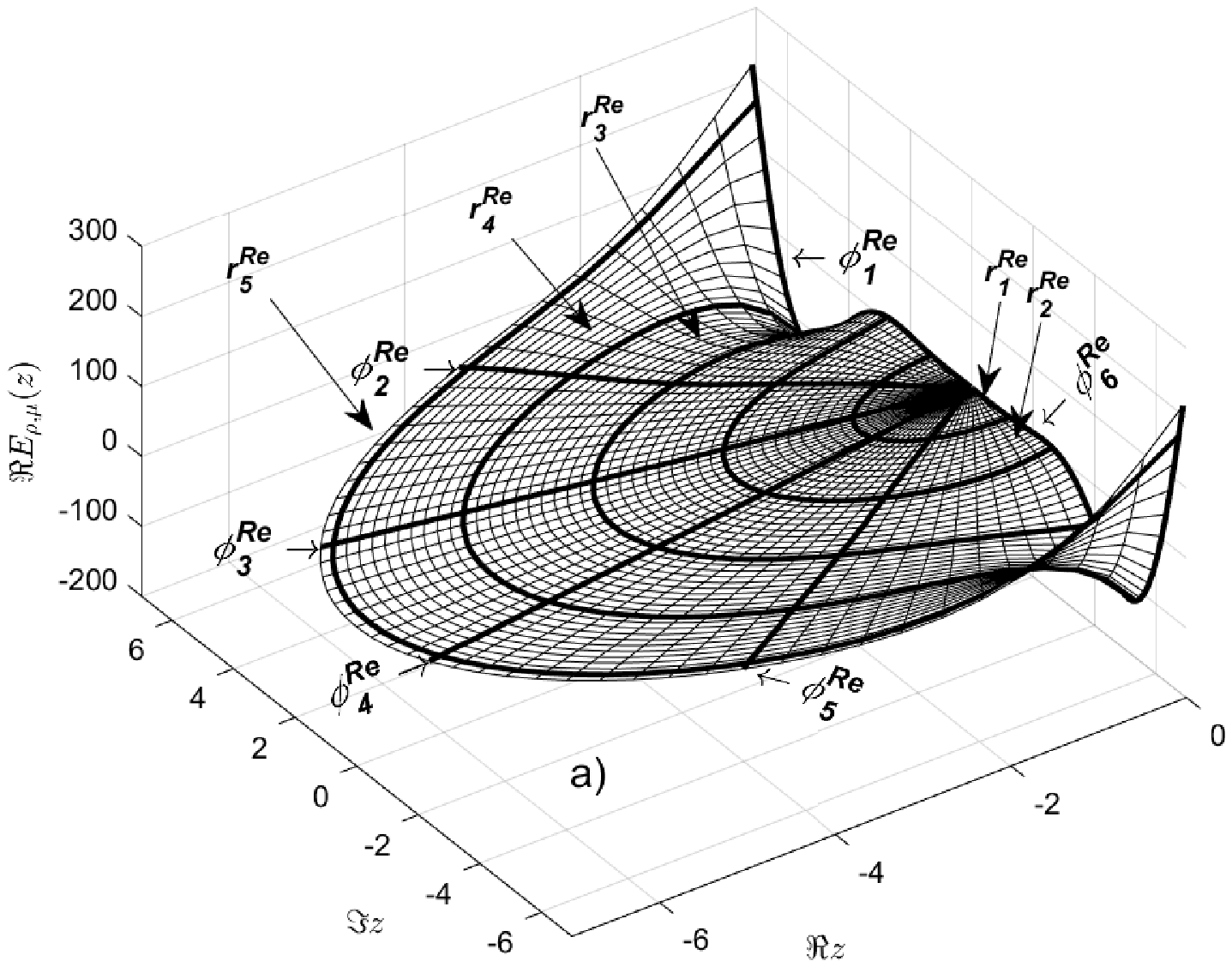}\hfill
  \includegraphics[width=0.43\textwidth]{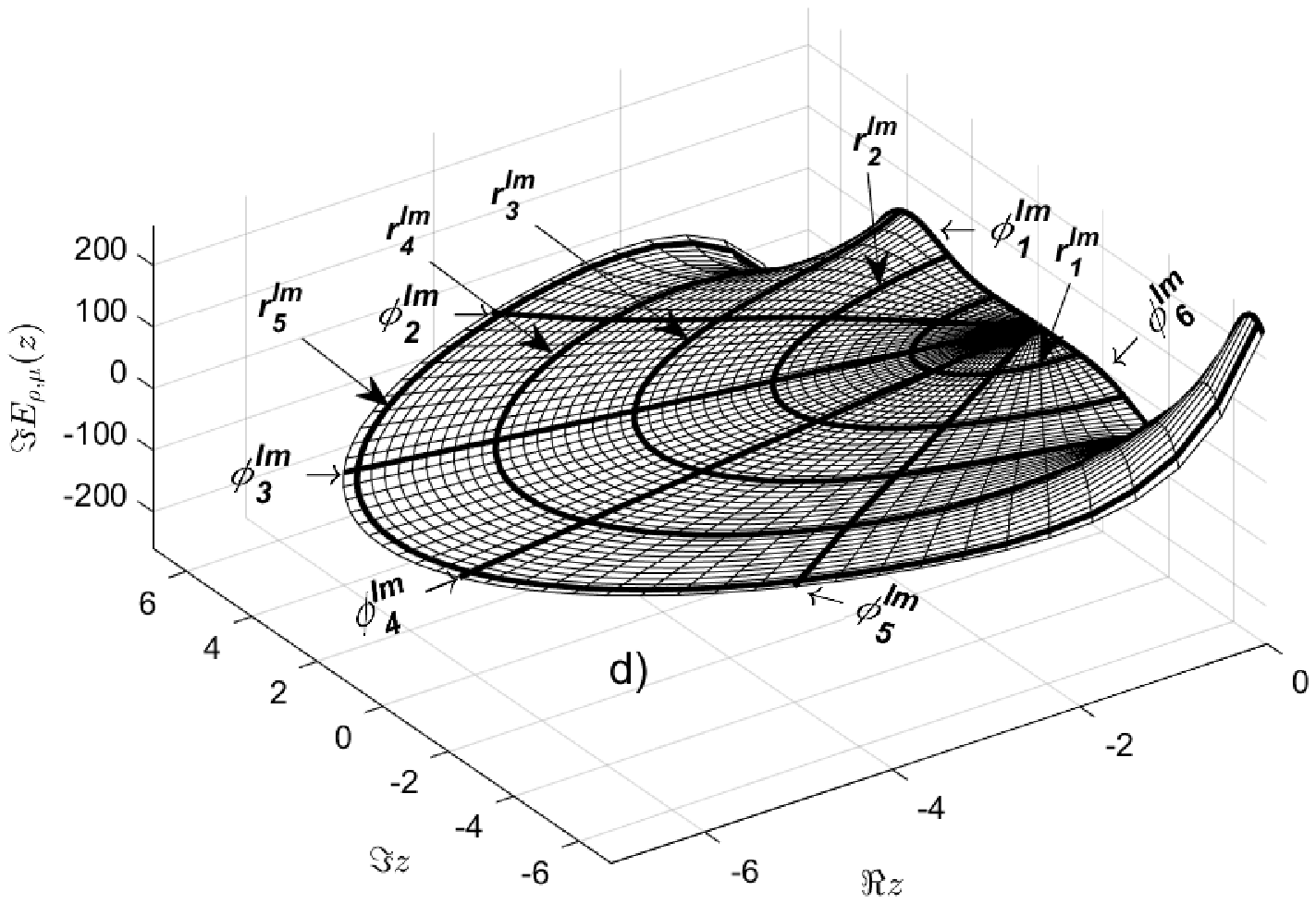}\\[4mm]
  \includegraphics[width=0.43\textwidth]{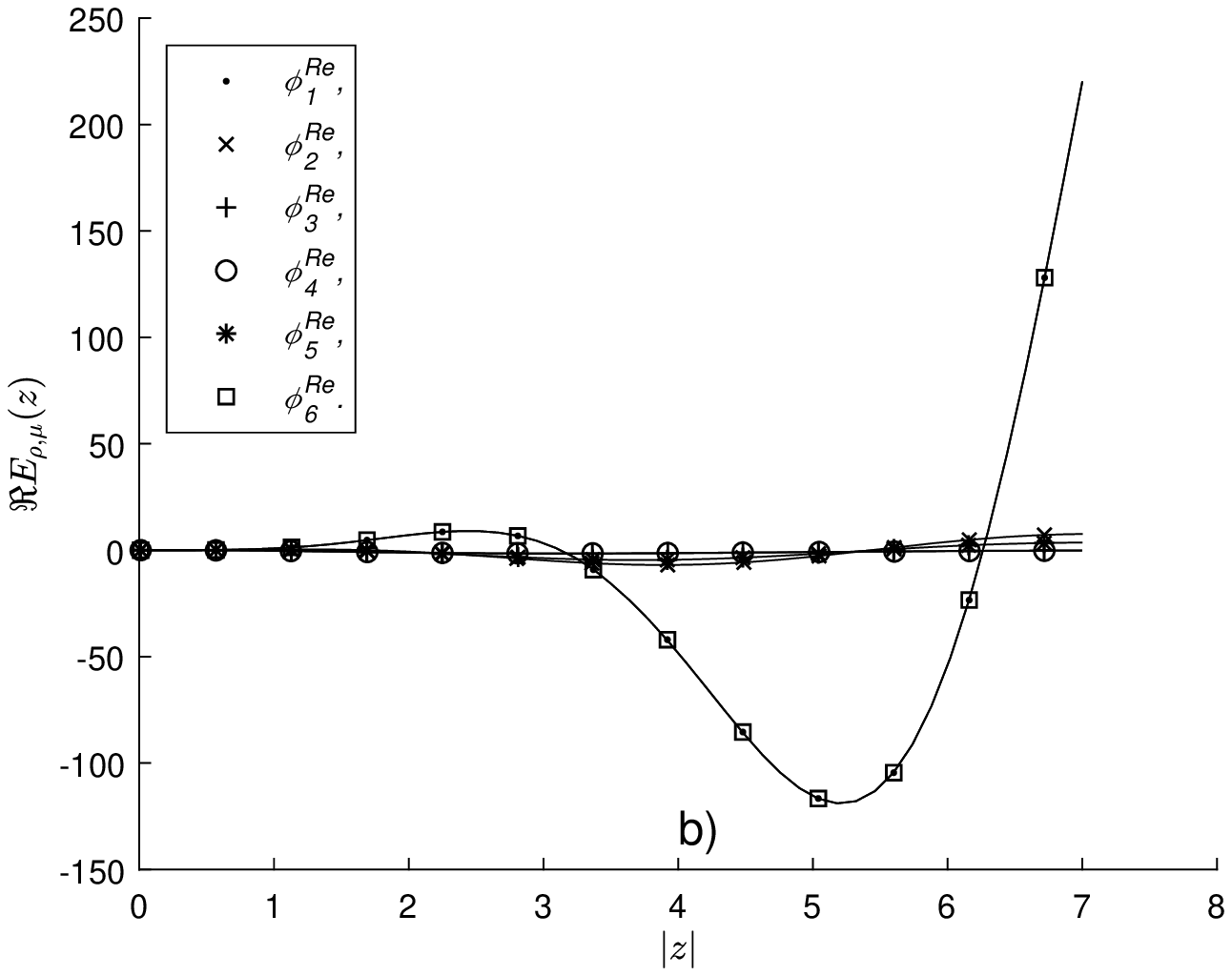}\hfill
  \includegraphics[width=0.43\textwidth]{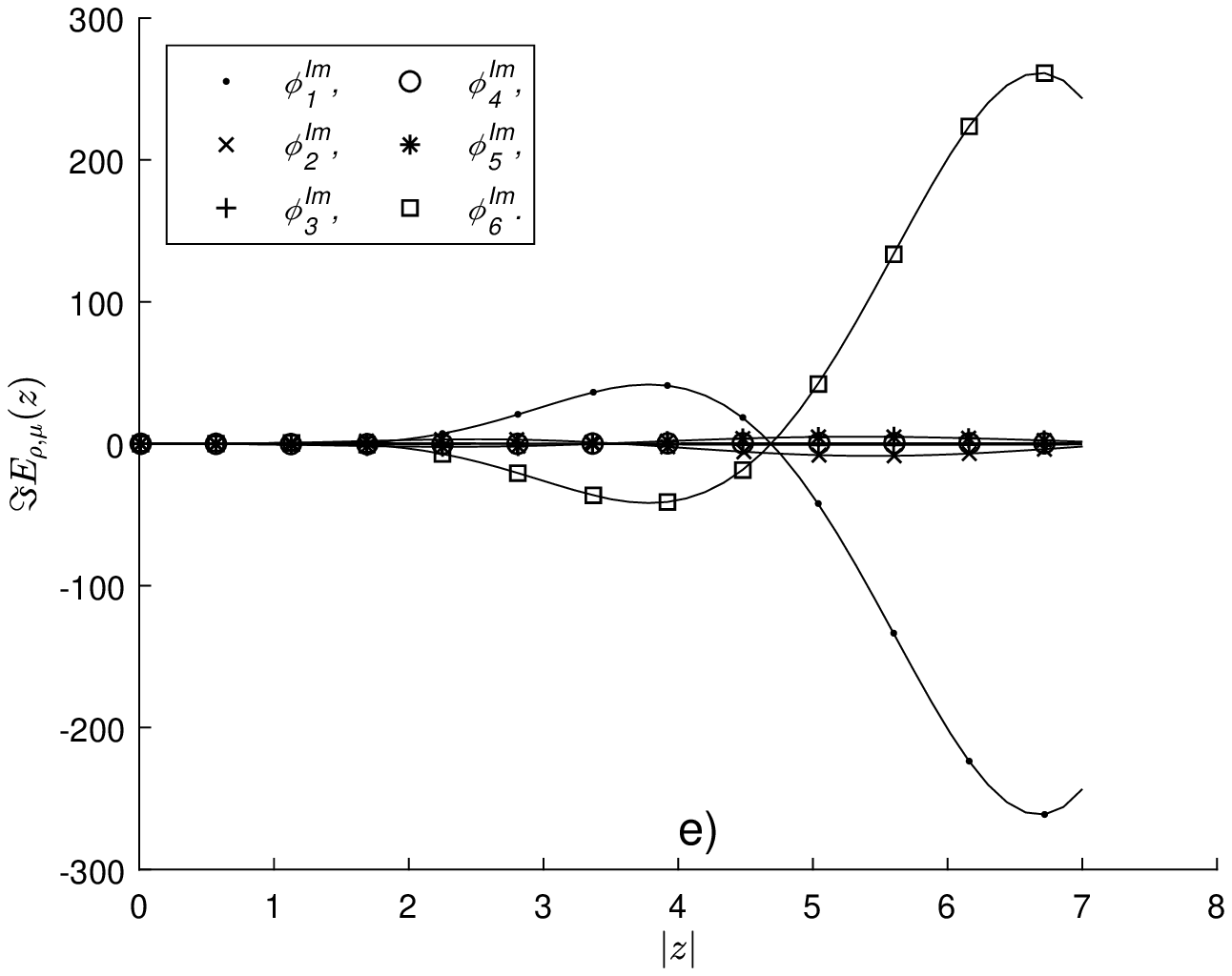}\\[4mm]
  \includegraphics[width=0.43\textwidth]{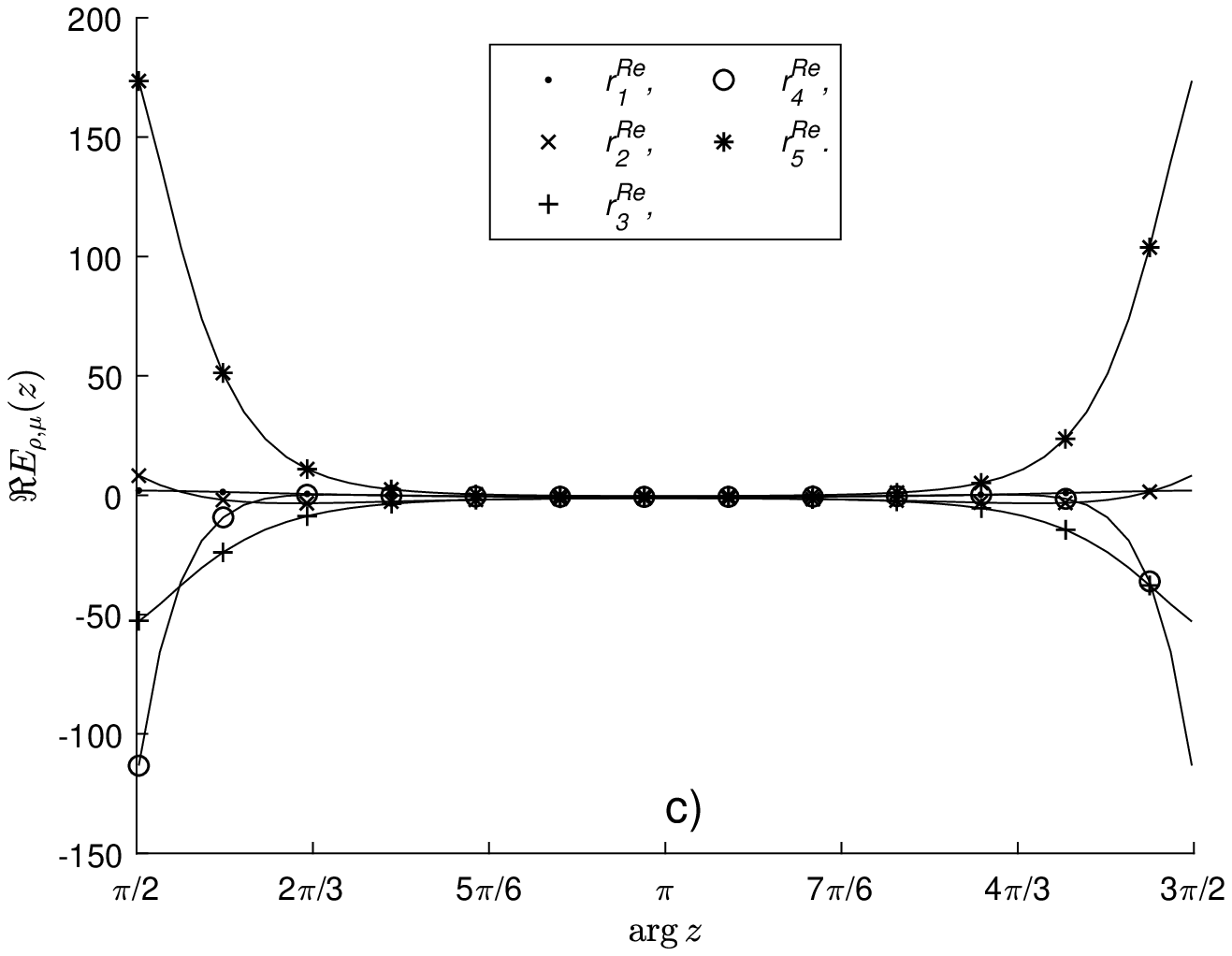}\hfill
  \includegraphics[width=0.43\textwidth]{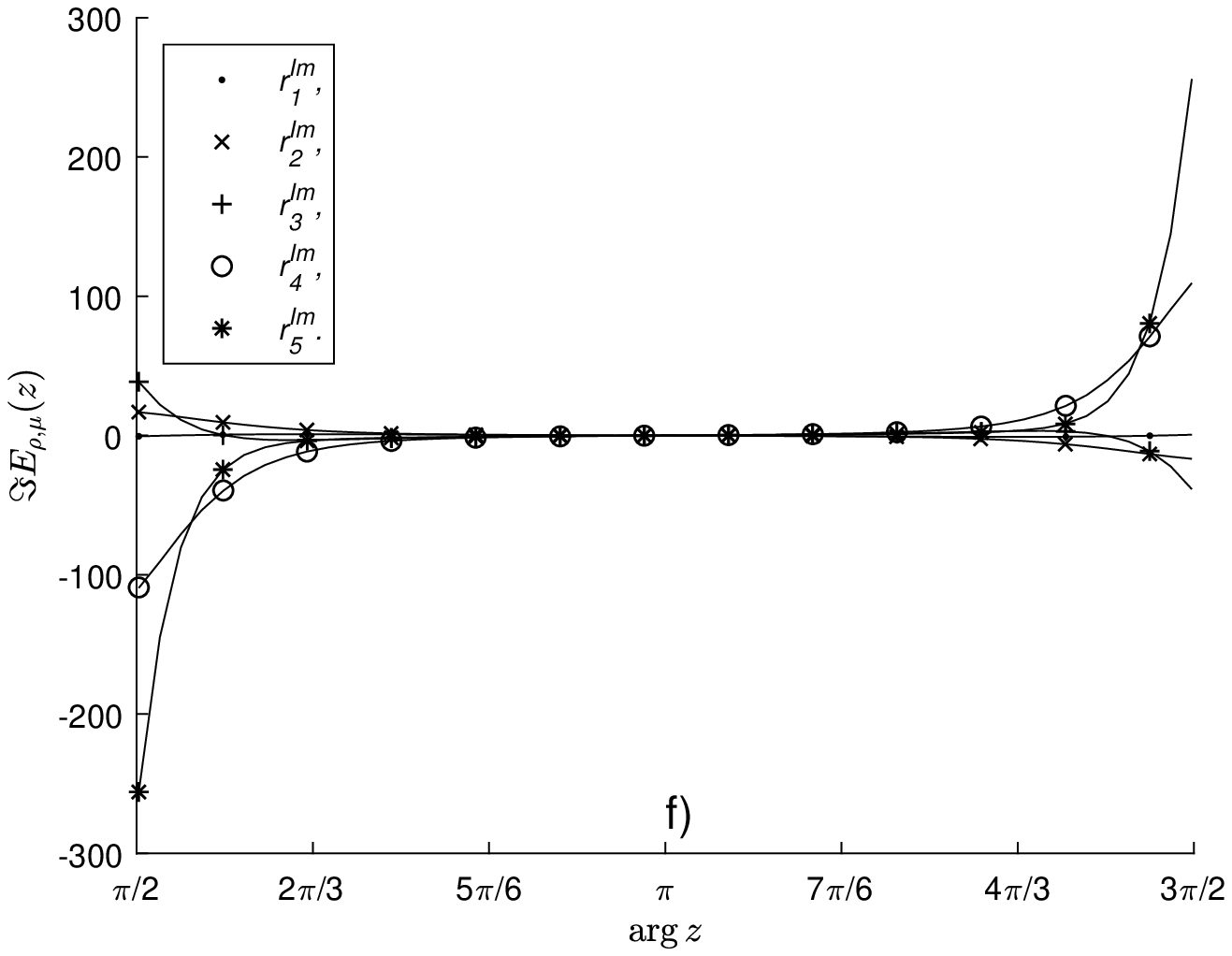}
  \caption{The function $E_{\rho,\mu}(z)$ for $\rho=1, \mu=-2$ and $\delta_{1\rho}=\pi$, $\delta_{2\rho}=\pi$, $0.01\leqslant|z|\leqslant7, \pi/2<\arg z<3\pi/2$.  On the figures a) and d) the surfaces – the formula (\ref{eq:MLF_int3_case4}), the curves – the formula (\ref{eq:MLF_mu<=1}). On the figures b), c), e) and f) the curves – the formula (\ref{eq:MLF_int3_case4}), the points – the formula (\ref{eq:MLF_mu<=1})
  }\label{fig:MLF_int3_case4_rho1_mu_2_BA}
\end{figure}

\begin{figure}
  \centering
  \includegraphics[width=0.43\textwidth]{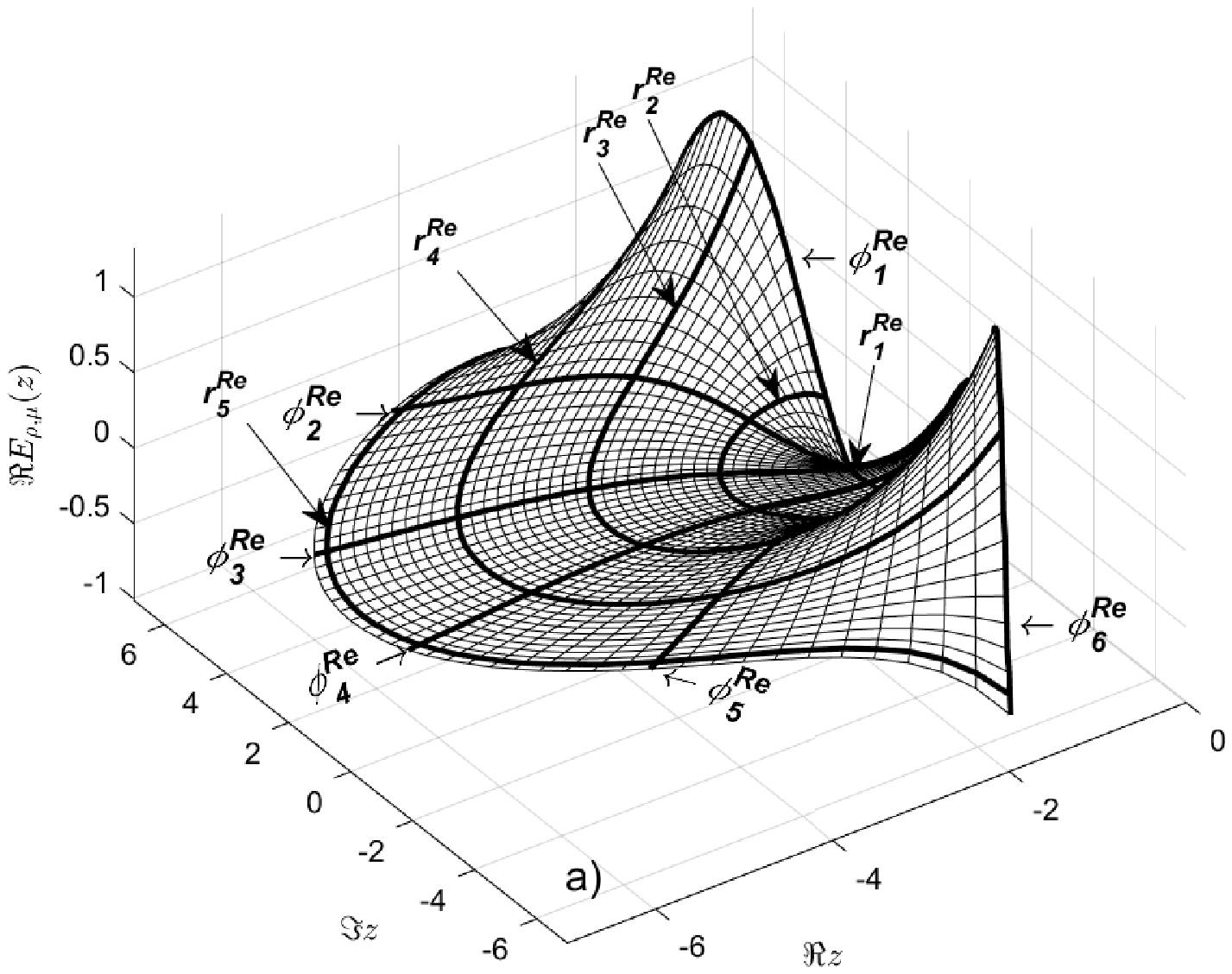}\hfill
  \includegraphics[width=0.43\textwidth]{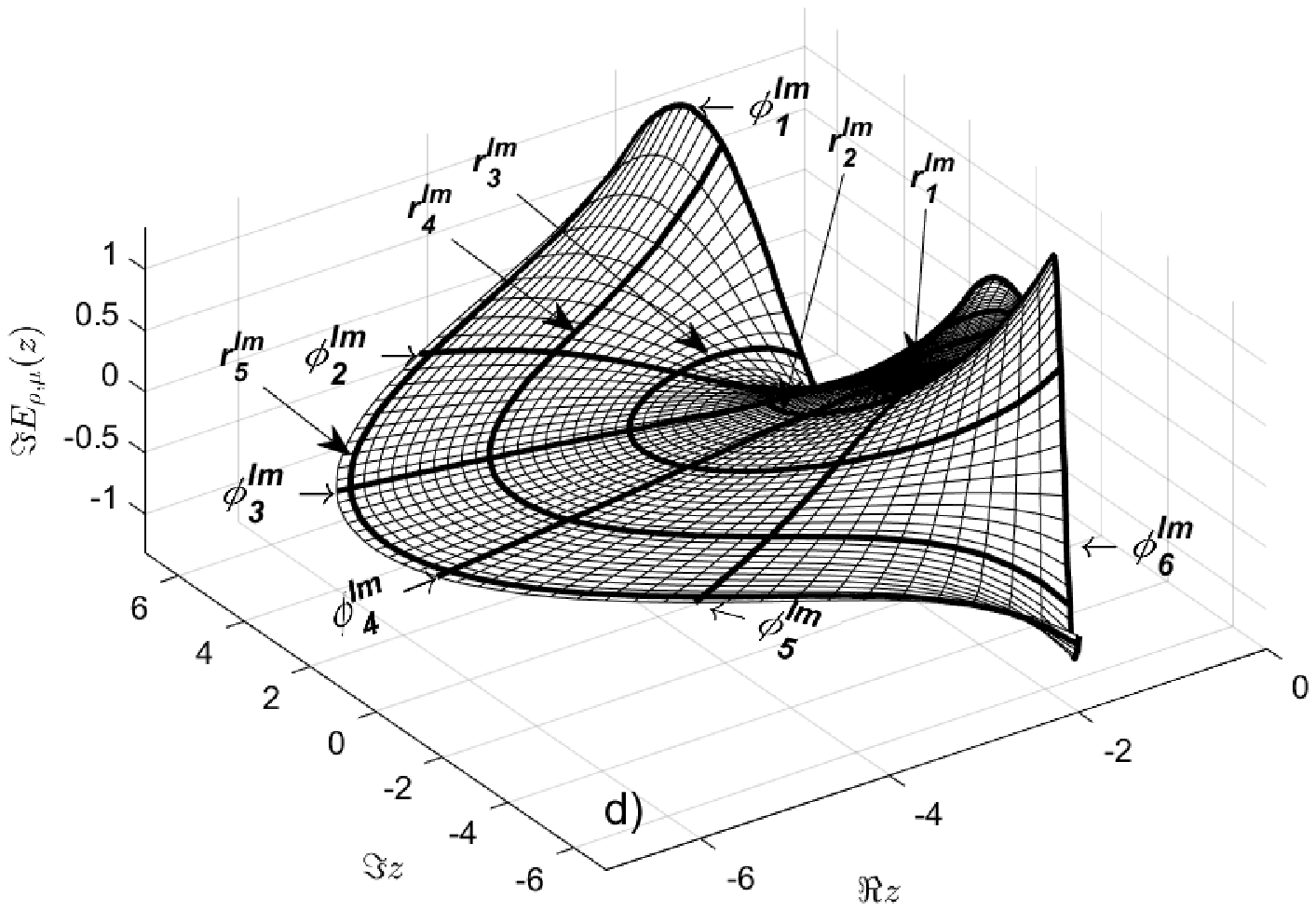}\\[4mm]
  \includegraphics[width=0.43\textwidth]{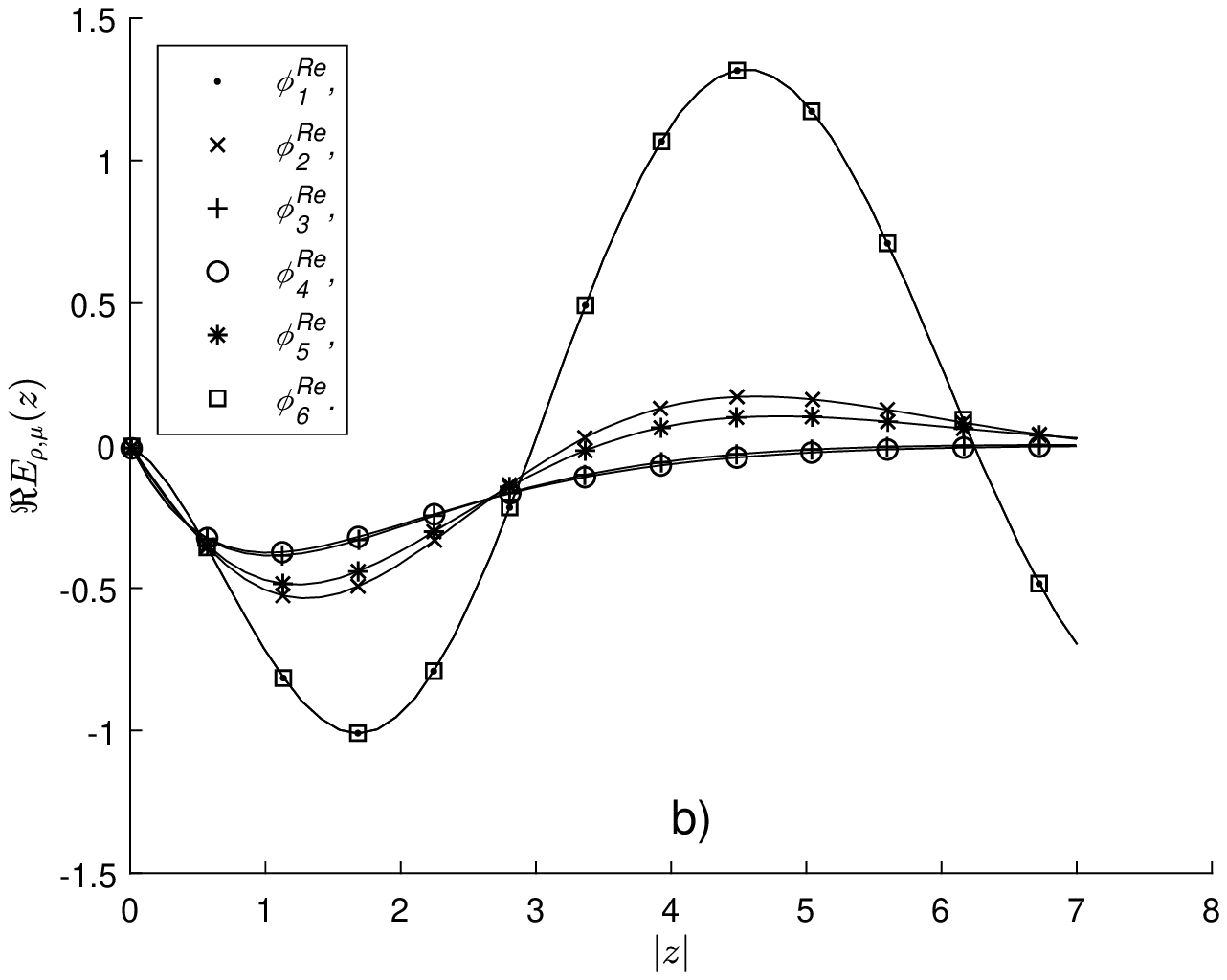}\hfill
  \includegraphics[width=0.43\textwidth]{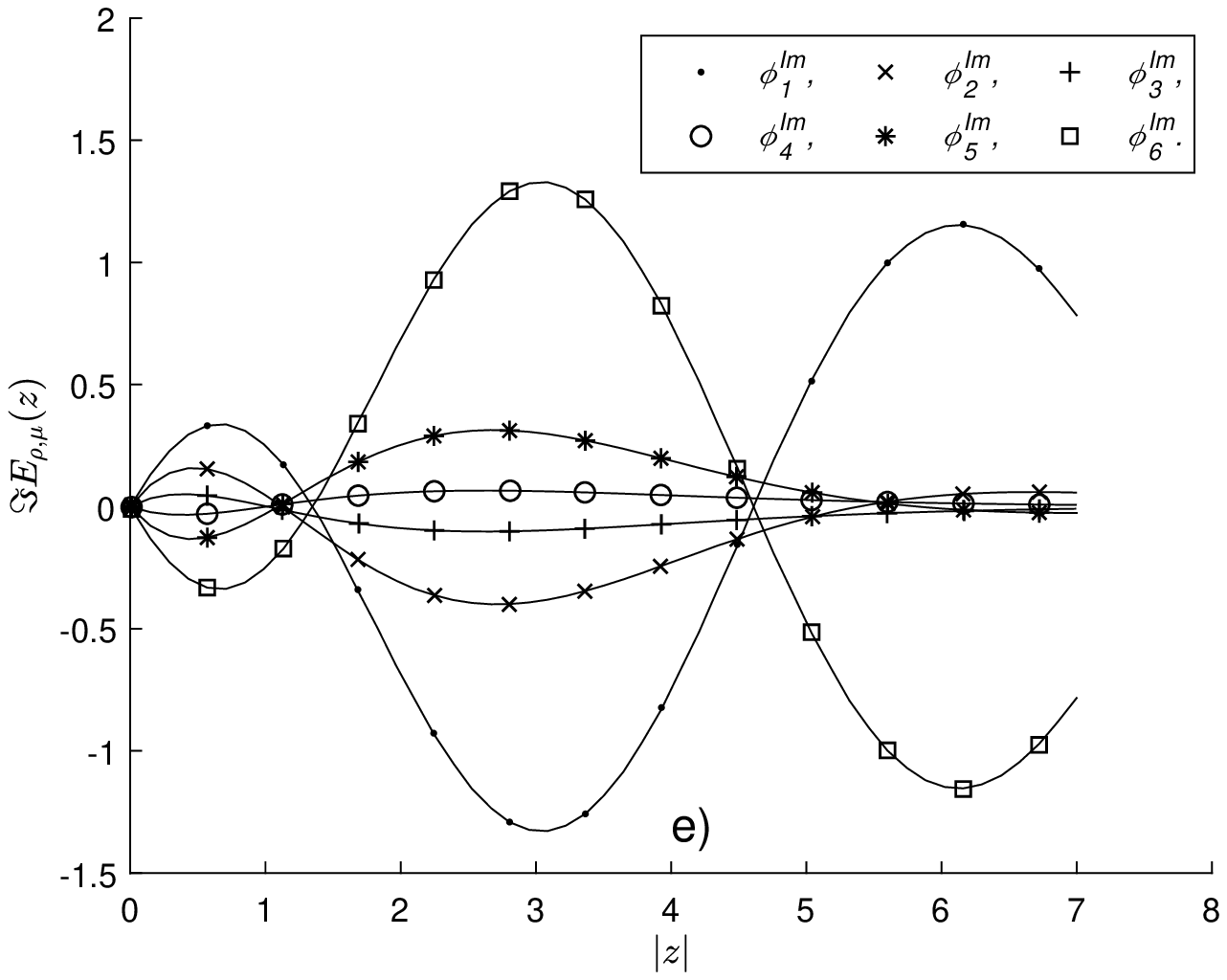}\\[4mm]
  \includegraphics[width=0.43\textwidth]{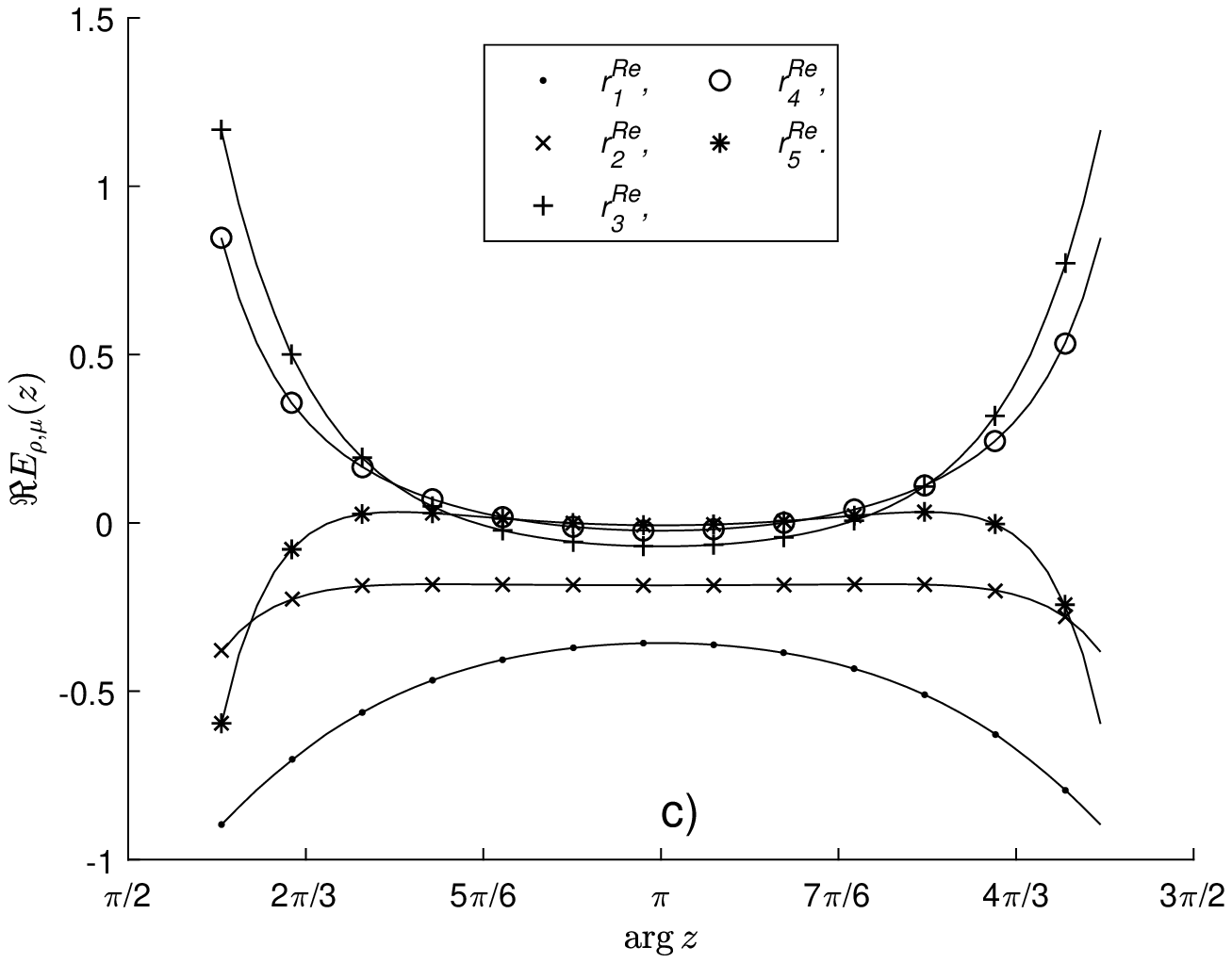}\hfill
  \includegraphics[width=0.43\textwidth]{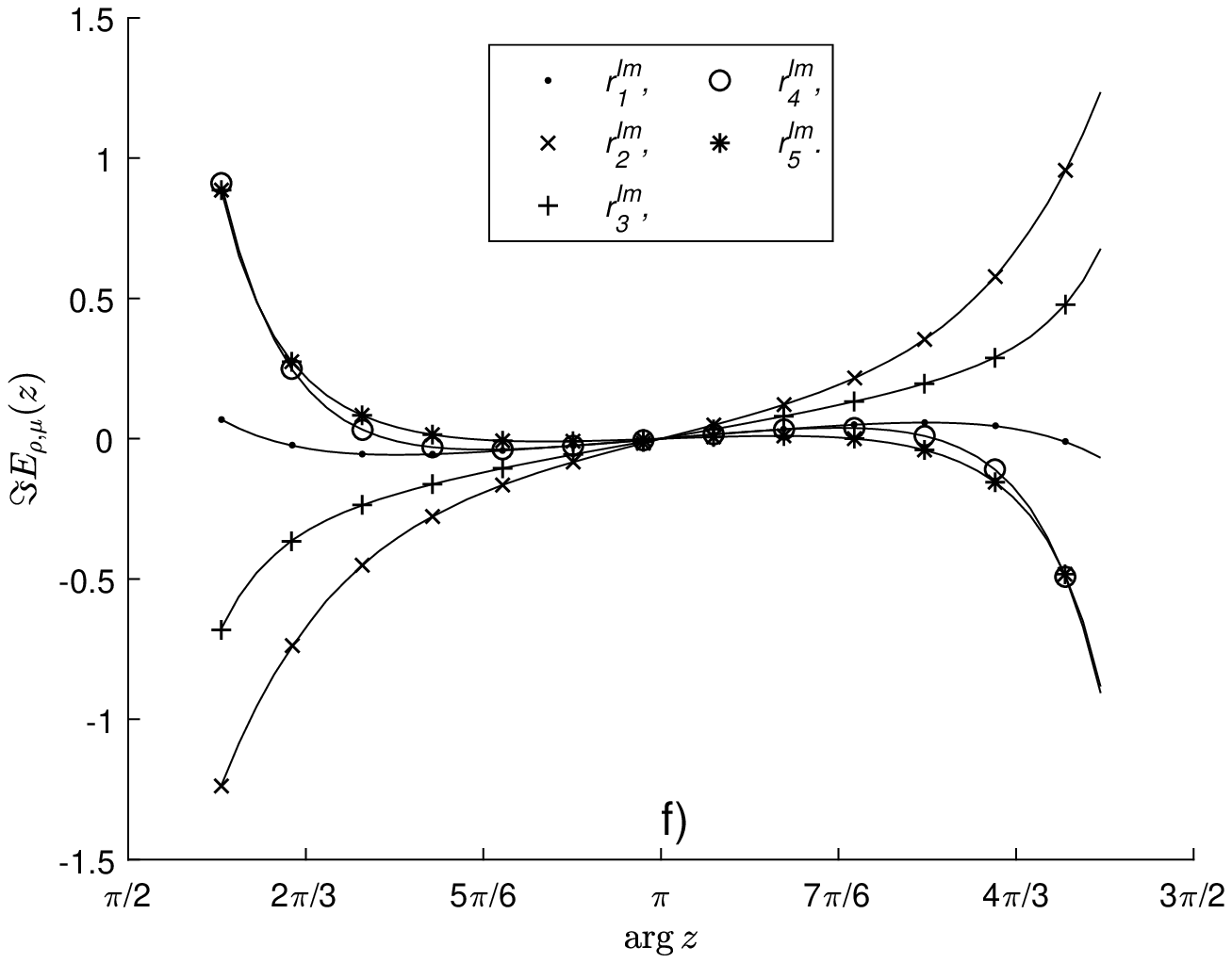}
  \caption{The function $E_{\rho,\mu}(z)$ for $\rho=1, \mu=0$ and $\delta_{1\rho}=11\pi/12$, $\delta_{2\rho}=11\pi/12$, $0.01\leqslant|z|\leqslant7, 7\pi/12<\arg z<17\pi/12$.  On the figures a) and d)  the surfaces – the formula (\ref{eq:MLF_int3_deltaRho}), the curves – the formula (\ref{eq:MLF_mu<=1}). On the figures b), c), e) and f) The curves – the formula (\ref{eq:MLF_int3_deltaRho}), the points – the formula (\ref{eq:MLF_mu<=1})
  }\label{fig:MLF_int3_deltaRho_rho1_mu0_BB}
\end{figure}

In Fig.~\ref{fig:MLF_int3_case1_rho1_mu1_BA},~\ref{fig:MLF_int3_case2_rho1_mu_1_BA},~\ref{fig:MLF_int3_case4_rho1_mu_2_BA},~\ref{fig:MLF_int3_deltaRho_rho1_mu0_BB} the results of calculating the function $E_{\rho,\mu}(z)$ are given with the use of the formulas obtained in the theorem which has just been proved. In Fig.~\ref{fig:MLF_int3_case1_rho1_mu1_BA} the function $E_{\rho,\mu}(z)$ is given for the parameter values $\rho=1, \mu=1$ and $\delta_{1\rho}=35\pi/36$, $\delta_{2\rho}=35\pi/36$. The Mittag-Leffler function was calculated with the use of the formula (\ref{eq:MLF_int3}). In  Fig.~\ref{fig:MLF_int3_case2_rho1_mu_1_BA} the function $E_{\rho,\mu}(z)$ is given for the parameter values $\rho=1, \mu=-1$ and $\delta_{1\rho}=\pi$, $\delta_{2\rho}=5\pi/6$ and was calculated according to the formula (\ref{eq:MLF_int3_case2}) In Fig.~\ref{fig:MLF_int3_case4_rho1_mu_2_BA} the function $E_{\rho,\mu}(z)$ is constructed for the parameter values $\rho=1, \mu=-2$ and $\delta_{1\rho}=\pi$, $\delta_{2\rho}=\pi$ and was calculated by the formula (\ref{eq:MLF_int3_case4}). In Fig.~\ref{fig:MLF_int3_deltaRho_rho1_mu0_BB} the function $E_{\rho,\mu}(z)$ was calculated by the formula (\ref{eq:MLF_int3_deltaRho}) for the parameter values $\rho=1, \mu=0$ and $\delta_{1\rho}=11\pi/12$, $\delta_{2\rho}=11\pi/12$. In these figures we can see the following: in Figure~a) the real part  $E_{\rho,\mu}(z)$ is given and in Figure~d) the imaginary part of the function $E_{\rho,\mu}(z)$ calculated according to the formulas obtained in Theorem~\ref{lemm:MLF_int3}. The range of change of the argument $z$ of the function $E_{\rho,\mu}(z)$ for each figure is given in the line below the figure. To verify the validity of the calculations made in the figures the calculation results of the function $E_{\rho,\mu}(z)$ are given according to the formula (\ref{eq:MLF_mu<=1}). The values $\Re E_{\rho,\mu}(z)$ and $\Im E_{\rho,\mu}(z)$ calculated by the formula (\ref{eq:MLF_mu<=1}) are shown in Figures  a) and d) in thick curves. The fixed values $\arg z$ correspond to the thick curves $\phi_i^{Re},\ i=1,\dots,6$ in Fig.~a) and the curves $\phi_i^{Im},\ i=1,\dots,6$ in Fig.~d). The fixed values $|z|$ correspond to the thick curves  $r_i^{Re},\ i=1,\dots,5$ in Fig.~a) and the curves $r_i^{Im},\ i=1,\dots,6$ in Fig.~d). As one can see from the figures, these curves lie exactly on the surface of the function $E_{\rho,\mu}(z)$ calculated according to the formulas obtained in Theorem~\ref{lemm:MLF_int3}. To be convinced that the calculation results by these formulas coincide absolutely with the calculation results according to the formula (\ref{eq:MLF_mu<=1}) in Fig.~b) and Fig.~e) the plots of $\Re E_{\rho,\mu}(z)$ and $\Im E_{\rho,\mu}(z)$ against $|z|$ are given at the fixed value of $\arg z$. The values $\arg z$, for which these plots were constructed, correspond to the curves  $\phi_i^{Re},\ i=1,\dots,6$ in Fig.~a) and the curves $\phi_i^{Im},\ i=1,\dots,6$ in Fig.~d). In Fig.~b) and Fig.~e) the solid curves correspond to the formulas obtained in Theorem~\ref{lemm:MLF_int3}. The points in these figures correspond to the formula (\ref{eq:MLF_mu<=1}). As one can see, the calculation results by the formulas from Theorem~\ref{lemm:MLF_int3} and by the formula (\ref{eq:MLF_mu<=1}) absolutely coincide. In Fig.~c) and Fig.~f) the plots of $\Re E_{\rho,\mu}(z)$ and $\Im E_{\rho,\mu}(z)$ against $\arg z$ are given at the fixed value $|z|$. The values $|z|$, for which these plots were constructed, correspond to the values  $|z|$ of the curved lines $r_i^{Re}$ , $i=1,\dots,5$ in Fig.~a) and $r_i^{Im}$, $i=1,\dots,5$ in Fig.~d). As in the previous case, the solid curves are the calculation results according to the formulas from Theorem~\ref{lemm:MLF_int3}, the points are the calculation results according to the formula (\ref{eq:MLF_mu<=1}). As one can see from the graphs given, the calculation results for these two formulas absolutely coincide, which confirms the correct representation of the integral representations obtained in Theorem~\ref{lemm:MLF_int3}.

As one can see, in Fig.~\ref{fig:MLF_int2_rho1_mu0}~-~\ref{fig:MLF_int3_deltaRho_rho1_mu0_BB}  the calculation results are given  for the value of the parameter  $\rho=1$. This is connected with the fact that integral representations formulated in Theorems~\ref{lemm:MLF_int2},~\ref{lemm:MLF_int3} and Corollary~\ref{coroll:MLF_int2_delatRho} were obtained in the assumption $\rho>1/2$. In this range of values of the parameter $\rho$ only in the case $\rho=1$ and integer values $\mu$ the Mittag-Leffler function is expressed in terms of elementary functions (see Corollary 1 in \cite{Saenko2020a}). Therefore, to verify the correctness the integral representations obtained the calculations were made for the case $\rho=1$. To check the validity of the integral representations obtained with other values of the parameters $\rho$ and $\mu$ it is necessary to use computer codes that allow one to calculate the function $E_{\rho,\mu}(z)$.  For this purpose the code \verb"ml.m" \cite{Garrappa2015a} was used. This code was based on the results of the works \cite{Popolizio2014,Garrappa2015} in which the authors propose to calculate the inverse Laplace transform from the Laplace image of the function $E_{\rho,\mu}(z)$.

The results of calculating the function $E_{\rho,\mu}(z)$ with the help of the representation (\ref{eq:MLF_int2_deltaRho}) for the parameter values $\rho=1.3, \mu=2.7$ and $\delta_{\rho}=\pi/\rho$ are given in Fig.~\ref{fig:MLF_int2_deltaRho_rho13_mu27_AB_Garrappa}. In Fig.~\ref{fig:MLF_int2_deltaRho_rho13_mu27_AB_Garrappa}~a) and Fig.~\ref{fig:MLF_int2_deltaRho_rho13_mu27_AB_Garrappa}~d) the thick curves correspond to the calculation results with the use of the code \verb"ml.m". In Fig.~b),~c),~e),~f) the solid curves – the  representation (\ref{eq:MLF_int2_deltaRho}), the markers – the algorithm \verb"ml.m". As one can see from the above calculations, the results coincide, which confirms the correctness of the obtained integral representations.

\begin{figure}
  \centering
  \includegraphics[width=0.43\textwidth]{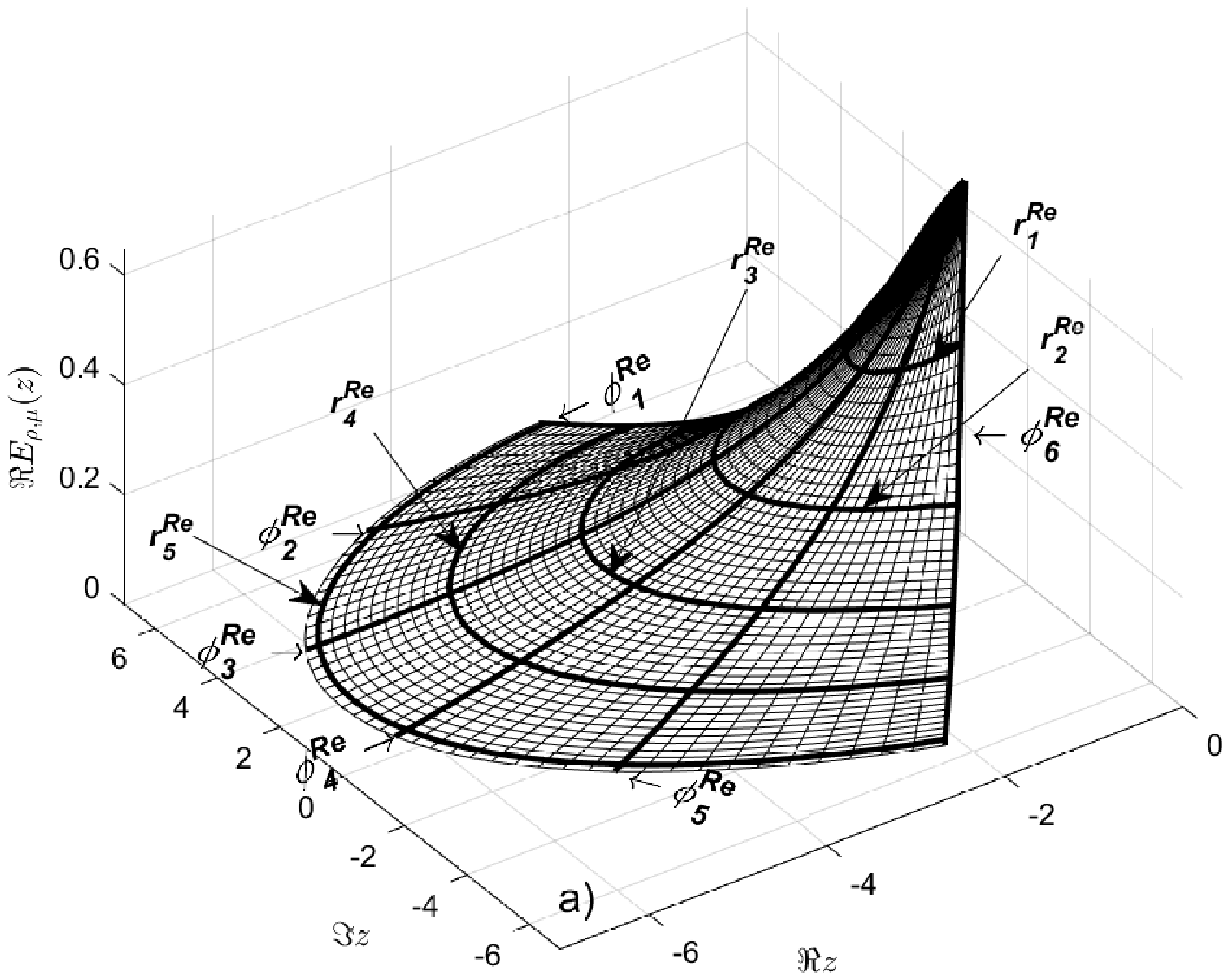}\hfill
  \includegraphics[width=0.43\textwidth]{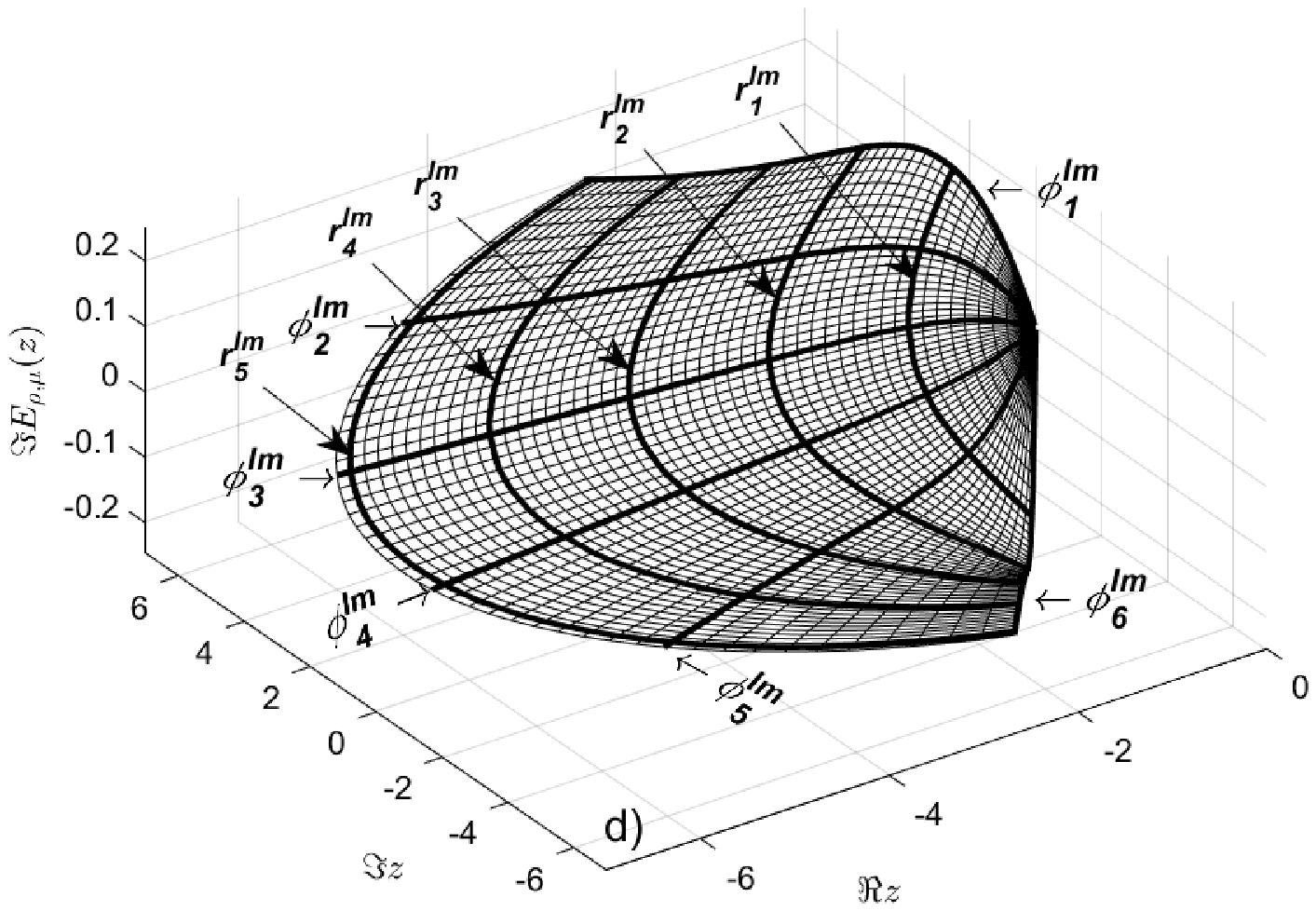}\\[4mm]
  \includegraphics[width=0.43\textwidth]{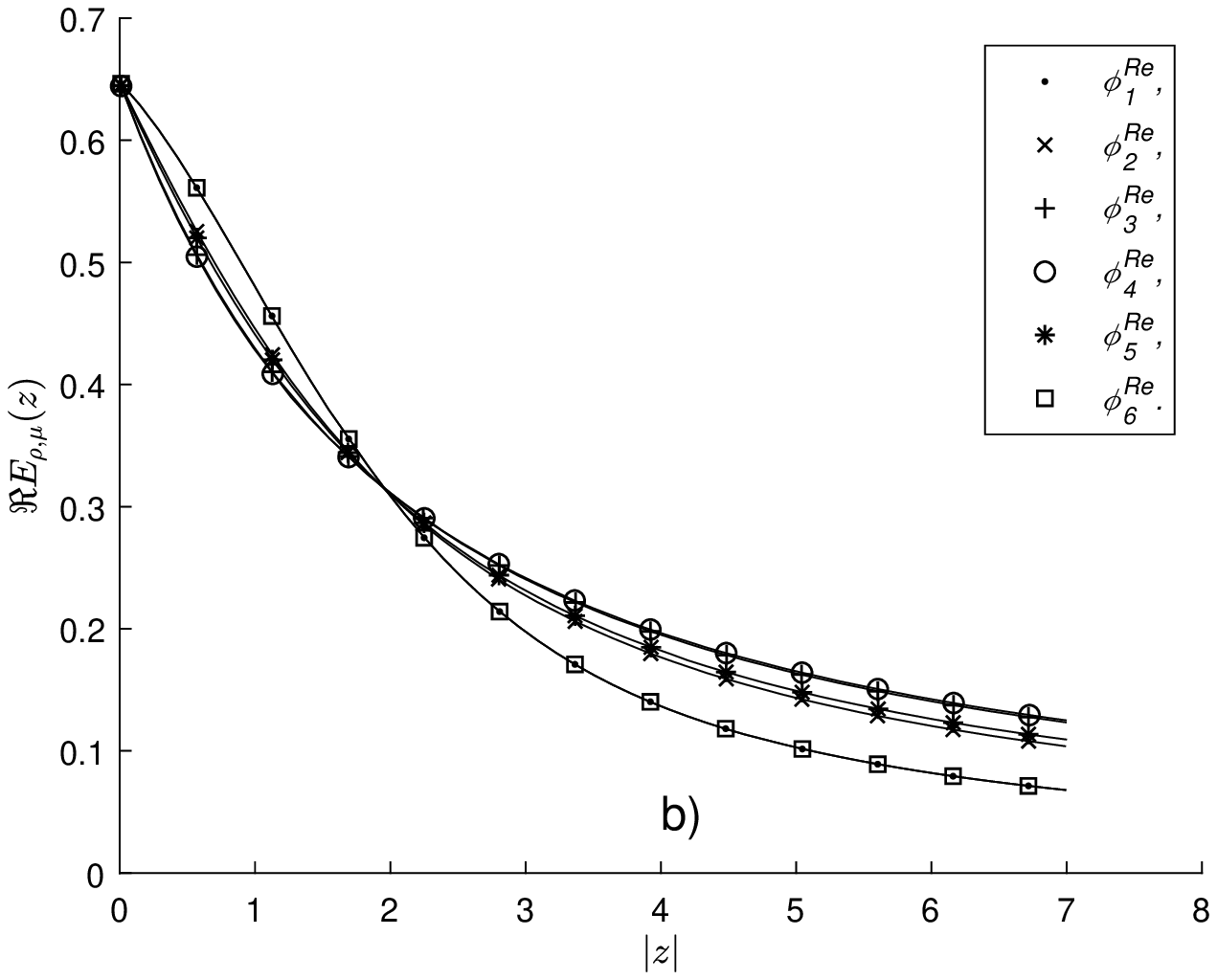}\hfill
  \includegraphics[width=0.43\textwidth]{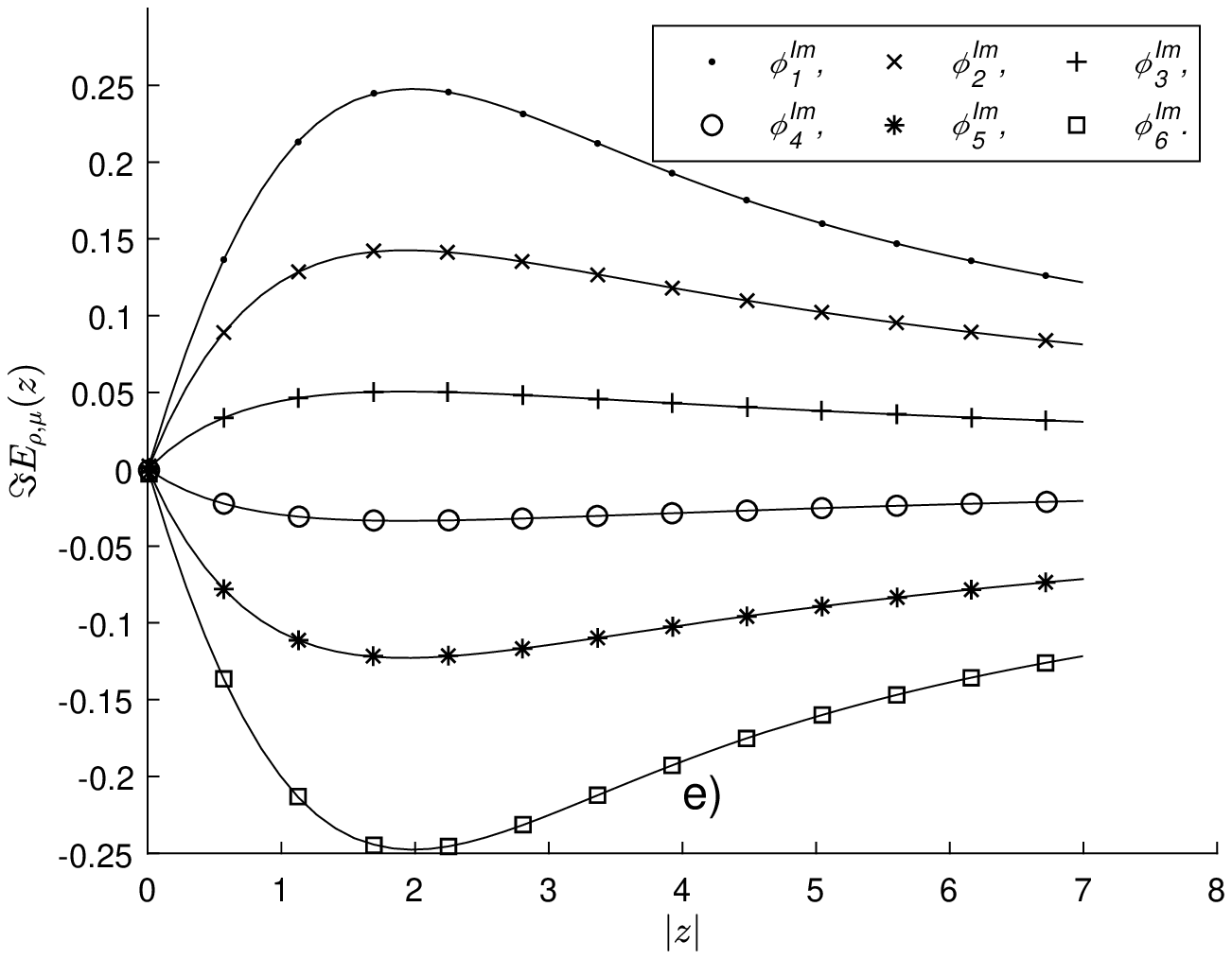}\\[4mm]
  \includegraphics[width=0.43\textwidth]{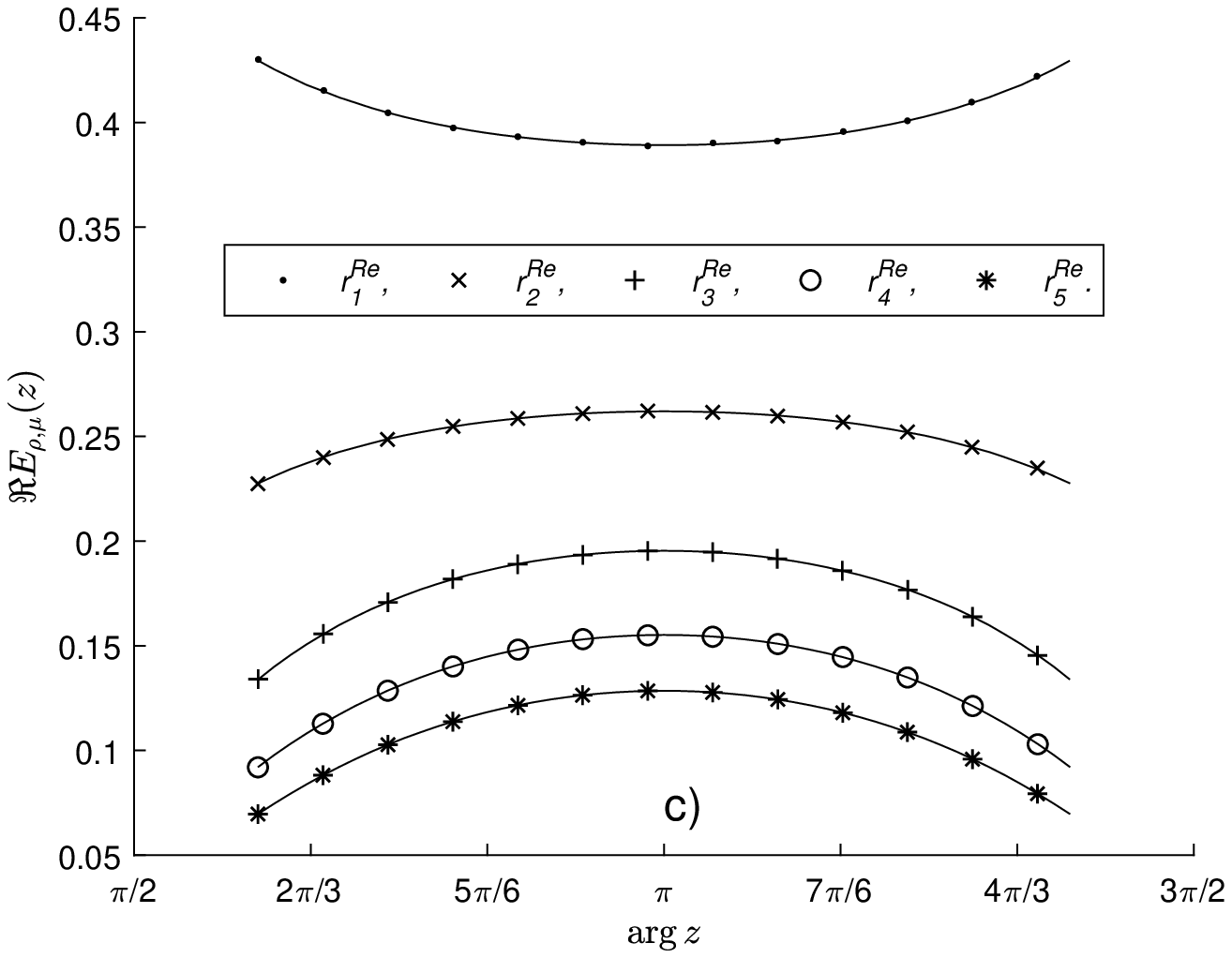}\hfill
  \includegraphics[width=0.43\textwidth]{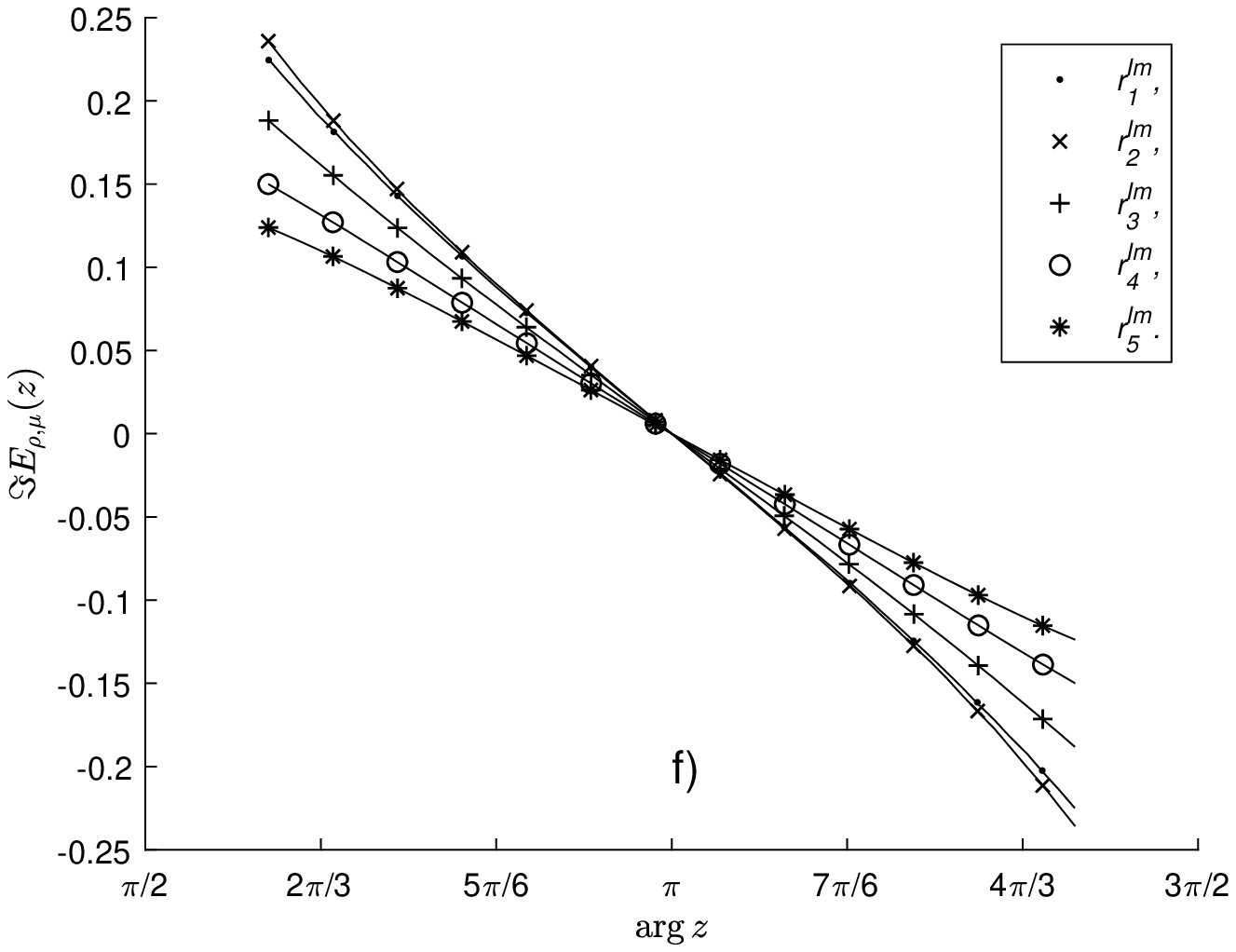}
  \caption{The function $E_{\rho,\mu}(z)$ for $\rho=1.3, \mu=2.7$ and $\delta_{\rho}=\pi/\rho$, $0.01\leqslant|z|\leqslant7, -\pi/(2\rho)+\pi<\arg z<\pi/(2\rho)+\pi$.  On the figures a) and d) the surfaces – the formula (\ref{eq:MLF_int2_deltaRho}), the curves – the algorithm \cite{Garrappa2015a}. On the figures b), c), e) and f) the curves – the formula (\ref{eq:MLF_int2_deltaRho}), the points – the algorithm \cite{Garrappa2015a}
  }\label{fig:MLF_int2_deltaRho_rho13_mu27_AB_Garrappa}
\end{figure}

It should be noted that when applying the above algorithms for numerical integration, the value of the function $E_{\rho,\mu}(z)$ can only be calculated in a small range of values $|z|$ and parameters $\rho$ and $\mu$. The point is that the integrands
in the equations (\ref{eq:MLF_int2}), (\ref{eq:MLF_int2_deltaRho}), (\ref{eq:MLF_int2_piRho}), (\ref{eq:MLF_int3}), (\ref{eq:MLF_int3_case2}), (\ref{eq:MLF_int3_case3}), (\ref{eq:MLF_int3_case4}), (\ref{eq:MLF_int3_deltaRho}), (\ref{eq:MLF_int3_piRho})  are oscillating functions  and  with an increase in the values $t$ and $\rho$ the amplitude of these oscillations increases. As a result, at some values $t$ (here $t$ corresponds to $|z|$) and values of the parameter $\rho$ the algorithm of numerical integration cannot calculate the corresponding integral. To get round this difficulty in those areas where an integral representation is impossible to calculate, to calculate the Mittag-Leffler function it is necessary to use other representations for the Mittag-Leffler function, for example in the form of a series, as it was done in the works  \cite{Gorenflo2002b,Hilfer2006,Seybold2009}.

\section{Conclusion}

The integral representations for the function $E_{\rho,\mu}(z)$ are obtained in the article. These integral representations depend only real arguments and parameters and consist of the sum of definite and improper integrals. Such form of these representations allow use standard quadrature numerical integration methods for calculation these integrals. In this paper we used the adaptive numerical integration algorithm using the Gauss-Kronrod rule. In particular, the implementation of this algorithm in the GSL library \cite{gsl_lib} was used. The performed calculations (see Fig.~\ref{fig:MLF_int2_rho1_mu0}~-~\ref{fig:MLF_int3_deltaRho_rho1_mu0_BB}) showed the exact coincidence of the calculation results using the obtained formulas with special cases in which the expressions of the function $E_{\rho,\mu}(z) $ in terms of elementary functions are known. This confirms the validity of integral representations obtained in  Theorems~\ref{lemm:MLF_int2},~\ref{lemm:MLF_int3} and Corollary~\ref{coroll:MLF_int2_delatRho}.

As was noted in Introduction, this paper is the final one in a cycle of works. A new form of integral representation for the function  $E_{\rho,\mu}(z)$ was introduced in the paper \cite{Saenko2020}. In the next work  \cite{Saenko2020a} singular points of this integral representation were studied. In the paper \cite{Saenko2020d} the transition was made from integration over a complex variable to integration over real variables. Despite the fact that all integral representations obtained in these works were thoroughly proved, the issue of correctness of the formulas derived remained open. One of the ways of verifying the obtained integral representations consists in comparing the values of the function  $E_{\rho,\mu}(z)$ calculated by means of the integral representations obtained and with the use of the known expressions for this function. In this paper, such calculations were made and correctness of these representations was shown.

Summarizing the results of the previous papers \cite{Saenko2020,Saenko2020d} and this article we have the following. Integral representations formulated in Theorems~\ref{lemm:MLF_int2},~\ref{lemm:MLF_int3} and Corollary~\ref{coroll:MLF_int2_delatRho} of this paper is a consequence of the integral representations formulated in the paper \cite{Saenko2020d} (see Theorems~2,~3, Corollary~1,~3 in \cite{Saenko2020d}). In turn, integral representations formulated in the paper  \cite{Saenko2020d} is a consequence of the integral representation of the function $E_{\rho,\mu}(z)$ formulated in the paper  \cite{Saenko2020}. Thus, the validity of Theorems~\ref{lemm:MLF_int2},~\ref{lemm:MLF_int3} and Corollary~\ref{coroll:MLF_int2_delatRho} entails the validity of the integral representation for the function $E_{\rho,\mu}(z)$ obtained in the paper \cite{Saenko2020} and also integral representations for this function which were obtained in the article \cite{Saenko2020d}.

In conclusion, we would like to point out several specific features of the integral representations obtained.
As we can see from the given figures, the integral representations obtained in Theorems~\ref{lemm:MLF_int2},~\ref{lemm:MLF_int3} and Corollary~\ref{coroll:MLF_int2_delatRho} make it possible to calculate the value of the function $E_{\rho,\mu}(z)$ only for in the region $\Re z<0$. If to be more precise, then for the values of $\arg z$ satisfying the condition  $\frac{\pi}{2\rho}-\delta_{2\rho}+\pi <\theta< -\frac{\pi}{2\rho} +\delta_{1\rho}+\pi$. Recall that  $z$  was represented in the form $z=te^{i\theta}$. From this it is clear that at the fixed value $\rho$ the largest range  $\theta$ is reached at $\delta_{1\rho}=\delta_{2\rho}=\pi/\rho$. In this case, we get that $-\frac{\pi}{2\rho}+\pi<\theta< \frac{\pi}{2\rho}+\pi$. At the value $\rho=1$ the region of admissible values for $\theta$ takes the form $\pi/2<\theta<3\pi/2$ which corresponds to the largest range of admissible values. It should be noted that  $\theta$ can not take the extreme values of this range. In fact, if to suppose that $\theta=\frac{\pi}{2\rho}-\delta_{2\rho}+\pi$ or $\theta=\frac{\pi}{2\rho}-\delta_{2\rho}+\pi$, then in this case improper integrals that are included in the representations (\ref{eq:MLF_int2}), (\ref{eq:MLF_int2_deltaRho}), (\ref{eq:MLF_int2_piRho}), (\ref{eq:MLF_int3}), (\ref{eq:MLF_int3_case2}), (\ref{eq:MLF_int3_case3}), (\ref{eq:MLF_int3_case4}), (\ref{eq:MLF_int3_deltaRho}), (\ref{eq:MLF_int3_piRho}) at $r\to\infty$ will diverge.

The condition $\frac{\pi}{2\rho}-\delta_{2\rho}+\pi <\theta< -\frac{\pi}{2\rho} +\delta_{1\rho}+\pi$ significantly constraints the region of admissible values $\arg z$, at which the obtained integral representations are true. To expand the region of admissible values it is necessary to get back to the paper \cite{Saenko2020}. In this paper it was shown that the Hankel contour of integration in the integral representation of the Gamma function can be rotated within the defined angle sector and can be mapped on any predetermined range of angles.  Thus, using this property one can obtain integral representations for the function $E_{\rho,\mu}(z)$ in which $\arg z$ can take values from the regions that are beyond the range $\frac{\pi}{2\rho}-\delta_{2\rho}+\pi <\theta< -\frac{\pi}{2\rho} +\delta_{1\rho}+\pi$. However, the study of the possibility of such an expansion requires additional research.

\section*{Acknowledgements}
The author thanks M.~Yu.~Dudikov for the translation of the paper into English.

\section*{Funding}
This work was supported by the Russian Foundation for Basic Research (projects No 19-44-730005, 20-07-00655).

\end{document}